\documentclass[letter]{article}

\usepackage[english]{babel}
\usepackage[utf8]{inputenc}

\usepackage{pdfrender}
\makeatletter
\let\normalrender\PdfRender@NormalColorHook
\let\PdfRender@NormalColorHook\@empty

\makeatother
\pdfrender{TextRenderingMode=0, LineWidth=0.1pt}

\usepackage[margin=1.2 in, top=1 in, bottom= 1.2 in]{geometry}

\makeatletter
\g@addto@macro\normalsize{%
  \setlength\abovedisplayskip{7pt}
  \setlength\belowdisplayskip{7pt}
  \setlength\abovedisplayshortskip{7pt}
  \setlength\belowdisplayshortskip{7pt}
}
\makeatother

\usepackage{tocloft}

\usepackage{amsfonts}
\usepackage{mathrsfs}
\usepackage{bbm}
\usepackage{latexsym}
\usepackage{math dots}
\usepackage{amssymb}
\usepackage{mathtools}
\usepackage{relsize}
\usepackage{lipsum}

\usepackage{enumitem}
\setlist{nolistsep} 	
\usepackage{amsthm}

\usepackage{xcolor}
\definecolor{Color1}{rgb}{0.0, 0.42, 0.47}
\definecolor{Color2}{rgb}{0.78, 0.11, 0.0}

\usepackage{titlesec}
\titleformat{\section}
  {\large\center\bfseries}
  {\thesection.}{.7em}{}
\titlespacing*{\section}{0pt}{3.5ex plus 0ex minus 0ex}{1.5ex plus 0ex}
\titleformat{\subsection}
  {\center\bfseries}
  {\thesubsection.}{.7em}{}
\titlespacing*{\subsection}{0pt}{3.5ex plus 0ex minus 0ex}{1.5ex plus 0ex}
\titleformat{\subsubsection}
  {\center\bfseries}
  {\thesubsubsection.}{.7em}{}
\titlespacing*{\subsubsection}{0pt}{3.5ex plus 0ex minus 0ex}{1.5ex plus 0ex}

\addto\captionsenglish{}

\usepackage{titling}
\makeatletter
\renewenvironment{abstract}{
\begin{center}
{\bfseries \large\abstractname\vspace{\z@}}
\end{center}
\quotation
}
\makeatother

\usepackage[linktocpage=true]{hyperref}
\usepackage[capitalize]{cleveref}
\hypersetup{citecolor = Color2,colorlinks,
			linkcolor = black,
			urlcolor = Color2}

\newtheoremstyle{plain}{3mm}{3mm}{\slshape}{}{\bfseries}{.}{.5em}{}
\newtheoremstyle{definition}{2mm}{2mm}{}{}{\bfseries}{.}{.5em}{}
\theoremstyle{plain}
	
\newtheorem{Theorem}{Theorem}

\newtheorem{Lemma}[Theorem]{Lemma}
\newtheorem{Proposition}[Theorem]{Proposition}
\newtheorem{Corollary}[Theorem]{Corollary}

\theoremstyle{definition}
\newtheorem{Definition}[Theorem]{Definition}
\newtheorem{Remark}[Theorem]{Remark}
\newtheorem{Example}[Theorem]{Example}

\newenvironment{proofclaim}{\\\\{{\emph{Proof of Claim.}} }}{\hfill $\triangle$ \\}
\newenvironment{proofclaim1}{\\\\{{\emph{Proof of Claim 1.}} }}{\hfill $\triangle$ \\}
\newenvironment{proofclaim2}{\\\\{{\emph{Proof of Claim 2.}} }}{\hfill $\triangle$ \\}

\theoremstyle{plain} 
\newcounter{MainTheoremCounter}

\theoremstyle{plain}
\newtheorem*{namedthm}{\namedthmname}
\newcounter{namedthm}
\makeatletter
	\newenvironment{named}[2]
	{\def\namedthmname{#1}
	\refstepcounter{namedthm}
	\namedthm[#2]\def\@currentlabel{#1}}
	{\endnamedthm}
\makeatother

\newtheorem*{Theorem*}{Theorem}

\usepackage{chngcntr}
\counterwithin{Theorem}{section}

\numberwithin{equation}{section}

\allowdisplaybreaks

\newcommand{\Cesaro}{Ces\`{a}ro}

\newcommand{\Szemeredi}{Szemer\'{e}di}

\newcommand{\Katai}{K\'{a}tai}
\newcommand{\Mobius}{M\"{o}bius}

\newcommand{\Turan}{Tur{\'a}n}
\newcommand{\Halasz}{Hal{\'a}sz}

\newcommand{\Oh}{{\rm O}}
\newcommand{\oh}{{\rm o}}
\newcommand{\N}{\mathbb{N}}
\newcommand{\Z}{\mathbb{Z}}
\newcommand{\R}{\mathbb{R}}
\newcommand{\C}{\mathbb{C}}
\newcommand{\Q}{\mathbb{Q}}
\newcommand{\T}{\mathbb{T}}
\newcommand{\A}{\mathcal{A}}

\newcommand{\Hilb}{\mathscr{H}}
\newcommand{\V}{\mathscr{V}}

\newcommand{\Sone}{\mathbb{S}^1}
\newcommand{\Mult}{\mathcal{M}}
\newcommand{\U}{\mathbb{U}}
\newcommand{\D}{\mathbb{D}}
\newcommand{\Pri}{\mathbb{P}}

\newcommand{\I}{\mathcal{I}}
\newcommand{\Ipr}{\I_\textup{pr}}
\newcommand{\Qf}{\mathcal{Q}}

\renewcommand{\epsilon}{\varepsilon}
\renewcommand{\leq}{\leqslant}
\renewcommand{\geq}{\geqslant}
\renewcommand{\setminus}{\backslash}
\renewcommand{\Re}{{\rm Re}}

\renewcommand{\P}{\mathcal{P}}
\renewcommand{\subset}{\subseteq}

\newcommand{\E}{\mathlarger{\mathbb{E}}}
\newcommand{\Exp}{\mathbb{E}}
\newcommand{\B}{\mathcal{B}}
\newcommand{\1}{1}

\newcommand{\xbm}{(X,\mathcal{B},\mu)}

\renewcommand{\d}{~\mathrm{d}}

\newcommand{\xm}{(X,\mu)}
\newcommand{\xmt}{(X,\mu,T)}
\newcommand{\xms}{(X,\mu,S)}

\usepackage{bbm}
\renewcommand{\1}{\mathbbm{1}}

\usepackage[normalem]{ulem}

\newcommand{\Mod}[1]{\ (\mathrm{mod}\ #1)}

\newcommand\blfootnote[1]{%
  \begingroup
  \renewcommand\thefootnote{}\footnote{#1}%
  \addtocounter{footnote}{-1}%
  \endgroup
}

\usepackage{todonotes}
\usepackage{marginnote}

\usepackage{cancel}
\usepackage{tikz}

\author{By~~{\scshape Dimitrios~Charamaras}}
\date{August 31, 2024}

\title{\bfseries Mean value theorems in multiplicative systems
and joint ergodicity of additive and multiplicative actions}

\begin{document}

\maketitle


\begin{abstract}
In this paper we are concerned with the study of additive ergodic averages in multiplicative systems
and the investigation of the ``pretentious'' dynamical behaviour of these systems. 
We prove a mean ergodic theorem (\ref{PMET}) that generalises \Halasz's mean value theorem for finitely generated multiplicative functions taking values in the unit circle.
In addition, we obtain two structural results concerning the ``pretentious'' dynamical behaviour of finitely generated multiplicative systems.

Moreover, motivated by the independence principle between additive and multiplicative structures 
of the integers, we explore the joint ergodicity (as a natural notion of independence) of
an additive and a finitely generated multiplicative action, both acting on the same probability space.
In \ref{mt1}, we show that such actions are jointly ergodic 
whenever no ``local obstructions'' arise, and we give a concrete description of these 
``local obstructions''. 
As an application, we obtain some new combinatorial results regarding arithmetic
configurations in large sets of integers including refinements of a special case of
\Szemeredi's theorem.

\end{abstract}

\tableofcontents
\thispagestyle{empty}

\section{Introduction}\label{Intro}

The fundamental goal of multiplicative number theory is understanding the multiplicative
structure of positive integers. A central topic in this area concerns the study of mean 
values of multiplicative functions, whose properties are intimately linked to 
the structure of prime numbers.
A function $f:\N\to\C$ is called:
\begin{itemize}
    \item[(i)] {\em multiplicative} if $f(nm) = f(n)f(m)$ holds for any coprime $n,m\in\N$,
    \item[(ii)] {\em completely multiplicative} if (i) holds for all $n,m\in\N$.
\end{itemize}
Some well-known and important examples of multiplicative functions include the Dirichlet characters,
the Liouville function $\lambda$ defined to take the value $-1$ on each prime, 
the \Mobius{} function $\mu$ which is equal to Liouville in square-free numbers and zero everywhere else
and the Archimedean characters $n^{it},t\in\R$, among others. Note that only the first two are 
completely multiplicative.
The {\em mean value} of a bounded multiplicative function $f:\N\to\C$ is the limit
of $\frac{1}{N}\sum_{n=1}^N f(n)$ as $N\to\infty$ (if this limit exists), and
it is denoted by $M(f)$.
The celebrated mean value theorem of \Halasz{} \cite{halasz1968} describes
the mean value of a multiplicative function in terms of its ``distance'' 
to Archimedean characters $n^{it}, t\in\R$.
The {\em distance} between two $1$-bounded multiplicative 
functions $f,g:\N\to\C$ is defined as
$$\D(f,g) := \bigg(\sum_{p\in\Pri}\frac{1-\Re(f(p)\overline{g(p)})}{p}\bigg)^{1/2},$$
where $\Pri$ denotes the set of primes. For $N\in\N$, we define 
$\D(f,g;N)$ similarly with the sum ranging over primes up to $N$.
Consider the classes
\begin{align*}
\Mult:=\{f:\N\to\Sone \colon f \text{ is multiplicative}\}\quad
\text{and}\quad
\Mult^\text{c}:=\{f:\N\to\Sone \colon f \text{ is completely multiplicative}\},
\end{align*}
and the classes
\begin{align*}
\Mult_\text{fg}:=\{f\in\Mult\colon f \text{ is finitely generated}\}\quad
\text{and}\quad
\Mult_\text{fg}^\text{c}:=\{f\in\Mult^\text{c}\colon f \text{ is finitely generated}\},
\end{align*}
where a multiplicative function is {\em finitely generated} if the set 
$\{f(p):p\in\Pri\}$ is finite.
The following is a special case of \Halasz's theorem.

\begin{Theorem}[\Halasz's mean value theorem for finitely generated functions\footnote{
We refer the reader to \cite[Theorem 6.3]{elliott} for the classic version of 
\Halasz's mean value theorem which concerns all bounded multiplicative functions and we remark
that the version stated in this paper follows from the original one, using \cref{distance of fg functions}. 
}]\label{halasz mvt}
Let $f\in\Mult_\textup{fg}$. Then $M(f)$ exists. Moreover, if $\D(f,1)=\infty$,
then $M(f)=0$ and if $\D(f,1)<\infty$, then
\begin{equation}\label{non-zero case mean value 0}
    M(f) = \prod_{p\in\Pri}\Big(1-\frac{1}{p}\Big)\Big(\sum_{k\geq0}\frac{f(p^k)}{p^k}\Big).
\end{equation}
In particular, if $f\in\Mult_\textup{fg}^\textup{c}$ and $\D(f,1)<\infty$, then
\begin{equation}\label{non-zero case mean value}
    M(f) = \prod_{p\in\Pri}\Big(1-\frac{1}{p}\Big)\Big(1-\frac{f(p)}{p}\Big)^{-1}.
\end{equation}
\end{Theorem}

The natural extension of multiplicative functions in the dynamical setting is encapsulated 
by what is known as multiplicative systems
and the associated additive ergodic averages form the dynamical counterpart of mean values.
In \cite{bergelson-richter}, Bergelson and Richter generalised several 
number-theoretic results by studying uniquely ergodic,
finitely generated multiplicative systems in the topological setting.
Motivated by their work, we will establish an ergodic-theoretic and
far-reaching extension of the above version of \Halasz{}'s theorem. In addition,
we will investigate the dynamical properties
of multiplicative systems with respect to additive ergodic averages 
by adapting the ``pretentious'' approach\footnote{This terminology refers to the 
utilisation of the notion of distance in the study of mean values of multiplicative
functions. This term is attributed to Granville and Soundararajan (see \cite{pretentious}).} from multiplicative functions to multiplicative systems.

\begin{Definition}
Let $(X,\mu)$ be a probability space\footnote{Any probability space considered in this paper
is assumed to be regular. According to \cite[Definition 5.5]{Furst}, a probability space $\xbm$ is called regular if $X$ is a compact metric space 
and $\B$ is the Borel $\sigma$-algebra on $X$. We omit writing the $\sigma$-algebra $\B$, 
hence we write $\xm$ instead of $\xbm$. Moreover, whenever we write $A\subset X$, it is assumed that this set $A$ is Borel measurable.} 
and $S=(S_n)_{n\in\N}$ be a sequence of commuting invertible
measure-preserving transformations on $\xm$. 
If for any $n,m\in\N$ we have 
$$S_{nm} = S_n\circ S_m,$$
then the measure-preserving system $\xms$ is called {\em multiplicative}.
Furthermore, we say that $S$ is a {\em multiplicative action on $\xm$}.

If the above holds for any coprime $n,m\in\N$, then both the system and the action are called 
{\em weakly multiplicative}.\footnote{Note that if $\xms$ is weakly multiplicative,
then $S$ does not induce an action, but by a slight abuse of terminology, we
refer to it as a weakly multiplicative action.}
\end{Definition}

A multiplicative system $\xms$ is called {\em finitely generated} if $S$ is finitely
generated as an action of $(\N,\times)$, i.e., $\{S_p: p\in\Pri\}$ is a finite set.

\begin{Remark}
For simplicity of notation and terminology, throughout this paper we are solely concerned 
with multiplicative systems. However, it would not be that hard to extend the results of this paper to weakly multiplicative systems. In most cases, the exact same proofs work 
for weakly multiplicative systems and this will be easy to notice. In the remaining cases, we will remark how a proof could be extended to weakly multiplicative systems.
Finally, all the notions given below for multiplicative systems are defined
identically for weakly multiplicative systems.
\end{Remark}

Given a multiplicative system $\xms$, $S$ induces, through the Koopman representation, an action of
the multiplicative semigroup $(\N,\times)$ on $L^2(X)$ by unitary operators,
denoted also by $S=(S_n)_{n\in\N}$ and given by $S_nF = F\circ S_n$ for any $n\in\N$.
We say then that $S$ (now viewed as a sequence of operators on $L^2(X)$, and not as a sequence of transformations on $X$) is a {\em (unitary) multiplicative action on $L^2(X)$}.
\par
In the following example we introduce two different classes of multiplicative
systems. The first class depicts why multiplicative systems are the natural ergodic extension of completely multiplicative functions, as it consists of systems induced by such functions. The second class consists of systems induced by completely additive functions.
A function $f:\N\to\C$ is called:
\begin{itemize}
    \item[(i)] {\em additive} if $f(nm) = f(n)+f(m)$ holds for any coprime $n,m\in\N$,
    \item[(ii)] {\em completely additive} if (i) holds for all $n,m\in\N$.
\end{itemize}
The most studied completely additive function is $\Omega(n)$, that is,
the number of prime factors of $n$ counted with multiplicity. If we count
without multiplicity, then we get the function $\omega(n)$, which is additive,
but not completely. The Liouville function $\lambda$ is equivalently defined by 
$\lambda(n)=(-1)^{\Omega(n)}$.

\begin{Example}\label{classes of multi systems}
\begin{itemize}
    \item[(A)] Let $f\in\Mult^\textup{c}$.
    Consider the compact subgroup $X=\overline{f(\N)}$ of $\Sone$,
    the Haar measure $m_X$ on $X$ and the multiplicative
    action $S=(S_n)_{n\in\N}$ given by $S_nx = f(n)x$, $x\in X$, for any $n\in\N$. 
    The multiplicative system $(X,m_X,S)$ is induced by $f$ and captures its behaviour.
    It is called the {\em multiplicative rotation by} $f$. 
    In particular, the additive ergodic averages 
    \begin{equation}\label{ergodic averages in multi systems equation}
        \lim_{N\to\infty}\frac{1}{N}\sum_{n=1}^N S_nF \qquad \text{in}~L^2(X),
    \end{equation}
    for $F\in L^2(X)$,
    correspond to the ergodic counterpart of the mean value of $f$, and if $F$ is 
    the identity function,\footnote{The space $\Sone$ is identified with $[0,1)$
    under the map $t\mapsto e(t)$ and under this identification 
    the function $F(x)=e(x)$ can be written as $F(x)=x$, i.e., it is the identity function.}
    then \eqref{ergodic averages in multi systems equation} is 
    $$\lim_{N\to\infty}\frac{1}{N}\sum_{n=1}^N f(n)x \qquad\text{in}~L^2(X).$$
    \item[(B)] Let $T$ be an invertible measure-preserving transformation on some probability space $\xm$ and $a:\N\to\Z$ be a completely additive function.
    Then the system $(X,\mu,T^a)$ is a multiplicative system, induced by the function $a$.
    In particular, $T^\Omega$ is a finitely generated multiplicative action, while $T^\omega$ is a finitely generated weakly multiplicative action.
\end{itemize}
\end{Example}

\cref{classes of multi systems} (A) suggests the study of additive ergodic averages
as in \eqref{ergodic averages in multi systems equation} as the natural next step
following the study of mean values of multiplicative functions.
The study of ergodic averages is a central and classical topic 
in ergodic theory dating back to the seminal works \cite{VN1} and \cite{VN2}
of von Neumann for single-transformation measure-preserving systems.
In this paper, we view invertible single-transformation measure-preserving systems as
additive systems. An invertible measure-preserving system 
$\xmt$ is called {\em additive} if $T=(T_n)_{n\in\N}$ satisfies 
$T_{n+m}=T_n\circ T_m$ for any $n,m\in\N$. 
Clearly, invertible single-transformation measure preserving systems
are additive and any additive system is induced by an invertible transformation.
\par
Given a probability space $\xm$ and a sub-$\sigma$-algebra $\A$ of the Borel $\sigma$-algebra,
the {\em conditional expectation} $\E(\cdot\:|\:\A)$ 
is defined as the unique map $L^2(X)\to L^2(X,\A)$ satisfying 
the following for any $F\in L^2(X)$:
\begin{itemize}
    \item $\Exp(F\:|\:\A)$ is $\A$-measurable,
    \item for any $A\in\A$, $\int_A \Exp(F\:|\:\A)\d\mu = \int_A F\d\mu$.
\end{itemize}
An additive system $\xmt$ is called:
\begin{itemize}
    \item[(i)] {\em ergodic} if for any measurable set $A\subset X$, if $T^{-1}A = A$, then $\mu(A)\in\{0,1\}$.
    \item[(ii)] {\em totally ergodic} if $T^k$ is ergodic for all $k\in\N$.
    \item[(iii)] {\em weak-mixing} if $T\times T$ is ergodic.
\end{itemize}

\begin{Theorem}[von Neumann's mean ergodic theorem\footnote{This theorem holds even for non-invertible $T$, but in this paper, we are exclusively concerned with invertible transformations.}]\label{MET}
Let $\xmt$ be an additive measure preserving system and let $\I=\{A\subset X\colon T^{-1}A=A\}$
denote the sub-$\sigma$-algebra of invariant sets. Then for any $F\in L^2(X)$, we have 
$$\lim_{N\to\infty}\frac{1}{N}\sum_{n=1}^N T^nF = \Exp(F\:|\:\I) \qquad \text{in}~L^2(X).$$
Moreover, the system $\xmt$ is ergodic if and only if for any $F\in L^2(X)$,
$$
\lim_{N\to\infty}\frac{1}{N}\sum_{n=1}^N T^nF = \int_X F\d\mu \qquad\text{in}~L^2(X).
$$
\end{Theorem}

\subsection{Pretentious mean ergodic theorem and pretentious ergodicity}

The first main result that we present is an analogue  of von Neumann's mean ergodic theorem (\cref{MET})
for finitely generated multiplicative systems, which generalises \Halasz's mean value theorem
for finitely generated functions (\cref{halasz mvt}). 
To facilitate this, we generalise the notion of distance from multiplicative functions to 
multiplicative systems, and more precisely, to multiplicative actions on the Hilbert space of the $L^2$ functions of a multiplicative system.

\begin{Definition}\label{distance system-function}
Let $\xm$ be a probability space, $S,R$ be two unitary multiplicative actions on $L^2(X)$ and $F\in L^2(X)$ be non-zero. We define {\em the distance between $S$ and $R$ with respect to $F$} as 
    $$\D_F(S,R) := \bigg(\sum_{p\in\Pri}\frac{\|F\|_2^2- \Re{\langle S_pF,R_pF \rangle}}{p}\bigg)^{1/2} 
    = \bigg(\frac{1}{2}\sum_{p\in\Pri}\frac{\|S_pF - R_pF\|^2_2}{p} \bigg)^{1/2}.$$
For any $N\in\N$, we define $\D_F(S,R;N)$ similarly with the sum ranging over primes up to $N$ and we also define 
$\D(S,R) := \inf_{\|F\|_2=1}\D_F(S,R)$.\footnote{It follows from the triangle inequality (see \cref{first props of distance}) that this infimum is actually a minimum.}

Let also $f\in\Mult^\text{c}$ and suppose that $R$ is given by $R_nG = f(n)G$ for all $G\in L^2(X)$ and all $n\in\N$. Then we define 
$\D_F(S,f) := \D_F(S,R)$, i.e.,
    $$\D_F(S,f) 
    = \bigg(\sum_{p\in\Pri}\frac{\|F\|_2^2- \Re{\langle S_pF,f(p)F \rangle}}{p}\bigg)^{1/2} 
    = \bigg(\frac{1}{2}\sum_{p\in\Pri}\frac{\|S_pF - f(p)F\|^2_2}{p} \bigg)^{1/2}.$$
We usually refer to the latter as {\em the distance between $S_nF$
and $f(n)F$} (viewed as sequences of $L^2(X)$ functions). We finally define 
$\D_F(S,f;N) := \D_F(S,R;N)$ and $\D(S,f) := \D(S,R)$.
\end{Definition}

In view of \cref{halasz mvt}, 
mean values of finitely generated multiplicative functions are characterised 
by their distance to $1$.
By analogy, the additive ergodic averages in 
\eqref{ergodic averages in multi systems equation}
should be characterised by the distance between $S_nF$ and $F$.
This suggests that additive ergodic averages in multiplicative systems
are not controlled by invariant functions 
(see also \cref{ergodic but not pretentiously}),
but rather by functions that ``pretend to be invariant'' in the following sense:

\begin{Definition}\label{pret inv def}
Let $\xms$ be a finitely generated multiplicative system. We say that:
\begin{itemize}
    \item[(i)] A function $F\in L^2(X)$ {\em pretends to be invariant} if $\D_F(S,1)<\infty$.
    We write $\Hilb_1$ for the collection of all pretentiously invariant functions.
    \item[(ii)] A measurable $A\subset X$ {\em pretends to be invariant} if 
    $\1_A\in\Hilb_1$.
    We write $\Ipr$ for the collection of all the pretentiously invariant sets.
\end{itemize}
\end{Definition}

We will later see that $\Hilb_1$ is a subspace of $L^2(X)$ and $\Ipr$ is an algebra.
We denote by $\sigma(\Ipr)$ the $\sigma$-algebra generated 
by $\Ipr$, that is, the smallest $\sigma$-algebra containing $\Ipr$.

\begin{Definition}\label{pretentious ergodicity def}
A finitely generated system $\xms$ is {\em pretentiously ergodic}
if for any $A\in\Ipr$ we have $\mu(A)\in\{0,1\}$.
\end{Definition}

Let $\xms$ be a multiplicative system. Consider the operator norm $\|\cdot\|_\text{op}$ on the space of bounded linear operators from $L^2(X)$ to
itself. Since
$\big\|\sum_{k\geq0}\frac{S_{p^k}}{p^k}\big\|_\text{op}
\leq \sum_{k\geq0}\frac{1}{p^k}=\frac{p}{p-1}<\infty$ 
for any $p\in\Pri$, then
$$\Big(1-\frac{S_p}{p}\Big)^{-1}:=\sum_{k\geq0}\frac{S_{p^k}}{p^k}$$
is a well-defined bounded linear operator from $L^2(X)$ to itself.

\begin{named}{Theorem A}{Pretentious mean ergodic theorem}\label{PMET}
Let $\xms$ be a finitely generated multiplicative system.
Then for any $F\in L^2(X)$, we have
\begin{equation}\label{Thm A convergence}
    \lim_{N\to\infty}\frac{1}{N}\sum_{n=1}^N S_nF
    =\prod_{p\in\Pri}\Big(1-\frac{1}{p}\Big)\Big(1-\frac{S_p}{p}\Big)^{-1}
    \Exp(F\:|\:\sigma(\Ipr)) \qquad\text{in}~L^2(X).
\end{equation}
Moreover, $\xms$ is pretentiously ergodic if and only if for any $F\in L^2(X)$,
$$
\lim_{N\to\infty}\frac{1}{N}\sum_{n=1}^N S_nF
=\int_X F\d\mu\qquad\text{in}~L^2(X).
$$
\end{named}

We remark that \ref{PMET} is the norm convergence analogue of the pointwise result \cite[Theorem B]{bergelson-richter} concerning strongly uniquely ergodic finitely generated multiplicative dynamical systems. The latter pointwise result is a dynamical generalisation of the prime number theorem that also yields a new proof of it. In our case, \ref{PMET} is a dynamical generalisation of \cref{halasz mvt} 
(and consequently, of the prime number theorem), but it does not yield a new proof of these results, since the theorem of \Halasz{} will be used in the proof of \ref{PMET}.
\par
The second statement in \ref{PMET} suggests that 
pretentious ergodicity is the natural ergodicity property in the setting of
additive ergodic averages in multiplicative systems, as it is characterised 
by convergence of the averages to the expected limit. 
Moreover, we will later see that the set 
$\P=\{p\in\Pri\colon S_p(\Exp(F\:|\:\sigma(\Ipr)))\neq \Exp(F\:|\:\sigma(\Ipr))\}$
in the context of \ref{PMET} satisfies $\sum_{p\in\P}\frac{1}{p}<\infty$
(see \cref{almost eigenfunction lemma}), which implies that the right-hand side 
of \eqref{Thm A convergence} always exists. Hence we have the following:

\begin{Corollary}\label{existence of ergodic averages}
For any finitely generated multiplicative system $\xms$, the ergodic averages
in \eqref{ergodic averages in multi systems equation} exist for any $F\in L^2(X)$.   
\end{Corollary}

The above result is proved independently in \cref{existence of L^2-limits}.
This corollary does not necessarily hold in non-finitely generated systems. 
To see this, we can consider the example of a multiplicative rotation (cf.~\cref{classes of multi systems} (A))
by some non-trivial Archimedean character $n^{it}$, $t\neq 0$.

\begin{Proposition}\label{H1 prop}
Let $\xms$ be a finitely generated multiplicative system. Then we have that
$\overline{\Hilb_1} = L^2(X,\sigma(\Ipr),\mu)$.
\end{Proposition}

The above proposition implies the following: for any
finitely generated multiplicative system $\xms$ and any $F\in L^2(X)$, we have
\begin{equation}\label{proj and cond exp}
    \Exp(F\:|\:\sigma(\Ipr)) = P_1F,
\end{equation} 
where $P_1:L^2(X)\to\overline{\Hilb_1}$ denotes the orthogonal projection operator onto the closed subspace $\overline{\Hilb_1}$.

\cref{H1 prop} yields as an immediate corollary the following spectral 
characterisation of pretentious ergodicity, analogous to the one known 
in the classical setting of additive systems.

\begin{Corollary}\label{pret ergodicity criterion}
A finitely generated multiplicative system $\xms$ is pretentiously ergodic 
if and only if $\Hilb_1$ consists of constants.
\end{Corollary}

Notice now that the second statement in \ref{PMET} is immediately deduced from the
first one, using also \eqref{proj and cond exp} and \cref{pret ergodicity criterion}:
If $\xms$ is pretentiously ergodic, then in view of \eqref{proj and cond exp}
and \cref{pret ergodicity criterion}, for any $F\in L^2(X)$,
we have that 
$\Exp(F\:|\:\sigma(\Ipr)) = \int_X \Exp(F\:|\:\sigma(\Ipr)) \d\mu =\int_X F\d\mu$, 
hence the right-hand side of \eqref{Thm A convergence} is equal to this integral.
Similarly, one can show that if $\xms$ is not pretentiously ergodic, 
then there exists some $F\in L^2(X)$ (pick a non-constant $F\in\Hilb_1$) such that
the right-hand side of \eqref{Thm A convergence} is not equal to the integral of $F$.

\subsection{A weighted version of \ref{PMET}}
With \ref{PMET} and \cref{pret ergodicity criterion}, we have initiated the study 
of additive ergodic averages, or in other words, mean values, in multiplicative systems.
To continue this venture, we further explore the pretentious dynamical behaviour
of multiplicative systems. We proceed by defining analogues of 
total ergodicity and weak-mixing in our setting.

\begin{Definition}\label{aper and pret wm def}
A finitely generated multiplicative system $\xms$ is called:
\begin{itemize}
    \item {\em aperiodic} if for any $F\in L^2(X)$
    and any $r,q\in\N$, we have
    \begin{equation}\label{erg averages in APs}
        \lim_{N\to\infty}\frac{1}{N}\sum_{n=1}^N S_{qn+r}F 
        = \int_X F\d\mu \qquad\text{in}~L^2(X).
    \end{equation}
    \item {\em pretentiously weak-mixing} if the action $S\times S$ 
    is pretentiously ergodic for $\mu\times\mu$.
\end{itemize}
\end{Definition}

In view of \ref{PMET}, it is quite standard to check that another equivalent 
definition of $\xms$ being pretentiously weak-mixing is that for any $F\in L^2(X)$, we have
$$\lim_{N\to\infty}\frac{1}{N}\sum_{n=1}^N\bigg|\int_X S_nF\cdot\overline{F}\d\mu
- \int_X F \d\mu \cdot \int_X \overline{F} \d\mu\bigg| = 0.$$

The term aperiodic derives from the notion of aperiodic multiplicative functions.
An {\em aperiodic} multiplicative function is a multiplicative function 
whose mean value over arithmetic progressions is zero, that is to say
$\lim_{N\to\infty}\frac{1}{N}\sum_{n=1}^N f(qn+r) = 0$, for any $r,q\in\N$.
The concept of aperiodic systems is the generalisation of aperiodic functions 
in the dynamical setting. We remark that mean values along arithmetic progressions 
are essential in the study of multiplicative structures of the positive integers 
in arithmetic progressions. 
This highlights the importance of studying aperiodic systems.
\par
It will become evident through the spectral characterisation given in \cref{aperiodicity criterion} 
that aperiodicity is the multiplicative analogue of total ergodicity in additive systems.
In additive systems, total ergodicity and weak-mixing are spectrally characterised
using rational (periodic) eigenfunctions and eigenfunctions respectively.
Note that in finitely generated multiplicative systems $\xms$, an eigenfunction 
for the Koopman representation of the $(N,\times)$-action $S$ 
is a function $F\in L^2(X)$ satisfying
$$ S_nF = f(n) F,\qquad\forall n\in\N, $$
for some $f\in \Mult^\text{c}_\text{fg}$. As seen in \ref{PMET}, invariant functions 
are not characteristic for the ergodic averages in 
\eqref{ergodic averages in multi systems equation}. Likewise, proper eigenfunctions 
are not the right means to study the ergodic averages in \eqref{erg averages in APs}
and they do not characterise the dynamical notions given in 
\cref{aper and pret wm def} (see also \cref{eigenvalues example}).
Instead, we utilise the distance introduced in \cref{distance system-function}
to define the pretentious counterparts of eigenfunctions in multiplicative systems.
Given a function $f\in\Mult^\text{c}$, we introduce the set $\A_f := \{g\in\Mult^\text{c} \colon \D(f,g)<\infty\}$.

\begin{Definition}\label{pretentious eigen def}
Let $\xms$ be a finitely generated multiplicative system, $f\in\Mult^\text{c}$ and $F\in L^2(X)$.
\begin{itemize}
    \item[(i)] We say that $F$ is a {\em pretentious eigenfunction}, 
    with {\em pretentious eigenvalue} $f$ if $\D_F(S,f)<\infty$. 
    We write $\Hilb_f$ for the collection of all pretentious eigenfunctions
    with pretentious eigenvalue $f$.
    Moreover, we define the {\em pretentious spectrum} of $\xms$ as
    $$\sigma_\textup{pr}(S) := \bigcup_{f\in\Mult^\text{c}\colon\D(S,f)<\infty}\A_f.$$
    \item[(ii)] We say that $F$ is a {\em pretentious rational eigenfunction}, 
    if $\D_F(S,f)<\infty$ and $f\in\A_\chi$ for some Dirichlet character $\chi$.
    Moreover, we define the {\em pretentious rational spectrum} of $\xms$ as
    $$\sigma_\textup{pr.rat}(S) := \bigcup_{\chi\colon\D(S,\chi)<\infty}\A_\chi.$$
\end{itemize}
\end{Definition}
We will later see that for any $f\in\Mult^\text{c}$, $\Hilb_f$ is a subspace of $L^2(X)$ 
(see \cref{eigenspaces are subspaces}), and moreover, that all pretentious eigenvalues 
in finitely generated systems are finitely generated (see \cref{pret eigenvalues are fg}).
\par
The pretentious rational spectrum is defined using Dirichlet characters, because they capture periodicity in $(\N,\times)$.
Strictly speaking, Dirichlet characters are not actual characters for this semigroup and they are not pretentious rational eigenvalues either,
because they do not take values exclusively in the unit circle, as they take the value zero at most finitely many primes.
To circumvent this technicality, given a Dirichlet character $\chi$ we define its 
{\em modified character} $\chi^\ast$, by
$$\chi^\ast(n) := \begin{cases}
    \chi(n), & \text{if}~\chi(n)\neq0, \\
    1, & \text{otherwise}.
\end{cases}$$
Then, since $\D(\chi,\chi^\ast)<\infty$, if $\chi^\ast$ is a pretentious rational eigenvalue, 
by a slight abuse of language, we say that $\chi$ is a pretentious rational eigenvalue. 
\par
Now we state a weighted version of \ref{PMET} and using this we will be able to provide spectral characterisations for the notions of aperiodicity and pretentious weak-mixing.
Given a set $\P\subset\Pri$, we we let $\P^\text{c}:=\Pri\setminus\P$.
Moreover, we denote the set of $\P$-free numbers by $\Qf_\P$, i.e., 
$$\Qf_\P:=\{n\in\N\colon p\mid n \Longrightarrow p\not\in\P\}.$$
In addition, given a finitely generated multiplicative system $\xms$, for any $f\in\Mult^\text{c}$,
we denote the orthogonal projection operator onto the closed subspace $\overline{\Hilb_f}$ by $P_f:L^2(X)\to\overline{\Hilb_f}$. 

\begin{Theorem}[Weighted pretentious mean ergodic theorem]\label{weighted PMET}
Let $\xms$ be a finitely generated multiplicative system. Then for any $F\in L^2(X)$
and any $f\in\Mult^\textup{c}_\textup{fg}$, we have
\begin{equation}\label{weighted ergodic averages}
    \lim_{N\to\infty}\frac{1}{N}\sum_{n=1}^N\overline{f(n)}S_nF
    =\prod_{p\in\P}\Big(1-\frac{1}{p}\Big)\sum_{n\in\Qf_{\P^\text{c}}}\frac{\overline{f(n)}}{n}S_n(P_fF)
    \qquad\text{in}~L^2(X),
\end{equation}
where $\P=\{p\in\Pri\colon S_p(P_fF) \neq f(p)P_fF\}$.
\end{Theorem}

The following corollary of \cref{weighted PMET} resembles \cref{halasz mvt} in the ergodic setting.

\begin{Corollary}\label{Halasz gen}
Let $\xms$ be a finitely generated multiplicative system
and $f\in\Mult^\textup{c}_\textup{fg}$. Then the following hold:
\begin{itemize}
    \item[\textup{(i)}] If $F\in L^2(X)$ is such that $P_fF$ is constant, then
    \begin{equation}\label{weighted ergodic averages 1}
    \lim_{N\to\infty}\frac{1}{N}\sum_{n=1}^N\overline{f(n)}S_nF
    = M(\overline{f})\cdot \int_X F\d\mu \qquad\text{in}~L^2(X).
    \end{equation}
    Therefore,
    if for any non-constant $G\in L^2(X)$ we have $\D_G(S,f)=\infty$, then \eqref{weighted ergodic averages 1} holds for all $F\in L^2(X)$.
    \item[\textup{(ii)}] If a non-constant $F\in L^2(X)$ satisfies $\D_F(S,f)<\infty$, then 
    \begin{equation}\label{weighted ergodic averages 2}
    \lim_{N\to\infty}\frac{1}{N}\sum_{n=1}^N\overline{f(n)}S_nF
    =\prod_{p\in\P}\Big(1-\frac{1}{p}\Big)\sum_{n\in\Qf_{\P^\text{c}}}
    \frac{\overline{f(n)}}{n}S_nF
    \qquad\text{in}~L^2(X),
\end{equation}
where $\P=\{p\in\Pri\colon S_pF \neq f(p)F\}$, and this is not equal to $M(\overline{f})\cdot\int_X F\d\mu$.
\end{itemize}
\end{Corollary}

It can be easily checked that the right-hand sides of 
\eqref{weighted ergodic averages 1} and \eqref{weighted ergodic averages 2} 
are distinct (for non-constant functions).
\par
One could possibly expect that in \cref{Halasz gen} (i) we could have the stronger statement:
$$\D_F(S,f)=\infty \Longrightarrow
\lim_{N\to\infty}\frac{1}{N}\sum_{n=1}^N \overline{f(n)}S_nF
= M(\overline{f})\cdot \int_X F \d\mu,$$
as in \cref{halasz mvt}, but the following example shows that this is not generally the case. 

\begin{Example}
Let $(\Sone\times\Sone,\mu,S)$ where $\mu$ is the Haar measure on 
$\Sone\times\Sone$ and the action $S$ is given by
$S_p(z,w) = (-z,w)$, $(z,w)\in\Sone\times\Sone$, for any $p\in\Pri$. Let 
$F\in L^2(\Sone\times\Sone)$ given by $F(z,w) = \frac{z+w}{2}$,
$(z,w)\in\Sone\times\Sone$, which has $\int_{\Sone\times\Sone}F\d\mu=0$.
Then using \cref{spectral thm}, we have
$$\int_{\Mult^\text{c}}g(n)\d\mu_F(g)
= \langle S_nF,F\rangle
= \frac{\lambda(n)+1}{2}
= \int_{\Mult^\text{c}} g(n) \d\Big(\frac{\delta_\lambda+\delta_1}{2}\Big)(g),$$
where $\lambda$ is the Liouville function and $\delta_a$ denotes the Dirac point
mass of a function $a\in\Mult^\text{c}$. It follows that
$$\mu_F = \frac{\delta_\lambda+\delta_1}{2}.$$
Therefore, we have
$$\D_F(S,1)^2 = \int_{\Mult^\text{c}}\D(g,1)^2\d\mu_F(g)
= \frac{\D(\lambda,1)^2 + \D(1,1)^2}{2} = \infty,$$
but on the other hand we also have
$$\lim_{N\to\infty}\bigg\|\frac{1}{N}\sum_{n=1}^NS_nF\bigg\|_2
= \frac{1}{2}\bigg(\lim_{N\to\infty}\bigg|\frac{1}{N}\sum_{n=1}^N\lambda(n)\bigg|
+ \lim_{N\to\infty}\bigg|\frac{1}{N}\sum_{n=1}^N1\bigg|\bigg)
= \frac{1}{2}.$$
\end{Example}

Utilising \cref{Halasz gen} we will be able to establish the following spectral characterisations of aperiodicity and pretentious weak-mixing.

\begin{Corollary}\label{aperiodicity criterion}
The following are equivalent for a finitely generated multiplicative system $\xms$:
\begin{itemize}
    \item[\textup{(i)}] $S$ is aperiodic.
    \item[\textup{(ii)}] $\Hilb_1$ consists of constant functions and $\Hilb_\chi=\{0\}$ 
    for any non-principal Dirichlet character $\chi$.
    \item[\textup{(iii)}] $S$ is pretentiously ergodic and $\sigma_\textup{pr.rat}(S) = \A_1$.
\end{itemize}
\end{Corollary}

\begin{Corollary}\label{pretentious wm criterion}
The following are equivalent for a finitely generated multiplicative system $\xms$:
\begin{itemize}
    \item[\textup{(i)}] $S$ is pretentiously weak-mixing.
    \item[\textup{(ii)}] $\Hilb_1$ consists of constant functions and $\Hilb_f=\{0\}$ for any 
    $f\in\Mult^\text{c}_\text{fg}$ with $\D(f,1)=\infty$.
    \item[\textup{(iii)}] $S$ is pretentiously ergodic and $\sigma_\text{pr}(S) = \A_1$.
    \item[\textup{(iv)}] $S$ has no non-constant pretentious eigenfunctions.
\end{itemize}
\end{Corollary}

\cref{aperiodicity criterion} can be viewed as a dynamical generalisation of the following corollary of \cref{halasz mvt}:
\begin{Corollary}\label{Halasz aperiodic}
Any $f\in\Mult_\textup{fg}$ is aperiodic if and only if $\D(f,\chi)=\infty$
for any Dirichlet character $\chi$.
\end{Corollary}

The next result exhibits the relation between the classical and the pretentious dynamical 
behaviour of finitely generated multiplicative system and is an immediate consequence of 
Corollaries \ref{pret ergodicity criterion},
\ref{aperiodicity criterion} and \ref{pretentious wm criterion}.

\begin{Corollary}\label{pretentious props are stonger}
Any pretentiously ergodic (aperiodic, or pretentiously weak-mixing) finitely
generated multiplicative system is ergodic (totally ergodic, or 
weak-mixing respectively).
\end{Corollary}

The opposite of the above does not hold in general as one can see from the next two examples.

\begin{Example}\label{ergodic but not pretentiously}
Consider the multiplicative rotation $(\Sone,m_{\Sone},S)$ by the function $f\in\Mult^\text{c}_\text{fg}$
given by
$$f(p) = 
\begin{cases}
    e(\alpha), & \text{ if } p=2, \\
    1, & \text{otherwise},
\end{cases}$$
for some irrational $\alpha\in\R$. The only invariant functions of this system are the constants.
This implies that the system is ergodic in the classical sense.
To see this, any $S$-invariant function is also invariant under the action of the single 
transformation $S_2$. But since this transformation is clearly ergodic, then 
any such function has to be almost everywhere equal to some constant.
On the other hand, considering the identity function $F(z)=z$, we can see that 
$\D_F(S,1)<\infty$, hence $F$ is a non-constant pretentiously invariant function.
We expect then that the ergodic averages of $F$ do not converge in the expected limit.
To verify this, we have
$$\lim_{N\to\infty}\bigg\|\frac{1}{N}\sum_{n=1}^N S_nF\bigg\|_2
= \lim_{N\to\infty}\frac{1}{N}\sum_{n=1}^N f(n) = \frac{1}{2+e(\alpha)}
\neq 1 = \int_{\Sone} F \d m_{\Sone}.$$ 
Consequently, $S$ is not pretentiously ergodic.
\end{Example}

\begin{Example}\label{eigenvalues example}
Consider the multiplicative rotation $(\Sone,m_{\Sone},S)$ by the function $f\in\Mult^\text{c}_\text{fg}$
given by
$$f(p) = 
\begin{cases}
    -1, & \text{ if } p=2, \\
    \chi^\ast(p), & \text{otherwise},
\end{cases}$$
for some non-principal Dirichlet character $\chi$.
This system has no rational eigenvalues, but it is easy to check that $\chi$ is a pretentious rational eigenvalue
(with pretentious rational eigenfunction the identity). Thus, the system is totally ergodic with the classical
sense, but far from being aperiodic. Moreover, the additive ergodic averages along arithmetic progressions
are not controlled by the classical eigenfunctions and eigenvalues, but by the pretentious ones. To see this,
notice that, for the identity function $F\in L^2(X)$, we have
$$\lim_{N\to\infty}\bigg\|\frac{1}{N}\sum_{n=1}^N \overline{\chi}(n)S_nF\bigg\|_2 > 0,$$
which, as we will see later (see \cref{aperiodic functions in systems equiv}),
implies that there exist some $r,q\in\N$ such that
$$\lim_{N\to\infty}\frac{1}{N}\sum_{n=1}^N S_{qn+r}F \neq \int_{\Sone} F\d m_{\Sone} \qquad\text{in}~L^2(X).$$
\end{Example}

Now applying the results stated so far, we can characterise the pretentious dynamical
properties of the two classes of multiplicative systems in \cref{classes of multi systems}.

\begin{Corollary}\label{class of pret erg/aper systems}
Let $f\in\Mult^\textup{c}_\textup{fg}$ and $\xms$ be the multiplicative
rotation by $f$. Then the following hold:
\begin{itemize}
    \item[\textup{(i)}] $\xms$ is pretentiously ergodic if and only if $\D(f^k,1)=\infty$ for any 
    positive integer $k<|X|$.
    \item[\textup{(ii)}] $\xms$ is aperiodic if and only if $f^k$ is aperiodic for any 
    positive integer $k<|X|$.
    \item[\textup{(iii)}] $\xms$ is not pretentiously weak-mixing .
\end{itemize}
\end{Corollary}

For convenience we introduce the following terminology.
Let $S(p)$ be a statement depending on a prime variable $p$ and let 
$\P=\{p\in\Pri: S(p)~\text{is true}\}$. Then we say that $S$ holds for:
\begin{itemize}
    \item {\em almost every prime} if $\sum_{p\not\in\P}\frac{1}{p}<\infty$,
    \item {\em many primes} if $\sum_{p\in\P}\frac{1}{p}=\infty$,
    \item {\em few primes} if $\sum_{p\in\P}\frac{1}{p}<\infty$.
\end{itemize}

\begin{Corollary}\label{class of pret erg/aper systems 2}
Let $a:\N\to\Z$ be a finitely generated completely additive function and $\xmt$
be an additive system. Then the following hold:
\begin{itemize}
    \item[\textup{(i)}] $(X,\mu,T^a)$ is pretentiously ergodic if and only if 
    $\xmt$ is ergodic and $a$ satisfies:
    \begin{itemize}
        \item for any $\frac{r}{q}\in\sigma_\text{rat}(T)\setminus\{0\}$ with $(r,q)=1$, $q\nmid a(p)$ holds for many primes, and
        \item $a(p)\neq 0$ for many primes.
    \end{itemize}
    \item[\textup{(ii)}] $(X,\mu,T^a)$ is aperiodic if and only if $\xmt$ is ergodic
    and $e(a(n)\alpha)$ is aperiodic for any $\alpha\in(0,1)\cap\Q$.
    \item[\textup{(iii)}] $(X,\mu,T^a)$ is pretentiously weak-mixing if and only if
    $\xmt$ is weak-mixing and $a$ satisfies the conditions of (i).
\end{itemize}
\end{Corollary}

Given $\P\subset\Pri$, we define the completely additive function
$\Omega_\P(n)$ to be the number of prime factors of $n$ belonging in $\P$ counted with
multiplicity, and the additive function $\omega_\P$ to be the number of prime factors of $n$ belonging in $\P$ counted without
multiplicity. 
We also define the completely multiplicative function $\lambda_\P(n):=(-1)^{\Omega_\P(n)}$. 
Clearly, all these functions are finitely generated.

\begin{Remark}
In this paper, we prove several results concerning the completely additive functions 
$\Omega$ and $\Omega_\P$. It would not be too hard to prove the same results
for the additive functions $\omega$ and $\omega_\P$. It would require adapting
all the proofs of this paper concerning multiplicative systems to weakly multiplicative
ones.
\end{Remark}

\begin{Corollary}\label{Omega are good}
Let $a:\N\to\{-1,0,1\}$ be a completely additive function. Then the following hold:
\begin{itemize}
    \item[\textup{(i)}] $(X,\mu,T^a)$ is pretentiously ergodic for all ergodic $\xmt$
    if and only if $a(p)\neq 0$ for many primes.
    \item[\textup{(ii)}] If $a(p)=1$ for almost every prime, then $(X,\mu,T^a)$ is aperiodic for all ergodic $\xmt$.
\end{itemize}
In particular, for any ergodic $\xmt$, the system $(X,\mu,T^\Omega)$ is aperiodic, 
and for any $\P\subset\Pri$ the system $(X,\mu,T^{\Omega_\P})$ is:
\begin{itemize}
    \item pretentiously ergodic if $\P$ contains many primes,
    \item aperiodic if $\P$ contains almost every prime.
\end{itemize}
\end{Corollary}

Combining \ref{PMET} with \cref{Omega are good} (or more generally, with 
\cref{class of pret erg/aper systems 2}), we immediately obtain the following known result, 
which can be proved independently as well (see for example, \cite[Theorem 2.5]{Loyd}).
\begin{Corollary}\label{pret erg of omega}
For any additive system $\xmt$ and any $F\in L^2(X)$, we have
$$\lim_{N\to\infty}\frac{1}{N}\sum_{n=1}^N T^{\Omega(n)}F = \int_X F \d\mu
\qquad\text{in}~L^2(X).$$
This remains true if we replace $\Omega$ with $\Omega_\P$, where $\P\subset\Pri$ is any set containing many primes, or more generally, with any $a$ satisfying the conditions in \cref{class of pret erg/aper systems 2} (i).
\end{Corollary}

This result is the norm convergence analogue of the pointwise result \cite[Theorem A]{bergelson-richter} concerning the convergence of averages of  $F(T^{\Omega(n)}x)$ for all points $x$ in uniquely ergodic systems.
The same result for (non-uniquely) ergodic systems and for almost every point $x$ is false, as it was shown in \cite[Theorem 1.2]{Loyd}.

\subsection{Decomposition theorems}

Decomposition theorems on Hilbert spaces play an essential role in ergodic theory,
especially in the study of ergodic averages. We are concerned with 
decomposition theorems in which the $L^2$ space of a measure-preserving system
is split into two components (subspaces) that exhibit contrasting behaviours.
Such theorems allow us to express any $L^2$ function $F$ as $F_1+F_2$, 
where $F_1$ is ``structured" and $F_2$ is ``pseudo-random".
When studying the convergence of ergodic averages, 
expressing an arbitrary $L^2$ function in this way allows us to reduce the problem
to conducting an appropriate analysis on the prescribed subspaces to which
the individual components belong.
The goal of this section is to generalise two of the most well-known 
decomposition theorems for additive systems to the setting of multiplicative systems. 
This is paramount to understanding the pretentious dynamical behaviour 
of finitely generated multiplicative systems. We remark that the first new structure theorem
that we present below is a key ingredient in the proof of \ref{mt1}.
\par
Given an additive system $\xmt$, let 
$$\Hilb_\text{tot.erg}(T)=\Big\{f\in L^2(X)\colon
\lim_{N\to\infty}\bigg\|\frac{1}{N}\sum_{n=1}^N T^{qn}F\bigg\|_2 = 0, \text{ for all } q\in\N\Big\}$$
and
$$\Hilb_\text{rat}(T)=\overline{\{F\in L^2(X)\colon
T^qF=F \text{ for some } q\in\N\}}.$$
A classical corollary of \cref{MET}, which for instance
can be found in \cite[p. 14]{bergelson_update}, asserts that
\begin{equation}\label{add decomp}
    L^2(X) = \Hilb_\text{tot.erg}(T) \oplus \Hilb_\text{rat}(T).
\end{equation}
Our first decomposition result is an analogue  of \eqref{add decomp} for
finitely generated multiplicative systems. Given such a system $\xms$,
we define the following closed subspaces of $L^2(X)$:
\begin{align*}
    \Hilb_\text{pr.rat}
    & := \overline{\text{span}\{F\in L^2(X)\colon F \text{ is a pretentious 
    rational eigenfunction for } S\}} \\
    & = \overline{\text{span}\{F\in L^2(X)\colon \text{there exists a Dirichlet character }
    \chi \text{ such that } \D_F(S,\chi)<\infty\}} \\
    & = \overline{\text{span}\bigg(\bigcup_\chi \Hilb_\chi\bigg)},
\end{align*}
and
$$\Hilb_\text{aper}
:= \Big\{F\in L^2(X)\colon \lim_{N\to\infty}\bigg\|\frac{1}{N}\sum_{n=1}^N S_{qn+r}F\bigg\|_2
= 0 \text{ for any } r\in\N,q\in\N\ \Big\}.$$
When confusion may arise, we mention the action considered by writing
$\Hilb_\text{pr.rat}(S)$ and $\Hilb_\text{aper}(S)$.

\begin{Theorem}\label{decomposition}
Let $\xms$ be a finitely generated multiplicative system. Then we have
$$L^2(X) = \Hilb_\textup{pr.rat} \oplus \Hilb_\textup{aper}.$$
\end{Theorem}

Our second decomposition theorem is an analogue  of the classical 
Jacobs-de Leeuw-Glicksberg theorem (see \cite{JdLG}), 
according to which the $L^2$ space of an additive system is decomposed as 
the direct sum of the subspace of almost-periodic functions
(the closed subspace of $L^2$ spanned by eigenfunctions) and the weak-mixing component
(the closed subspace of $L^2$ consisting of weak-mixing functions). 
Now, given a finitely generated multiplicative system $\xms$, we define the
following closed subspaces of $L^2(X)$:
\begin{align*}
    \Hilb_\text{pr.eig}
    & := \overline{\text{span}\{F\in L^2(X)\colon F \text{ is a pretentious 
     eigenfunction for } S\}} \\
    & = \overline{\text{span}\{F\in L^2(X)\colon \text{there exists } 
    f\in\Mult^\text{c}_\text{fg} \text{ such that } \D_F(S,f)<\infty\}} \\
    & = \overline{\text{span}\bigg(\bigcup_{f\in\Mult^\text{c}_\text{fg}} \Hilb_f\bigg)},
\end{align*}
and
$$\Hilb_\text{pr.wm}
:= \Big\{F\in L^2(X)\colon \lim_{N\to\infty}\frac{1}{N}\sum_{n=1}^N 
\bigg|\int_X S_nF\cdot \overline{F} \d\mu\bigg| = 0 \Big\}.$$
When confusion may arise, we mention the action considered by writing
$ \Hilb_\text{pr.eig}(S)$ and $\Hilb_\text{pr.wm}(S)$.

\begin{Theorem}\label{decomposition2}
Let $\xms$ be a finitely generated multiplicative system. Then we have
$$L^2(X) = \Hilb_\textup{pr.eig} \oplus \Hilb_\textup{pr.wm}.$$
\end{Theorem}

\subsection{Joint ergodicity of additive and multiplicative actions}

For the rest of the introduction, motivated by another number-theoretic topic,
we explore how additive and multiplicative actions intertwine.
Gaining better understanding on the ways in which additive and multiplicative structures of 
positive integers interact with each other is a fundamental objective in number theory.
It is generally believed that such structures are independent when no ``local obstructions''
arise. This philosophy underpins several theorems, open problems and conjectures in number theory.
Chowla's conjecture (see \cite{Chowla}) serves as a great 
example to elaborate on this topic. In its simplest form
-which is still open- it asserts that for all $h\in\N$,
$$\lim_{N\to\infty} \frac{1}{N} \sum_{n=1}^N \lambda(n)\lambda(n+h) = 0,$$
or in other words, that the sequences $(\lambda(n))_{n\in\N}$ and 
$(\lambda(n+h))_{n\in\N}$ exhibit independent asymptotic behaviour. 
The asymptotic independence of arithmetic functions is captured by the 
following notion:
Given two bounded functions $f,g:\N\to\C$, we say that $f$ and $g$ are {\em (asymptotically) uncorrelated} if 
$$\lim_{N\to\infty}\bigg(\frac{1}{N}\sum_{n=1}^N f(n)\overline{g(n)}
- \bigg(\frac{1}{N}\sum_{n=1}^N f(n)\bigg) 
\bigg(\frac{1}{N}\sum_{n=1}^N \overline{g(n)}\bigg)\bigg)
= 0.$$
\par
In \cite{bergelson-richter}, Bergelson and Richter explored the concept 
of independence of additive and multiplicative structures in the context of 
topological dynamical systems, by examining when two sequences arising from an additive and a multiplicative system respectively are uncorrelated.
Our next goal is to transfer the independence principle postulated in \cite{bergelson-richter}
to the setting of ergodic theory.
\par
Let us first see how the notion of independence can be understood in the ergodic
theoretic setting, analogously to the way that is described for arithmetic functions
above. Given a Hilbert space $(\Hilb,\|\cdot\|)$, 
we say that two essentially bounded functions $F,G:\N\to\Hilb$ are {\em asymptotically uncorrelated} if 
$$ \lim_{N\to\infty}\bigg\|\frac{1}{N}\sum_{n=1}^N F(n)\overline{G(n)}
- \bigg(\frac{1}{N}\sum_{n=1}^N F(n)\bigg) 
\bigg(\frac{1}{N}\sum_{n=1}^N \overline{G(n)}\bigg)\bigg\|
= 0. $$
An additive action $T$ and a multiplicative action $S$ 
on a probability space $\xm$ are {\em independent} if for any $F,G\in L^\infty(X)$,
$(T^nF)_{n\in\N}$ and $(S_nG)_{n\in\N}$ are asymptotically uncorrelated as sequences
in the Hilbert space $(L^2(X),\|\cdot\|_2)$ .
Under the natural assumptions that $T$ is ergodic and $S$ is pretentiously ergodic, 
and in view of \cref{MET} and \ref{PMET}, the notion
of independence coincides with that of joint ergodicity.

\begin{Definition}\label{joint ergodicity}
Let $\xm$ be a probability space, $T$ be an ergodic additive action 
and $S$ be an pretentiously ergodic multiplicative action on $\xm$. 
We say that $T$ and $S$ are {\em jointly ergodic} if for any 
$F,G\in L^\infty(X)$,
$$\lim_{N\to\infty} \frac{1}{N}\sum_{n=1}^N T^nF \cdot S_nG
= \int_X F \d\mu \int_X G\d\mu \qquad\text{in}~L^2(X),$$
or equivalently if $T$ and $S$ are independent.
\end{Definition}

We transfer the independence principle in the ergodic-theoretic setting as follows:
``An ergodic additive action $T$ and a pretentiously ergodic 
finitely generated multiplicative action $S$ on some probability space $\xm$ are jointly ergodic when no ``local obstructions" arise".
The second main theorem of this paper verifies this principle by identifying 
the ``local obstructions" that violate the independence of such actions.
What we found is that the obstructions are caused by the dependence of their 
{\em (pretentiously) periodic parts}.
\par
The {\em spectrum} of an additive action $T$ can be defined as
$$\sigma(T) := \{\alpha\in[0,1) \colon e(\alpha) \text{ is an eigenvalue for } T\}$$
and the {\em rational spectrum} of $T$ is given by
$$\sigma_{\text{rat}}(T) := \sigma(T)\cap\Q.$$
The rational spectrum captures the {\em periodic behaviour} of the action $T$.
In particular, $T$ is totally ergodic if and only if $T$ is ergodic and
$\sigma_\text{rat}(T)=\{0\}$. 
\par
Let $S$ be a finitely generated multiplicative action. We have seen that the dynamical
behaviour of $S$ with respect to additive ergodic averages is captured by the 
pretentious spectrum. The {\em pretentious periodic part} of $S$ is 
detected by the pretentious rational spectrum $\sigma_\text{pr.rat}(S)$.
We claim that ``local obstructions'' between $T$ and $S$ occur if and only if 
$\sigma_\text{rat}(T)$ and $\sigma_\text{pr.rat}(S)$ are dependent. 
Hence, we shall find a way to link pretentious rational eigenvalues for $S$ to
rational numbers $\frac{r}{q}$ which induce the eigenvalues for $T$. To this end, 
we define
$$\widetilde{\sigma}_{\text{pr.rat}}(S) 
= \bigg\{\frac{r}{q}\in\Q \colon (r,q)=1~\text{and}~\exists~\chi\in\sigma_\text{pr.rat}(S)~\text{primitive Dirichlet character mod}~q\bigg\}\cup\{0\}$$

The next lemma follows immediately from \cref{aperiodicity criterion}.

\begin{Lemma}\label{aperiodicity and new pret rat spectrum}
A finitely generated multiplicative system $\xms$ is aperiodic if and only if it is pretentiously ergodic and 
$\widetilde{\sigma}_\textup{pr.rat}(S)=\{0\}$.
\end{Lemma}

We are now ready to state the second main theorem of this paper.

\begin{named}{Theorem B}{Joint ergodicity of actions}\label{mt1}
Let $\xm$ be a probability space, $T$ be an ergodic additive action on $\xm$ 
and $S$ be a pretentiously ergodic finitely generated multiplicative action on $\xm$.
Then $T,S$ are jointly ergodic if and only if 
$\sigma_\textup{rat}(T)\cap\widetilde{\sigma}_\textup{pr.rat}(S)=\{0\}$.
\end{named}

A way to interpret this theorem is that under some natural assumptions,
the independence of the actions is guaranteed by the independence of their (pretentiously) 
periodic parts.
\par
Using the fact that for an ergodic $T$ and a pretentiously ergodic $S$ we have that
$\sigma_\text{rat}(T)=\{0\}$ if and only if $T$ is totally ergodic, and
$\widetilde{\sigma}_{\text{pr.rat}}(S)=\{0\}$ if and only if $S$ is aperiodic (by \cref{aperiodicity and new pret rat spectrum}), 
\ref{mt1} immediately implies the following:

\begin{Corollary}\label{cor1}
Let $\xm$ be a probability space, $T$ be an ergodic additive action on $\xm$
and $S$ be a pretentiously ergodic finitely generated multiplicative action
on $\xm$. 
If $T$ is totally ergodic, or $S$ is aperiodic,
then $T$ and $S$ are jointly ergodic.
\end{Corollary}

This corollary is the norm convergence analogue of the pointwise 
result \cite[Theorem C]{bergelson-richter} concerning the independence
(or disjointness, in the language of \cite{bergelson-richter}) 
of finitely generated multiplicative dynamical systems and nilsystems.

\cref{cor1} leads to nice applications that are to be presented
in the following subsection.

\subsection{An application of \ref{mt1}}

A central topic in ergodic theory is concerned with understanding the limiting behaviour 
of multiple ergodic averages of the form
\begin{equation}\label{MEA}
    \frac{1}{N}\sum_{n=1}^N T^{a_1(n)}F_1\cdot\ldots\cdot T^{a_\ell(n)}F_\ell
    \qquad\text{in}~L^2,
\end{equation}
for any (ergodic) additive system $\xmt$ and any $F_1,\dots,F_\ell\in L^2(X)$,
for various families of sequences $a_k:\N\to\Z$, $1\leq k\leq\ell$.
The study of multiple ergodic averages was initiated in Furstenberg's seminal work \cite{Furstenberg_Szemeredi}, 
where he gave an ergodic-theoretic proof of \Szemeredi's theorem on arithmetic progressions
(see \cite{Szemeredi}). Furstenberg's result corresponds to the case $a_k(n) = kn$, $1\leq k\leq\ell$
in \eqref{MEA}.
Subsequently to Furstenberg's work, there have been several works concerned 
with the averages in \eqref{MEA}, where the iterates are taken along various families 
of sequences such as polynomials (\cite{Bergelson_PET}, \cite{Bergelson-Leibman}, \cite{Kra-F}, \cite{HK_nonconventional},\cite{BLL}), Hardy field sequences (\cite{Frantzi_Hardy1}, \cite{Frantzi_bracket}, \cite{Kara-Koutso}, \cite{BMR2}, \cite{BMR1}, \cite{Koutso}, \cite{frantzi}, \cite{Tsinas_JE}) or fractional prime powers (\cite{Frantzi_primes}). Here we obtain 
a result on multiple ergodic averages as an application of \ref{mt1}, which
is of different flavor compared to the aforementioned ones.
\par
Combining \cref{cor1} with \cref{Omega are good} (or more generally, with \cref{class of pret erg/aper systems 2}), we obtain the following:

\begin{Corollary}\label{j.e of n with omega}
For any ergodic invertible measure-preserving transformations $T_1,T_2$ on some probability space $\xm$ and for any $F,G\in L^2(X)$, we have
\begin{equation}\label{j.e equation of n with omega}
    \lim_{N\to\infty}\frac{1}{N}\sum_{n=1}^NT_1^nF\cdot T_2^{\Omega(n)}G 
    = \int_X F\d\mu \int_X G\d\mu \qquad \text{in}~L^2(X).
\end{equation}

The result remains true if we replace $\Omega$ with $\Omega_\P$,
for any $\P\subset\Pri$ containing almost every prime, or 
more generally, with any finitely generated completely additive function $a:\N\to\Z$
such that $e(a(n)\alpha)$ is aperiodic for all $\alpha\in(0,1)\cap\Q$.
\end{Corollary}

\begin{Remark}
A quite surprising fact in \ref{mt1}, and consequently in Corollaries \ref{cor1} and \ref{j.e of n with omega},
is that $T$ and $S$ are not required to commute with each other. This makes \cref{j.e of n with omega} one of the very few results where convergence of the form \eqref{j.e equation of n with omega} holds for non-commuting transformations. Similar results to \cref{j.e of n with omega} with $\Omega$ replaced by any increasing sequence are known to be false for non-commuting transformations (see \cite[Lemma 4.1]{frantzikinakis2012}).
\end{Remark}

In the language of Frantzikinakis \cite{frantzi}, \cref{j.e of n with omega} applied for the same transformation $T_1=T_2=T$
says that $n$ and $\Omega(n)$ are jointly ergodic, and more generally, that
if $a:\N\to\Z$ is any function as in \cref{j.e of n with omega}, then $n$ and $a(n)$
are jointly ergodic.

\cref{j.e of n with omega} is equivalent to the following multiple recurrence result.
Before we state it we shall give the notion of upper density.
Given a set $E\subset\N$, the {\em upper density} of $E$ is given by
$$\overline{d}(E):=\limsup_{N\to\infty}\frac{|E\cap[1,N]|}{N}.$$

Our last results hold in the more general context where $\Omega$ is 
replaced by any function $a:\N\to\Z$ as in \cref{j.e of n with omega}, and 
in particular by $\Omega_\P$ for any $\P\subset\Pri$ containing almost every prime,
but we state them only for $\Omega$ for sake of simplicity.

\begin{Corollary}\label{multiple recurrence of omega} 
For any ergodic invertible measure-preserving transformations $T_1,T_2$ on some probability space $\xm$ and for any $A\subset X$ with $\mu(A)>0$,
we have
$$\lim_{N\to\infty}\frac{1}{N}\sum_{n=1}^N\mu(A\cap T_1^{-n}A\cap T_2^{-\Omega(n)}A)
\geq \mu(A)^3.$$
In particular, for any $\epsilon>0$, we have
$$\overline{d}(\{n\in\N\colon \mu(A\cap T_1^{-n}A\cap T_2^{-\Omega(n)}A)\geq \mu(A)^3-\epsilon\})>0.$$
\end{Corollary}

Furstenberg's correspondence principle \cite[Lemma 3.17]{Furst} offers a standard way in which 
multiple recurrence results are utilized to obtain combinatorial applications
concerning the richness of arithmetic structures in large subsets of the integers. For our purposes we use the following form the correspondence:

\begin{Theorem}[Furstenberg's correspondence principle, see {\cite[Theorem 1.1]{Bergelson_ERT87}}]\label{correspondence}
Let $E\subset\N$ be a set of positive upper density. 
Then there exists an invertible additive system $\xmt$ and
a measurable set $A$ with $\mu(A)=\overline{d}(E)>0$, such that for any
$k\in\N$ and any $n_1, n_2, \dots, n_k\in\Z$, we have
$$\overline{d}(E\cap (E-n_1)\cap\dots\cap(E-n_k))\geq \mu(A\cap T^{-n_1}A\cap\dots\cap T^{-n_k}A).$$
\end{Theorem}

Combining \cref{multiple recurrence of omega} with \cref{correspondence}
yields the following combinatorial application saying that large 
subsets of the positive integers contain many configurations of the 
form $\{m,m+n,m+\Omega(n)\}$.

\begin{Corollary}\label{combi app for omega}
Let $E\subset\N$ be a set of upper density $\delta>0$.
Then for any $\epsilon>0$, we have
$$\limsup_{N\to\infty}\limsup_{M\to\infty}
\frac{|\{(n,m)\in[1,N]\times[1,M]\colon m,m+n,m+\Omega(n)\in E\}|}{NM}
\geq 
\delta^3-\epsilon.$$
In particular, there exists $n\in\N$ such that 
$$\{m,m+n,m+\Omega(n)\}\subset E$$
holds for $m\in\N$ in a set of positive upper density.
\end{Corollary}

\begin{Remark}
All the results in this subsection concerning the sequences $n$ and $\Omega(n)$
(or $n$ and $a(n)$, where $a$ is as in \cref{j.e of n with omega})
remain true if we replace the sequence $n$ with any polynomial with no constant term.
The same proof that we will give for $n$ works for such polynomials, utilising also
Weyl's equidistribution theorem for polynomials (see \cite{Weyl_equid}).
\end{Remark}

\subsection{Proof ideas}

We conclude the introduction 
with a brief discussion about the ideas and tools that are used to
prove our main theorems.

\textbf{Proof of \ref{PMET}:}
For the proof of the pretentious mean ergodic theorem and in fact, for the proof of its
weighted version (\cref{weighted PMET}), we combine ideas and methods from ergodic theory and number theory. For the rest of this paragraph we consider a finitely generated multiplicative system $\xms$.
First, we prove a decomposition result in \cref{small decomposition}, according to which we can
split $L^2(X)$ into $\Hilb_f$ and $\V_f$, for any
$f\in\Mult^\text{c}_\text{fg}$, where $\Hilb_f$ is the subspace of $L^2(X)$ defined in
\cref{pretentious eigen def} (i) and $\V_f$ is the closed subspace of $L^2(X)$
consisting of functions whose weighted ergodic averages converge to zero. 
This allows us to reduce the proof of \cref{weighted PMET} to dealing with 
weighted ergodic averages of pretentious eigenfunctions. To handle these,
we adapt a classical convolution idea, that was used in \cite{Delange} to calculate 
mean values of multiplicative functions in the pretentious case, to our setting.
Now, to prove the decomposition $L^2(X) = \Hilb_f \oplus \V_f$ in \cref{small decomposition}, we first show that $\Hilb_f\perp\V_f$ (\cref{orthogonality prop 1}) and then that $\Hilb_f^\perp \subset\V_f$ (\cref{orthogonality prop 2}).
To prove the latter we use a similar method to the one used to prove such results in additive systems, and the main tools we use are the spectral theorem on unitary multiplicative actions
(see \cref{spectral thm}) along with \cref{halasz mvt}.
To establish the orthogonality of $\Hilb_f$ and $\V_f$, all we have to do is prove that
if the averages of form $\overline{f(n)}S_nF$ are zero, then the same averages over a set
of $\P$-free numbers, where $\P\subset\Pri$ contains few primes, is also zero (see \cref{Q_P convergence corollary}). Its proof 
is quite technical and involves some number-theoretic ideas along with a simple approximation argument.
\par
\textbf{Proof of \ref{mt1}:}
The proof of ``$T,S$ are jointly ergodic $\Longrightarrow
\sigma_\textup{rat}(T)\cap \widetilde{\sigma}_{\text{pr.rat}}(S) = \{0\}$'' does not rely on any of
the new results that we develop on multiplicative systems. We only use \cref{spectral thm} along with
some number-theoretic results concerning correlations of linear phases with multiplicative functions 
pretending to be some Dirichlet character. The proof of converse implication is more interesting and challenging, as it heavily depends on some
of the new results on multiplicative systems. The critical idea in this proof is the use of 
decomposition theorems for both actions. We want to show that \Cesaro{} averages of $T^nF\cdot S_nF$
converge to the product of the integrals of $F$ and $G$. Assuming that the integral of $G$ is zero,
we want to show that the above averages converge to zero. Using the classical decomposition result
in \eqref{add decomp}, we can write $F = F_\text{tot.erg} + F_\text{rat}$.
Using an orthogonality criterion (\cref{orthogonality criterion}),
we deduce that the totally ergodic component has zero contribution to the averages
in question, hence we only have to deal with the averages of $T^nF_\text{rat}\cdot S_nG$.
Now, using the decomposition result for the multiplicative action given in \cref{decomposition},
we can write $G=G_\text{aper} + G_\text{pr.rat}$. Then we are reduced to 
handling the averages of $T^nF_\text{rat}\cdot S_nG_\text{aper}$ and $T^nF_\text{rat}\cdot S_nG_\text{pr.rat}$.
By some simple calculations, we easily see that the averages of the former expression
converge to zero, while for the averages of the latter one we employ our spectral assumption along with
some number-theoretic input.

\vspace{2mm} 
\noindent 
\textbf{Acknowledgements.}
The author would like to thank Florian K. Richter for his invaluable guidance
and support throughout the writing of this paper, as well as Ethan Ackelsberg, 
Vitaly Bergelson, Felipe Hern\'andez Castro, Nikos Frantzikinakis and the anonymous referee for their 
useful comments and suggestions on earlier drafts.
The author gratefully acknowledges support from the 
Swiss National Science Foundation grant TMSGI2-211214.


\section{Preliminaries}\label{prelims}

\subsection{Basics in multiplicative systems}

\begin{Lemma}[Triangle inequalities for the distance]\label{first props of distance}
Let $\xms$ be a multiplicative system. Then the following hold for any 
$F,G\in L^2(X)$ and for any $f,g\in\Mult^\text{c}$:
\begin{itemize}
    \item[\textup{(i)}] $\D_F(S,f) \leq \D_F(S,g) + \|F\|_2\D(f,g)$.
    \item[\textup{(ii)}] $\D(f,g) \leq \frac{1}{\|F\|_2}(\D_F(S,f) + \D_F(S,g))$.
    \item[\textup{(iii)}] $\D_{F+G}(S,f)\leq \D_F(S,f) + \D_G(S,f)$.
    \item[\textup{(iv)}] $\D_{FG}(S,fg)\leq \|G\|_2\D_F(S,f) + \|F\|_2\D_G(S,g)$.
\end{itemize}
All the triangle inequalities are also true for the partial
distances up to any $N\in\N$.
\end{Lemma}
\begin{proof}
The proofs of all the triangle inequalities are identical, so we just prove the first one.
\par
By the triangle inequality for the $L^2(X)$ norm and by expanding
the square and then by using Cauchy-Schwarz inequality for the middle term, 
we have that
\begin{align*}
\D_F(S,f)^2 & = \frac{1}{2}\sum_{p\in\Pri} \frac{\|S_pF - f(p)F\|_2^2}{p} 
\leq \frac{1}{2}\sum_{p\in\Pri} \frac{(\|S_pF-g(p)F\|_2+\|F\|_2|f(p)-g(p)|)^2}{p} \\
& \leq \D_F(S,g)^2 + 2\|F\|_2\D_F(S,f)\D(f,g) + \D_F(f,g)^2 
= (\D_F(S,g) + \|F\|_2 \D(f,g))^2.
\end{align*}
The proof is complete.
\end{proof}

Given a multiplicative system $\xms$, $f\in\Mult^\text{c}$ and $F\in L^2(X)$
with $\|F\|_2\leq 1$, we have that
\begin{equation}\label{Dist comparison ineq}
    \D_F(S,f)^2 \leq \sum_{p\in\Pri} \frac{\|S_pF-f(p)F\|_2}{p} \leq 4\D_F(S,f)^2,
\end{equation}
and also for any $f,g\in\Mult^\text{c}$, we have that 
\begin{equation}\label{Dist comparison ineq 2}
    \D(f,g)^2 \leq \sum_{p\in\Pri} \frac{|f(p)-g(p)|}{p} \leq 4\D(f,g)^2,
\end{equation}
This can be easily checked, since for any $z\in\C$ with $|z|\leq 1$, it holds:
$\frac{1}{2}|z-1|^2 \leq |z-1| \leq 2|z-1|^2$.

\begin{Lemma}\label{set of eigenvalues}
Let $\xms$ be a finitely generated multiplicative system and
$F\in L^2(X)$ be a pretentious eigenfunction. 
Then the set of all pretentious eigenvalues for $S$ 
corresponding to $F$ is exactly $\A_f$, 
where $f$ is any pretentious eigenvalue corresponding to $F$.
\end{Lemma}

\begin{proof}
Fix a pretentious eigenvalue $f$ corresponding to $F$ and let $g\in\A_f$.
Then by the first triangle inequality, we have that
$\D_F(S,g) \leq \D_F(S,f) + \|F\|_2\D(f,g) < \infty$, by assumption.
On the other hand, let $g\in\Mult^\text{c}$ be a pretentious eigenvalue of $S$ 
corresponding to $F$. Then by \cref{first props of distance} (ii), 
we have that $\D(f,g) \leq \frac{1}{\|F\|_2}(\D_F(S,f) + \D_F(S,g)) < \infty$,
by assumption. This shows that $g\in\A_f$, concluding the proof.
\end{proof}

\begin{Remark}
In view of \cref{set of eigenvalues}, whenever we consider a pretentious rational
eigenvalue $f\not\in\A_1$, we can assume that it is a Dirichlet character. Moreover, we may also 
assume that this Dirichlet character is primitive. 
\par
To see why the latter is true,
let $\chi$ be a Dirichlet character mod $q$ that is a pretentious rational eigenvalue
and let $\chi_1$ the unique primitive Dirichlet character inducing $\chi$ with
modulus $q_1$. Then $q_1\mid q$ and we have 
$$\chi(n) 
= \begin{cases}
    \chi_1(n), & \text{if } (n,q)=1, \\
    0, & \text{otherwise}.
\end{cases}$$
Then $\D(\chi_1,\chi)<\infty$, and by the first triangle inequality,
it follows that $\chi_1$ is a pretentious rational eigenfunction corresponding
to the same pretentious rational eigenfunction as $\chi$.
\end{Remark}

Consider an arbitrary finitely generated multiplicative system $\xms$. 
It is not hard to see that for any $F\in L^2(X)$ and any $f\in\Mult_\text{fg}^\text{c}$ 
if $S_pF = f(p)F$ for almost every prime, then $F$ is a pretentious eigenfunction 
with pretentious eigenvalue $f$. In the next lemma, we show that the converse is also true.

\begin{Lemma}\label{almost eigenfunction lemma}
Let $\xms$ be a finitely generated multiplicative system. For any
$F\in L^2(X)$ and any $f\in\Mult^\text{c}_\text{fg}$, $\D_F(S,f)<\infty$ if and only if 
$S_pF=f(p)F$ for almost every prime.
In other words, 
$\Hilb_f = \{F\in L^2(X)\colon S_pF=f(p)F ~\text{for almost every prime}\}.$
\end{Lemma}

This lemma provides us with a powerful characterisation of pretentious eigenvalues
that we will use several times in what follows (without explicitly referring to the
lemma every time). However, we avoid using it when it is not necessary. 
Moreover, it yields a similar result for multiplicative functions (\cref{fg functions description}) as an immediate consequence.

\begin{proof}[Proof of \cref{almost eigenfunction lemma}]
Let $f\in\Mult^\text{c}_\text{fg}, F\in\Hilb_f$ and we shall show that $S_pF=f(p)F$ holds
for almost every prime. Consider integers $k,\ell$, measure-preserving
transformations $T_1,\dots,T_\ell$ and a partition of primes 
$(\P_i)_{1\leq i\leq\ell}$ with $\P_i=\{p\in\Pri\colon S_p=T_i\}$ for any 
$1\leq i\leq\ell$, such that for any $1\leq i\leq k$ we have 
$\sum_{p\in\P_i}\frac{1}{p}=\infty$ and for any $k+1\leq i\leq\ell$
we have $\sum_{p\in\P_i}\frac{1}{p}<\infty$.
Moreover, there exist some integer $m\geq1$, distinct $a_1,\dots,a_m\in\C$ 
and a partition of primes $(\P_j')_{1\leq j\leq m}$ with $\P_j'=\{p\in\Pri\colon f(p)=a_i\}$,
such that for any $1\leq j\leq m$ we have $\sum_{p\in\P_j'}\frac{1}{p}=\infty$.
Here we have assumed that none of the finitely many values $a_1,\dots,a_m$
of $f$ is taken in a set $\P_j'\subset\Pri$ containing few primes. We are allowed to do so, because
otherwise, we can always find $g$ that is close to $f$ and satisfies 
this assumption, and we know that then $\Hilb_f=\Hilb_g$.
\par
Now consider the partition 
$\{\P_i\cap\P_j'\colon 1\leq i\leq\ell, 1\leq j\leq m\}$
of the primes and let 
$$\P = 
\bigcup_{\substack{k+1\leq i\leq\ell \\ 1\leq j\leq m}}
\P_i\cap\P_j'.$$
Clearly, $\sum_{p\in\P}\frac{1}{p}<\infty$. Then we have
\begin{equation}\label{almost eigenf eq 1}
    \D_F(S,f)^2 
    = \sum_{\substack{1\leq i\leq\ell \\ 1\leq j\leq m}}
    \sum_{p\in\P_i\cap\P_j'}\frac{\|S_pF-f(p)F\|_2^2}{p}
    = \sum_{\substack{1\leq i\leq k \\ 1\leq j\leq m}}
    \|T_iF-a_jF\|_2^2\sum_{p\in\P_i\cap\P_j'}\frac{1}{p} 
    + \Oh(1).
\end{equation}
\underline{Claim}.
For any $1\leq i\leq k$ there exists a unique
$1\leq j(i)\leq m$ such that $\sum_{p\in\P_i\cap\P_{j(i)}'}\frac{1}{p}=\infty$.
\begin{proofclaim}
Fix $1\leq i\leq k$. Assume for contradiction that for any $1\leq j\leq m$, $\sum_{p\in\P_i\cap\P_j'}\frac{1}{p}<\infty$. Since $(\P_j')_{1\leq j\leq m}$ is a partition of 
primes, it follows that $\sum_{p\in\P_i}\frac{1}{p}<\infty$, yielding a contradiction.
Thus, there exists some $j(i)$ satisfying the claim. Now suppose again for
contradiction that there exists $j\neq j(i)$ also satisfying the claim. Then, by \eqref{almost eigenf eq 1}, we have that 
$a_jF = T_iF = a_{j(i)}F$, since otherwise we would have $\D_F(S,f)=\infty$. It follows that $a_j=a_{j(i)}$, yielding a contradiction. This concludes the proof of the claim.
\end{proofclaim}

Note that the Claim implies that $m\leq k$. 
Using the Claim, \eqref{almost eigenf eq 1} becomes
\begin{equation}\label{almost eigenf eq 3}
    \D_F(S,f)^2 = 
    \sum_{i=1}^k\|T_iF-a_{j(i)}F\|_2^2\sum_{p\in\P_i\cap\P_{j(i)}'}\frac{1}{p}
    +\Oh(1).
\end{equation}
By assumption, the left-hand side of \eqref{almost eigenf eq 3} is finite, and by the Claim, for each $1\leq i \leq k$, we have $\sum_{p\in\P_i\cap\P_{j(i)}}\frac{1}{p}=\infty$. Thus, for each $1\leq i\leq k$, we have $T_iF = a_{j(i)}F$.

This proves that if $S_pF\neq f(p)F$, then $p$ must belong in the set
$$\widetilde{\P}
:=\P \cup \bigg(\bigcup_{\substack{1\leq i\leq k\\ 1\leq j\leq m \\ j\neq j(i)}} 
\P_i\cap\P_j'\bigg),$$
which is exactly the complement of $\bigcup_{1\leq i\leq k}\P_i\cap\P_{j(i)}'$.
It follows that $\sum_{p\in\widetilde{\P}}\frac{1}{p}<\infty$. 
This concludes the proof.
\end{proof}

We continue with some useful results concerning the spectral measures of 
$L^2$ functions in multiplicative systems (see \cref{spectral thm}).
Throughout, we consider $\Mult^\text{c}$ equipped with the pointwise convergence topology and the 
Borel $\sigma$-algebra. Moreover, the set $\A_f$, for any $f\in\Mult^\text{c}$, is Borel measurable.
To see why, observe that for any $f\in\Mult^\text{c}$ and any $N\in\N$, the function 
$\D(\cdot,f;N)$ is continuous, hence the set 
$\{g\in\Mult^\text{c}\colon \D(g,f;N)\leq k\}$ is Borel measurable for any integer $k\geq0$.
It follows that $\{g\in\Mult^\text{c}\colon \D(g,f)\leq k\}$ is Borel measurable for any $k\geq0$
and then $\A_f$ is also Borel measurable, as a countable union of Borel measurable sets.
\par
Now, for any multiplicative system $\xms$ and any $F\in L^2(X)$, we denote by 
$\mu_F$ the unique finite Borel measure in $\Mult^\text{c}$ satisfying 
$$\langle S_nF,F\rangle = \int_{\Mult^\text{c}} f(n)\d\mu_F(f).$$
Moreover, for any finite Borel measure $\nu$ on $\Mult^\text{c}$, we define, for each $n\in\N$,
$e_n\in L^2(\Mult^\text{c},\nu), e_n(f) := f(n)$.

\begin{Lemma}\label{spectral thm lemma}
Let $\xms$ be a multiplicative system, $F\in L^2(X)$ and consider a 
measurable set $\A\subset\Mult^\text{c}$ such that $\mu_F(\A)>0$. 
Then there exists $G\in L^2(X)$ such that
$\mu_G$ is supported on $\A$.
\end{Lemma}
\begin{proof}
We may assume that $\|F\|_2=1$ and we will find $G\in L^2(X)$ with $\|G\|_2=1$ such that
$\mu_G$ is supported on $\A$.
By \cref{spectral thm}, there is a unitary isomorphism
$\Phi$ from $L^2(\Mult^\text{c},\mu_F)$ to the cyclic sub-representation of $L^2(X)$
which is generated by $F$. Consider the non-zero function 
$\mu_F(\A)^{-1/2}\1_\A\in L^2(\Mult^\text{c},\mu_F)$
and let $G\in L^2(X)$ be the image of this element under $\Phi$.
This isomorphism also conjugates $S_n$ to multiplication by
$e_n$ for any $n\in\N$. Therefore, applying \cref{spectral thm} twice, first for $G$ and then for $F$, we have that 
\begin{align*}
    \int_{\Mult^\text{c}} f(n) \d\mu_G(f)
    & = \int_X S_nG\cdot \overline{G}\d\mu
    = \int_{\Mult^\text{c}} f(n)\mu_F(\A)^{-1/2}\1_\A(f)
    \overline{\mu_F(\A)^{-1/2}\1_\A(f)}\d\mu_F(f) \\
    & = \mu_F(\A)^{-1}\int_\A f(n) \d\mu_F(f)
    = \int_{\Mult^\text{c}} f(n) \d\big(\mu_F(\A)^{-1}\mu_F|_\A\big)(f).
\end{align*}
By uniqueness of the spectral measure, it follows that
$\mu_G = \mu_F(\A)^{-1}\mu_F|_\A$, and hence we have that $\mu_G$
is supported on $\A$.
Moreover, using again \cref{spectral thm},
it is easy to check that $\|G\|_2 = 1$.
The proof is complete.
\end{proof}

\begin{Lemma}\label{spectral measure supported on fg functions}
Let $\xms$ be a finitely generated multiplicative system. Then, for any
$F\in L^2(X)$, the spectral measure $\mu_F$ of $F$ is supported on $\Mult^\textup{c}_\textup{fg}$.
\end{Lemma}

\begin{proof}
Suppose that for any $p\in\Pri$, $S_p\in\{T_1,\dots,T_d\}$, for some 
$d\in\N$ and some measure preserving transformations $T_1,\dots, T_d$ on $X$.
Let $F\in L^2(X)$, let $\mu_F$ be its spectral measure with respect to
the $(\N,\times)$ action $S$, and let $\nu_F$ be its spectral measure 
with respect to the $(\N^d,+)$ action induced by $T_1,\dots,T_d$.
For any $n\in\N$, there exist $k_1,\dots,k_d\in\N$ such that
$S_nF = T_1^{k_1}\cdots T_d^{k_d}F$. Therefore, by applying \cref{spectral thm} for $S$ 
and \cref{classic spectral thm} for $S_n$, it follows that for any $n\in\N$,
$$\int_{\Mult^\text{c}} f(n)\d\mu_F(f) 
= \int_X S_nF\cdot\overline{F} \d\mu 
= \int_{\T^d} e(k_1x_1+\cdots k_dx_d)\d\nu_F(x_1,\dots,x_d).$$
This defines a map $\eta:(\N,\times)\to (\N^d,+)$, which induces a map 
$\widehat{\eta}: \T^d=\widehat{(\N^d,+)}\to\widehat{(\N,\times)}=\Mult^\text{c}$
such that $\eta$ maps the measure $\nu_F$ to the measure $\mu_F$.
The map $\eta$ is given by 
$\widehat{\eta}(x_1,\dots,x_d)(n)
= e(\langle(x_1,\dots,x_d),\eta(n)\rangle_{(\N^d,+)})$,
where $\langle\mathbf{x},\mathbf{y}\rangle_{(\N^d,+)}
:= x_1y_1 + \dots + x_dy_d$ for any 
$\mathbf{x}=(x_1,\dots,x_d),\mathbf{y}=(y_1,\dots,y_d)\in\N^d$.
By assumption, the map $\eta$ is finitely generated, since for any $p\in\Pri$,
$$\int_{\Mult^\text{c}}f(p)\d\mu_F(f)\in 
\bigg\{\int_\T e(x_i)\d(\nu_F)_i \colon i\in\{1,\dots,d\}\bigg\},$$
where $(\nu_F)_i$ is the projection of the measure $\nu_F$ on the $i$-th 
coordinate (or equivalently, the spectral measure of $F$ with respect to $T_i)$.
Therefore, $\widehat{\eta}$ is finitely generated. Moreover, since 
$\mu_F$ is the push-forward of $\nu_F$ with respect to $\widehat{\eta}$,
it follows that $\mu_F$ is supported on a subset of $\widehat{\eta}(\T^d)$,
and any $f\in\widehat{\eta}(\T^d)$ is finitely generated. The result then follows.
\end{proof}

It follows from the previous lemma that for any finitely generated multiplicative 
system and any $F\in L^2(X)$, we have 
$$\langle S_nF,F\rangle = \int_{\Mult^\text{c}_\text{fg}}f(n) \d\mu_F(f).$$

\begin{Lemma}\label{relation between the two distances}
Let $\xms$ be a finitely generated multiplicative system. For any $F\in L^2(X)$
and any $f\in\Mult^\text{c}_\text{fg}$, we have that 
\begin{equation}\label{relation between the two distances eq}
    \D_F(S,f)^2 = \int_{\Mult^\text{c}_\text{fg}} \D(g,f)^2 \d\mu_F(g).
\end{equation}
Moreover, we have that
\begin{equation}\label{Hilb_f and spectral measure eq}
   \Hilb_f
   \subset \{F\in L^2(X)\colon \mu_F~\text{is supported on}~\A_f\cap\Mult^\text{c}_\text{fg}\}
   \subset \overline{\Hilb_f}.
\end{equation}
\end{Lemma}

\begin{proof}
Let $F\in L^2(X)$ and $f\in\Mult^\text{c}_\text{fg}$. We may assume that $\|F\|_2=1$.
For any $N\in\N$, by \cref{spectral thm}, we have 
\begin{align*}
    \D_F(S,f;N)^2 
    & = \sum_{p\leq N} \frac{\|F\|_2^2 - \Re\langle S_pF,f(p)F \rangle}{p}
    = \sum_{p\leq N} \frac{1-\Re\big(\int_{\Mult^\text{c}_\text{fg}} g(p)\overline{f(p)}\d\mu_F(g)\big)}{p} \\
    & = \int_{\Mult^\text{c}_\text{fg}} \D(g,f;N)^2 \d\mu_F(g).    
\end{align*}
For any $N\in\N$, we define $\Psi_N,\Psi:\Mult^\text{c}_\text{fg}\to[0,\infty]$, by
$\Psi_N(g) = \D(g,f;N)^2, \Psi(g) = \D(g,f)^2$. 
Then $(\Psi_N)_{N\in\N}$ is an increasing sequence converging pointwise to $\Psi$,
so by the monotone convergence theorem, we obtain 
$$\D_F(S,f)^2
= \lim_{N\to\infty}\int_{\Mult^\text{c}_\text{fg}} \D(g,f;N)^2 \d\mu_F(g)
= \int_{\Mult^\text{c}_\text{fg}} \D(g,f)^2 \d\mu_F(g).$$
This completes the proof of \eqref{relation between the two distances eq}.
\par
Now we prove \eqref{Hilb_f and spectral measure eq}. The first inclusion is immediate by 
\eqref{relation between the two distances eq}. Thus, we show the second inclusion.
Let $F,f$ be as above and assume that $\mu_F$ is supported on 
$\A_f' := \A_f\cap\Mult^\text{c}_\text{fg}$. For any integer $k\geq0$,
we define the Borel measurable set $\A_{f,k}' = \{g\in\Mult^\text{c}_\text{fg}\colon \D(g,f)\leq k\}$.
Since the sequence $(\A_{f,k}')_{k\in\N}$ is increasing and the union of all the sets
in the sequence is the set $\A_f'$, we have that
$$\lim_{k\to\infty}\mu_F(\A_{f,k}') = \mu_F(\A_f') =1,$$
hence $\mu_F(\A_{f,k}')\geq 1/2$ for all sufficiently large $k$. Therefore, we may assume 
that $\mu_F(\A_{f,k}')\geq 1/2>0$ for all $k\geq 0$.
Consider the unitary isomorphism $\Phi$ from $L^2(\Mult^\text{c}_\text{fg},\mu_F)$
to the cyclic sub-representation of $L^2(X)$ which is generated by $F$,
as guaranteed by \cref{spectral thm}. For each $k\geq0$,
we let $F_k = \Phi(\1_{\A_{f,k}'})$, and since $\mu_F(\A_{f,k}')>0$, then, as we did in the proof 
of \cref{spectral thm lemma}, we can show that $\mu_{F_k}$ is supported on $\A_{f,k}'$.
Then, by \eqref{relation between the two distances eq}, we have that $\D_{F_k}(S,f)\leq k$,
hence $F_k\in\Hilb_f$ for all $k\geq0$.
Now, by the dominated convergence theorem, we have that 
$$\lim_{k\to\infty}\|F_k-F\|_2^2
= \lim_{k\to\infty}\int_X |F_k-F|^2 \d\mu
= \lim_{k\to\infty}\int_X |\1_{\A_{f,k}'} - \1_{\A_{f}'}|^2 \d\mu
= 0.$$
This shows that $F\in\overline{\Hilb_f}$, establishing the second conclusion. 
The proof of \eqref{Hilb_f and spectral measure eq}, and hence the proof of the lemma, is complete.
\end{proof}

\begin{Remark}\label{pret eigenvalues are fg}
In any finitely generated multiplicative system, \eqref{Hilb_f and spectral measure eq} implies
that all the pretentious eigenvalues are finitely generated.
\end{Remark}

Now we establish some interesting properties of the subspaces $\Hilb_f$.

\begin{Lemma}\label{eigenspaces are subspaces}
Let $\xms$ be a finitely generated multiplicative system and $f\in\Mult^\text{c}$. 
Then $\Hilb_f$ is a subspace of $L^2(X)$.
\end{Lemma}
\begin{proof}
Let $F,G\in\Hilb_f$ and $a,b\in\C$. By the third triangle inequality, 
we have
$$\D_{aF+bG}(S,f)\leq |b|\|G\|_2\D_{aF}(S,f) + |a|\|F\|_2\D_{bG}(S,f)
= |b|\|G\|_2\D_F(S,f) + |a|\|F\|_2\D_G(S,f) < \infty.$$
Thus, $aF+bG\in\Hilb_f$.
This shows that $\Hilb_f$ is a subspace of $L^2(X)$.
\end{proof}

\begin{Lemma}\label{same eigenspaces}
Let $\xms$ be a finitely generated multiplicative system and $f,g\in\Mult^\text{c}$. 
Then $\Hilb_f = \Hilb_g$ if and only if $\D(f,g)<\infty$. In fact, 
if $\D(f,g)=\infty$, then $\Hilb_f\perp\Hilb_g$.
\end{Lemma}
\begin{proof}
We may assume that $\|F\|_2,\|G\|_2\leq 1$.
\par
Suppose first that $\D(f,g)<\infty$.
Then by the first triangle inequality, it follows that $\Hilb_f=\Hilb_g$.\par
Now suppose that $\D(f,g)=\infty$. Let $F\in\Hilb_f$ and $G\in\Hilb_g$.
Suppose, for sake of contradiction, that $\langle F,G\rangle\neq 0$. 
We have that
$$ \langle F,G\rangle (1-f(p)\overline{g(p)}) 
= \langle S_pF-f(p)F,S_pG\rangle + \langle f(p)F, S_pG-g(p)G\rangle, $$
hence, by the triangle inequality and the Cauchy-Schwarz inequality, we have that
$$ \sum_{p\in\Pri}\frac{|\langle F,G\rangle (1-f(p)\overline{g(p)})|}{p}
\leq \sum_{p\in\Pri}\frac{\|S_pF-f(p)F\|_2\|G\|_2}{p} + \sum_{p\in\Pri}\frac{\|F\|_2\|S_pG-g(p)G\|_2}{p}, $$
which, in view of \eqref{Dist comparison ineq} and \eqref{Dist comparison ineq 2} gives that
$$|\langle F,G \rangle| \D(f,g)^2 \leq 4\D_F(S,f)^2 + 4\D_G(S,g)^2.$$
Since $\langle F,G \rangle \neq 0$, the left-hand side of the above diverges, while
the right-hand side converges, by assumption. This yields a contradiction.
Therefore, $\langle F,G\rangle = 0$, showing that $\Hilb_f\perp\Hilb_g$.
\end{proof}

The following remark concerning the subspaces $\Hilb_f$ and constant functions will be useful later.
\begin{Remark}\label{constants}
Given a multiplicative system $\xms$, for any constant function $F\in L^2(X)$,
we have that $\D_F(S,f)^2 = \|F\|_2^2 \D(f,1)^2$. So, in finitely generated systems, we have that
$\Hilb_f$ contains the constants if $\D(f,1)<\infty$, and it does not contain any constant otherwise. 
\end{Remark}

In \cref{H1 prop}, we saw that in any finitely generated multiplicative system $\xms$,
$\overline{\Hilb_1}$ coincides with the square-integrable $\sigma(\Ipr)$-measurable functions.
In the next lemma, we see that $\Ipr$ is an algebra, and this fact will be useful to show that \ref{PMET} 
implies \cref{halasz mvt}.

\begin{Lemma}\label{invariant algerba}
Let $\xms$ be a finitely generated multiplicative system. 
Then $\Ipr$ is an algebra.
\end{Lemma}
\begin{proof}
To show that $\Ipr$ is an algebra, we show that it contains $X$ 
and it is closed under complements and finite intersections:
\begin{itemize}
    \item $\1_X=1$ and so $X\in\Ipr$.
    \item Let $A\in\Ipr$. Then we have
    $$\D_{\1_{A^\textup{c}}}(S,1)^2 
    = \frac{1}{2}\sum_{p\in\Pri} \frac{\mu(S_p^{-1}A^\textup{c}\triangle A^\textup{c})}{p}
    = \frac{1}{2}\sum_{p\in\Pri} \frac{\mu(S_p^{-1}A\triangle A)}{p}
    = \D_{\1_A}(S,1)^2 
    < \infty,$$
    hence, $A^\textup{c}\in\Ipr$.
    \item Let $A_1,\dots,A_k\in\Ipr$. Then we have
    $$\bigg(S_p^{-1}\bigcap_{i=1}^k A_i\bigg)\setminus\bigg(\bigcap_{j=1}^k A_j\bigg) 
    \subset \bigcup_{j=1}^k (S_p^{-1}A_j\setminus A_j)
    \quad\text{and}\quad\bigg(\bigcap_{i=1}^k A_i\bigg)\setminus\bigg(S_p^{-1}\bigcap_{j=1}^k A_j\bigg)
    \subset\bigcup_{j=1}^k(A_j\setminus S_p^{-1}A_j),$$
    hence we have
    $$\bigg(S_p^{-1}\bigcap_{i=1}^k A_i\bigg)\triangle\bigg(\bigcap_{j=1}^k A_j\bigg)
    \subset \bigcup_{j=1}^k (S_p^{-1}A_j\triangle A_j).$$
    Therefore, 
    \begin{align*}
        \D_{\1_{\bigcap_{i=1}^k A_i}}(S,1)^2 = \frac{1}{2}\sum_{p\in\Pri}
        \frac{\mu\big(\big(S_p^{-1}\bigcap_{i=1}^k A_i\big)\triangle\big(\bigcap_{j=1}^k A_j\big)\big)}{p}
        & \leq \sum_{j=1}^k\frac{1}{2}\sum_{p\in\Pri}\frac{\mu(S_p^{-1}A_j\triangle A_j)}{p} \\
        & = \sum_{j=1}^k \D_{\1_{A_j}}(S,1)^2 
        < \infty.
    \end{align*}
\end{itemize}
This concludes the proof of the lemma.
\end{proof}

The next result is a simple lemma that provides us with
characterisations for aperiodic systems and aperiodic $L^2$ functions.

\begin{Lemma}\label{aperiodic functions in systems equiv}
Let $\xms$ be a finitely generated multiplicative system and $F\in L^2(X)$. 
Then the following are equivalent for any $q\in\N$:
\begin{itemize}
    \item[\textup{(i)}] $\lim_{N\to\infty}\frac{1}{N}\sum_{n=1}^N S_{qn+r}F = \int_X F\d\mu$
    in $L^2(X)$, for any $r\in\N$.
    \item[\textup{(ii)}] $\lim_{N\to\infty}\frac{1}{N}\sum_{n=1}^N
    e(\frac{rn}{q})S_nF =
   \begin{cases}
        \int_X F \d\mu, & \text{if}~q\mid r, \\
        0, & \text{otherwise}
    \end{cases}$
    in $L^2(X)$,
    for any $r\in\N$.
    \item[\textup{(iii)}] $\lim_{N\to\infty}\frac{1}{N}\sum_{n=1}^N \chi(n)S_nF =
    \begin{cases}
        \frac{1}{\phi(q)}\int_X F \d\mu, & \text{if}~\chi=\chi_0, \\
        0, & \text{otherwise}
    \end{cases}$
    in $L^2(X)$, for any Dirichlet character $\chi$ of modulus $q$, where $\chi_0$
    denotes the principal character of this modulus.\footnote{
    $\chi_0$ always stands for a principal character.}
\end{itemize}
\end{Lemma}
This lemma is proved in the \cref{proof of aperiodic equiv}.

\subsection{An orthogonality criterion}

The following lemma is a variant of a classical result of \Katai{} concerning complex-valued
arithmetic functions (see \cite[Lemma 1]{daboussi1975}, \cite[Eq. (3.1)]{katai}, 
\cite[Theorem 2]{BSZ}; see also \cite[Proposition 9.5]{Frantzi_Host}), 
which we are stating for arbitrary Hilbert spaces.

\begin{Lemma}\label{Katai for Hilbert}
Let $\Hilb$ be a Hilbert space and $A(n)\in\Hilb$ for each $n\in\N$. If $\P\subset\Pri$
is a set of positive relative density\footnote{We say that the set $\P\subset\Pri$ 
has positive relative density if $\lim_{N\to\infty}\frac{|\P\cap[1,N]|}{\pi(N)}$ exists
and is positive, where $\pi(N):=|\{p\in\Pri\colon p\leq N\}|$.} 
such that for any distinct $p_1,p_2\in\P$, 
$$\lim_{N\to\infty}\frac{1}{N}\sum_{n=1}^N \langle A(p_1n),A(p_2n)\rangle = 0,$$
then 
$$\lim_{N\to\infty}\bigg\|\frac{1}{N}\sum_{n=1}^N A(n)\bigg\|= 0.$$
\end{Lemma}

If $\Hilb=\C$, then \cref{Katai for Hilbert} reduces to the classical statement 
of the lemma whose proof can be found in the above mentioned references.
Replacing the complex-valued arithmetic function with an arithmetic function 
taking values on a Hilbert space, and the modulus with the norm on this Hilbert space,
does not affect the proof at all, hence there is no need to repeat it here.

As a consequence we obtain the following orthogonality criterion that is a key ingredient in proving \ref{mt1}.

\begin{Lemma}\label{orthogonality criterion}
Let $\xm$ be probability space, $T$ be an additive
and $S$ be a finitely generated multiplicative action on $\xm$.
Let also $b:\N\to\Z$ and $F\in L^2(X)$ with $\int_X F \d\mu = 0$.
If for any distinct $p_1,p_2\in\Pri$,
$$\lim_{N\to\infty}\frac{1}{N}\sum_{n=1}^N 
\int_X T^{b(p_1n)-b(p_2n)}F\cdot\overline{F} \d\mu = 0,$$
then for any $G\in L^2(X)$, we have
$$\lim_{N\to\infty}\bigg\|\frac{1}{N}\sum_{n=1}^N 
T^{b(n)}F \cdot S_nG\bigg\|_2 = 0.$$
\end{Lemma}

\begin{proof}
First, note that we may assume that $|G|\leq 1$ and then using that any such $G$ can be written as $G=\frac{1}{2}(G_1+G_2)$ for some $G_1,G_2$ with $|G_1|=|G_2|=1$, we may further assume that $|G|=1$.
Since $S$ is finitely generated, there exist invertible measure-preserving transformations $T_1,\dots,T_d$ on $\xm$ such that $\{S_p\colon p\in\Pri\}=\{T_1,\dots,T_d\}$.
We define for each $1\leq i\leq d$ the set $\P_i=\{p\in\Pri\colon S_p=T_i\}$ and then we have $\Pri = \bigcup_{i=1}^d \P_i$. It follows that there exists some $1\leq i_0\leq d$ such that $\P:=\P_{i_0}$ has positive relative density. 
We now define $A(n) = T^{b(n)}F\cdot S_nG$ for each $n\in\N$ and then, in view of \cref{Katai for Hilbert}, it suffices to show that 
\begin{equation}\label{orth crit suff eq}
    \lim_{N\to\infty}\frac{1}{N}\sum_{n=1}^N \langle A(p_1n),A(p_2n)\rangle = 0
    \qquad\forall~p_1\neq p_2\in\P.
\end{equation}
For any distinct $p_1,p_2\in\P$,
we have 
\begin{align*}
    \langle A(p_1n),A(p_2n)\rangle
    & = \int_X T^{b(p_1n)}F\cdot T^{b(p_2n)}\overline{F}\cdot S_n(T_{i_0}G)\cdot S_n(T_{i_0}\overline{G}) \d\mu \\
    & = \int_X T^{b(p_1n)}F\cdot T^{b(p_2n)}\overline{F}\cdot S_n(T_{i_0}|G|^2) \d \mu \\
    & = \int_X T^{b(p_1n)}F\cdot T^{b(p_2n)}\overline{F} \d\mu
    = \int_X T^{b(p_1n)-b(p_2n)}F\cdot\overline{F} \d\mu,
\end{align*}
thus, \eqref{orth crit suff eq} holds, concluding the proof.
\end{proof}

\begin{Remark}
The proof of \cref{Katai for Hilbert} is an application of the \Turan-Kubilius inequality. Doing the proof, one could easily check that the assumption there could be weakened as follows: for any distinct $p_1,p_2\in\P$, 
$$\lim_{N\to\infty}\frac{1}{N}\sum_{\substack{n=1 \\p_1,p_2\nmid n}}^N \langle A(p_1n),A(p_2n)\rangle = 0.$$
Using this strengthened form of \cref{Katai for Hilbert}, \cref{orthogonality criterion} can be shown to hold in the more general case of $S$ being weakly multiplicative.
\end{Remark}


\section{Proof of the results for multiplicative systems}\label{multi proof}

\subsection{Proof of \cref{H1 prop}}

First we show that $\overline{\Hilb_1}\subset L^2(X,\sigma(\Ipr),\mu)$. It is enough to show that
$\Hilb_1\subset L^2(X,\sigma(\Ipr),\mu)$, and hence it suffices to show that functions of $\Hilb_1$ 
are $\sigma(\Ipr) $-measurable.
\par
Let $F\in\Hilb_1$ and we may assume that it is real valued. Then $S_pF=F$ for almost every $p$.
Let $A = \{x\in X\colon F(x)<a\}$ for some $a\in\R$.
Notice that $S_p^{-1}A\triangle A \subset \{x\in X\colon F(S_px)\neq F(x)\}$, thus
for almost every prime we have $\mu(S_p^{-1}A\triangle A)=0$. This implies that $A\in\Ipr\subset\sigma(\Ipr)$,
which shows that $F$ is $\sigma(\Ipr)$-measurable, establishing the first inclusion.
\par
For the other inclusion, we claim that $\overline{\Hilb_1}=L^2(X,\mathcal{C},\mu)$, 
for some sub-$\sigma$-algebra $\mathcal{C}$ of the Borel and then we will show that $\Ipr\subset\mathcal{C}$.
By \cite[Lemma 3.1]{fk91}, the claim is equivalent to the following statement: $\overline{\Hilb_1}$ 
is a closed subspace of $L^2(X)$, containing the constants, and there is a set of bounded functions 
$\mathcal{U}$ that is dense in $\overline{\Hilb_1}$,
such that for any $F,G\in\mathcal{U}$, we have $FG\in\overline{\Hilb_1}$.
Since $\overline{\Hilb_1}$ is indeed a closed subspace of $L^2(X)$ and the fact that it contains the 
constants can be readily checked, it is enough to show the last assertion.
\par
Let $\mathcal{U}=\Hilb_1\cap L^\infty(X)$. This is a set of bounded functions that is dense in 
$\overline{\Hilb_1}$. Let $F,G\in\mathcal{U}$. Then $\D_F(S,1)<\infty, \D_G(S,1)<\infty$ and $FG\in L^2(X)$.
Then, by the third triangle inequality in \cref{first props of distance}, we have that 
$$\D_{FG}(S,1)\leq \|G\|_2\D_F(S,1) + \|F\|_2\D_G(S,1)<\infty,$$
which shows that $FG\in \Hilb_1\subset\overline{\Hilb_1}$.
\par
We have proved that $\overline{\Hilb_1} = L^2(X,\mathcal{C},\mu)$ for some $\sigma$-algebra 
$\mathcal{C}$. 
Now observe that by the definition of $\Hilb_1$ and $\Ipr$, we have that
$\Ipr\subset\mathcal{C}$, and then $\sigma(\Ipr)\subset\mathcal{C}$.
Hence, $L^2(X,\sigma(\Ipr),\mu)\subset\overline{\Hilb_1}$.
This concludes the proof.

\subsection{Proof of \ref{PMET}}

We will now prove \cref{weighted PMET}, which implies \ref{PMET}.
For the rest of this subsection, we fix a finitely generated multiplicative system $\xms$. 
For any $f\in\Mult^\text{c}_\text{fg}$, we define $\V_f$ to be the closed subspace of $L^2(X)$
given by 
$$\V_f := \Big\{F\in L^2(X) \colon \lim_{N\to\infty}\bigg\|\frac{1}{N}\sum_{n=1}^N 
\overline{f(n)}S_nF\bigg\|_2
= 0 \Big\}.$$
The fact that this a closed subspace can be easily checked.
We begin by proving the following decomposition result.
\begin{Theorem}\label{small decomposition}
For any $f\in\Mult^\textup{c}_\textup{fg}$, we have 
$$L^2(X) = \overline{\Hilb_f} \oplus \V_f.$$
\end{Theorem}

Combining \cref{same eigenspaces} and \cref{small decomposition} we have that
for any $f,g\in\Mult^\text{c}_\text{fg}$ with $\D(f,g)<\infty$, it holds that 
$\V_f = \V_g$.
In other words, we have the following:

\begin{Corollary}\label{change of function in averages}
For any $f,g\in\Mult^\text{c}_\text{fg}$ with $\D(f,g)<\infty$ and any $F\in L^2(X)$, we have that
$$\lim_{N\to\infty}\frac{1}{N}\sum_{n=1}^N f(n)S_nF = 0 \qquad \text{in}~L^2(X)$$
if and only if 
$$\lim_{N\to\infty}\frac{1}{N}\sum_{n=1}^N g(n)S_nF = 0 \qquad \text{in}~L^2(X).$$
\end{Corollary}

To prove \cref{small decomposition}, we split it into the following two propositions.
\begin{Proposition}\label{orthogonality prop 1}
For any $f\in\Mult^\textup{c}_\textup{fg}$, we have $\Hilb_f\perp\V_f$.
\end{Proposition}

\begin{Proposition}\label{orthogonality prop 2}
Let $F\in L^2(X)$ and $f\in\Mult^\textup{c}_\textup{fg}$. 
If $F\perp \Hilb_f$, then
$$\lim_{N\to\infty}\bigg\|\frac{1}{N}\sum_{n=1}^N \overline{f(n)}S_nF\bigg\|_2 = 0.$$
Consequently, we have $\Hilb_f^\perp\subset\V_f$.
\end{Proposition}
We remark that we shall not need the full strength of \cref{small decomposition} to prove 
\cref{weighted PMET}, as \cref{orthogonality prop 2} is enough. Nevertheless, the full 
strength of \cref{small decomposition} will be useful in the proof of 
\cref{decomposition}.

We start with the proof of \cref{orthogonality prop 1}, for which we need 
several preliminary results that we are going to state and prove below. 
\par
For simplicity we adopt the notation $\E_{n\in A}a(n):=\frac{1}{|A|}\sum_{n\in A} a(n)$
for any finite $A\subset\N$ and any arithmetic function $a$.
Given $f\in\Mult_\text{fg}^\text{c}$, $F\in L^2(X)$ and $A\subset\N$, if the sequence of $L^2(X)$ functions $\big(\E_{n\in A\cap[1,N]}\overline{f(n)}S_nF\big)_{N\in\N}$ converges in $L^2(X)$, then we denote its $L^2$-limit by $\E_{n\in A}\overline{f(n)}S_nF$.

\begin{Lemma}\label{existence of L^2-limits}
Let $f\in\Mult_\text{fg}^\text{c}$, $F\in L^2(X)$ and $\P\subset\Pri$. Then $\E_{n\in\N}\overline{f(n)}S_nF$,
$\E_{n\in\N}\1_{\Qf_\P}(n)\overline{f(n)}S_nF$
$\E_{n\in\Qf_\P}\overline{f(n)}S_nF$ are well-defined functions in $L^2(X)$, and moreover,
\begin{equation}\label{Qf averages equation}
    \E_{n\in\Qf_\P}\overline{f(n)}S_nF 
    = \frac{1}{d(\Qf_\P)}\E_{n\in\N}\1_{\Qf_\P}(n)\overline{f(n)}S_nF
\end{equation}
\end{Lemma}

\begin{proof}
If we show that the second function is well-defined, then it trivially follows that the first one is also well-defined. Moreover, using that
$$\E_{n\in\Qf_\P\cap[1,N]}\overline{f(n)}S_nF = \frac{N}{|\Qf_\P\cap[1,N]|}\E_{n\leq N}\1_{\Qf_\P}(n)\overline{f(n)}S_nF,$$
if we show that the second function is well-defined, it follows that the third one is also well-defined and that \eqref{Qf averages equation} holds.
Hence we only have to show that $\E_{n\in\N}\1_{\Qf_\P}(n)\overline{f(n)}S_nF$ is well-defined. It suffices to prove that the sequence
$\big(\E_{n\leq N}\1_{\Qf_\P}(n)\overline{f(n)}S_nF\big)_{N\in\N}$ is Cauchy with respect to the $L^2$-norm. For positive integers $N,M$, using \cref{spectral thm}, we have 
\begin{align*}
    & \Big\|\E_{n\leq N}\1_{\Qf_\P}(n)\overline{f(n)}S_nF - \E_{m\leq M}\1_{\Qf_\P}(m)\overline{f(m)}S_mF\Big\|_2^2 = \\
    & =  
    \int_{\Mult_\text{fg}^\text{c}}\Big|\E_{n\leq N}\1_{\Qf_\P}(n)\overline{f(n)}g(n)\Big|^2\d\mu_F(g)
    + \int_{\Mult_\text{fg}^\text{c}}\Big|\E_{m\leq M}\1_{\Qf_\P}(m)\overline{f(m)}g(m)\Big|^2\d\mu_F(g) \\
    & \:\:\:\:\: - 2\Re\bigg(\int_{\Mult_\text{fg}^\text{c}} \Big(\E_{n\leq N}\1_{\Qf_\P}(n)\overline{f(n)}g(n)\Big)\overline{\Big(\E_{m\leq M}\1_{\Qf_\P}(m)\overline{f(m)}g(m)\Big)}\d\mu_F(g)\bigg).
\end{align*}
For any $g\in\Mult_\text{fg}^\text{c}$, the function $\1_{\Qf_\P}(n)\overline{f(n)}g(n)$ is also in $\Mult_\text{fg}^\text{c}$ and then by \cref{halasz mvt}, its mean value exists. Thus, sending $N,M\to\infty$, and using the dominated convergence theorem, the last equation gives
\begin{multline*}
    \lim_{N,M\to\infty}\Big\|\E_{n\leq N}\1_{\Qf_\P}(n)\overline{f(n)}S_nF - \E_{m\leq M}\1_{\Qf_\P}(m)\overline{f(m)}S_mF\Big\|_2^2 = \\
    = \int_{\Mult_\text{fg}^\text{c}}|M(\1_{\Qf_\P}\overline{f}g)|^2\d\mu_F(g) + \int_{\Mult_\text{fg}^\text{c}}|M(\1_{\Qf_\P}\overline{f}g)|^2\d\mu_F(g)
    - 2\Re\bigg(\int_{\Mult_\text{fg}^\text{c}}|M(\1_{\Qf_\P}\overline{f}g)|^2\d\mu_F(g)\bigg)
    = 0.
\end{multline*}
This completes the proof.
\end{proof}

\begin{Proposition}\label{finite Q_P convergence lemma}
Let $f\in\Mult^\textup{c}_\textup{fg}$, $F\in\V_f$ and $\P\subset\Pri$ be finite. Then 
$$\E_{n\in\Qf_\P}\overline{f(n)}S_nF = 0.$$
\end{Proposition}
\begin{proof}
The proof of the proposition consists of two steps.
First, we write $\P = \{p_1,\dots,p_s\}$, where $s=|\P|\in\N$.\\
\underline{Claim 1}. For any $F\in L^2(X)$, we have
$$\E_{n\in\N} \overline{f(n)}S_nF
= \sum_{k_1,\dots,k_s\in\N\cup\{0\}}\prod_{i=1}^s\frac{(p_i-1)\overline{f(p_i)^{k_i}}}{p_i^{k_i+1}}
\bigg(\prod_{i=1}^s S_{p_i^{k_i}}\bigg)
\Big(\E_{n\in\Qf_\P} \overline{f(n)}S_nF\Big).$$
\begin{proofclaim1}
We prove the claim by induction on $s\in\N$.
Let $s=1$ and then we write $\P=\{p\}$. Notice that we can express the natural numbers as
$$\N = \bigsqcup_{k\in\N\cup\{0\}}p^k\Qf_\P,$$
where throughout we use the symbol $\sqcup$ to denote a disjoint union. Then, using \eqref{Qf averages equation}, we have
\begin{align*}
    & \E_{n\in\N} \overline{f(n)}S_nF
    = \sum_{k\in\N\cup\{0\}}\E_{n\in\N}
    \1_{p^k\Qf_\P}(n)\overline{f(n)}S_nF \\
    & = \lim_{N\to\infty}\sum_{k\in\N\cup\{0\}}\overline{f(p)^k}S_{p^k}\bigg(\frac{1}{N}
    \sum_{n\leq N/p^k}\1_{\Qf_\P}(n)\overline{f(n)}S_nF\bigg) \\
    & = \lim_{N\to\infty}\sum_{k\in\N\cup\{0\}}\frac{\overline{f(p)^k}}{p^k}S_{p^k}\Big(
    \E_{n\leq N/p^k}\1_{\Qf_\P}(n)\overline{f(n)}S_nF\Big) 
    = \sum_{k\in\N\cup\{0\}}\frac{\overline{f(p)^k}}{p^k}S_{p^k}
    \Big(\E_{n\in\N}\1_{\Qf_\P}(n)\overline{f(n)}S_nF\Big) \\
    & = d(\Qf_\P)\sum_{k\in\N\cup\{0\}}\frac{\overline{f(p)^k}}{p^k}S_{p^k}
    \Big(\E_{n\in\Qf_\P}\overline{f(n)}S_nF\Big) 
    = \sum_{k\in\N\cup\{0\}}\frac{(p-1)\overline{f(p)^k}}{p^{k+1}}S_{p^k}
    \Big(\E_{n\in\Qf_\P}\overline{f(n)}S_nF\Big),
\end{align*}
where all the limits are in $L^2(X)$.
This proves the base case of the induction. Now let $s\geq2$ and suppose
that the statement is true for all the positive integers smaller than $s$.
Let $\P' = \P\setminus\{p_s\} = \{p_1,\dots,p_{s-1}\}$.
By the induction hypothesis, we have that 
\begin{equation}\label{weird eq 1}
    \E_{n\in\N}\overline{f(n)}S_nF
    = \sum_{k_1,\dots,k_{s-1}\in\N\cup\{0\}}\prod_{i=1}^{s-1}
    \frac{(p_i-1)\overline{f(p_i)^{k_i}}}{p_i^{k_i+1}} 
    \bigg(\prod_{i=1}^{s-1} S_{p_i^{k_i}}\bigg)
    \Big(\E_{n\in\Qf_{\P'}} \overline{f(n)}S_nF\Big).
\end{equation}
Now we notice that the set of $\P'$-free numbers can be expressed as 
$$\Qf_{\P'} 
= \bigsqcup_{k_s\in\N\cup\{0\}} p_s^{k_s}\Qf_\P.$$
Therefore, as we did in the base case, we have 
\begin{align}\label{weird eq 2}
    \E_{n\in\Qf_{\P'}} \overline{f(n)}S_nF 
    & = \sum_{k_s\in\N\cup\{0\}}\E_{n\in\Qf_{\P'}}
    \1_{p_s^{k_s}\Qf_\P}(n)\overline{f(n)}S_nF \notag \\
    & = \lim_{N\to\infty}\sum_{k_s\in\N\cup\{0\}}\overline{f(p_s)^{k_s}}S_{p_s^{k_s}}
    \bigg(\frac{1}{|\Qf_{\P'}\cap[1,N]|} 
    \sum_{n\in\Qf_{\P'}\cap[1,N/{p_s^{k_s}}]}
    \1_{\Qf_\P}(n)\overline{f(n)}S_nF\bigg) \notag \\
    & = \lim_{N\to\infty}\sum_{k_s\in\N\cup\{0\}}\frac{\overline{f(p_s)^{k_s}}}{p_s^{k_s}}S_{p_s^{k_s}}
    \Big(\E_{n\in\Qf_{\P'}\cap[1,N/p_s^{k_s}]}
    \1_{\Qf_\P}(n)\overline{f(n)}S_nF\Big) \notag \\
    & = \sum_{k_s\in\N\cup\{0\}}\frac{\overline{f(p_s)^{k_s}}}{p_s^{k_s}}S_{p_s^{k_s}}
    \Big(\E_{n\in\Qf_{\P'}}\1_{\Qf_\P}(n)
    \overline{f(n)}S_nF\Big) \notag \\
    & = \frac{d(\Qf_\P)}{d(\Qf_{\P'})}\sum_{k_s\in\N\cup\{0\}}
    \frac{\overline{f(p_s)^{k_s}}}{p_s^{k_s}}S_{p_s^{k_s}}
    \Big(\E_{n\in\Qf_\P}\overline{f(n)}S_nF \Big) \notag \\
    & = \sum_{k_s\in\N\cup\{0\}}\frac{(p_s-1)\overline{f(p_s)^{k_s}}}{p_s^{k_s+1}}
    S_{p_s^{k_s}}\Big(\E_{n\in\Qf_\P}\overline{f(n)}S_nF \Big),
\end{align}
where all the limits are taken in $L^2(X)$. Substituting \eqref{weird eq 2} 
to \eqref{weird eq 1}, we obtain
\begin{align*}
    & \E_{n\in\N} \overline{f(n)}S_nF \\
    & = \sum_{k_1,\dots,k_{s-1}\in\N\cup\{0\}}\prod_{i=1}^{s-1}
    \frac{(p_i-1) \overline{f(p_i)^{k_i}}}{p_i^{k_i+1}}
    \bigg(\prod_{i=1}^{s-1} S_{p_i^{k_i}}\bigg)
    \bigg(\sum_{k_s\in\N\cup\{0\}}\frac{p_s-1}{p_s^{k_s+1}}
    \overline{f(p_s)^{k_s}}S_{p_s^{k_s}}
    \Big(\E_{n\in\Qf_\P}\overline{f(n)}S_nF \Big)\bigg) \\
    & = \sum_{k_1,\dots,k_s\in\N\cup\{0\}}\prod_{i=1}^s
    \frac{(p_i-1)\overline{f(p_i)^{k_i}}}{p_i^{k_i+1}}
    \bigg(\prod_{i=1}^s S_{p_i^{k_i}}\bigg)
    \Big(\E_{n\in\Qf_\P}\overline{f(n)}S_nF \Big).
\end{align*}
This completes the induction and concludes the proof of the claim.
\end{proofclaim1}
\underline{Claim 2}. For any function $G\in L^2(X)$ satisfying 
$$\sum_{k_1,\dots,k_s\in\N\cup\{0\}}\prod_{i=1}^s
\frac{\overline{f(p_i)^{k_i}}}{p_i^{k_i+1}}
\bigg(\prod_{i=1}^s S_{p_i^{k_i}}\bigg)G = 0,$$
it holds that $G=0$.
\begin{proofclaim2}We also prove this claim by induction on $s\in\N$.
Let $s=1$ and then we write $\P=\{p\}$. Then for any $G\in L^2(X)$, we have that
$$0 = \sum_{k\in\N\cup\{0\}}\frac{\overline{f(p)^k}}{p^{k+1}}S_{p^k}G 
= \frac{1}{p}S_1G + \sum_{k\in\N}\frac{\overline{f(p)^k}}{p^{k+1}}S_{p^k}G
= \frac{1}{p}S_1G 
+ \frac{1}{p}S_1\bigg(\sum_{k\in\N\cup\{0\}}\frac{\overline{f(p)^k}}{p^{k+1}}S_{p^k}G\bigg)
= \frac{1}{p}S_1G$$
and therefore, $G=0$. This proves the base case of the induction.
Now let $s\geq 2$ and suppose that the statement is true for all positive integers 
smaller than $s$. Let $G\in L^2(X)$ satisfying the assumption of the claim.
Then we have that
\begin{align*}
    0 & = \sum_{k_1,\dots,k_s\in\N\cup\{0\}}\prod_{i=1}^s\frac{\overline{f(p_i)^{k_i}}}{p_i^{k_i+1}}
    \bigg(\prod_{i=1}^s S_{p_i^{k_i}}\bigg)G \\
    &  = \sum_{k_1,\dots,k_{s-1}\in\N\cup\{0\}}\prod_{i=1}^{s-1}
    \frac{\overline{f(p_i)^{k_i}}}{p_i^{k_i+1}}\bigg(\prod_{i=1}^{s-1} S_{p_i^{k_i}}\bigg) 
    \bigg(\sum_{k_s\in\N\cup\{0\}}\frac{\overline{f(p_s)^{k_s}}}{p_s^{k_s+1}}S_{p_s^{k_s}}G \bigg).
\end{align*}
Then, by induction hypothesis, it follows that the function
$\widetilde{G} 
:= \sum_{k_s\in\N\cup\{0\}}\frac{\overline{f(p_s)^{k_s}}}{p_s^{k_s+1}}S_{p_s^{k_s}}G
\in L^2(X)$ satisfies $\widetilde{G}=0$. Then, by the base case, this implies that
$G=0$, concluding the induction, and hence, the proof of the claim.
\end{proofclaim2}
Since $F\in\V_f$, it follows by Claim 1 that 
$$\sum_{k_1,\dots,k_s\in\N\cup\{0\}}\prod_{i=1}^s
\frac{(p_i-1)\overline{f(p_i)^{k_i}}}{p_i^{k_i+1}}
\bigg(\prod_{i=1}^s S_{p_i^{k_i}}\bigg)
\Big(\E_{n\in\Qf_\P} \overline{f(n)}S_nF\Big) = 0$$
and so, 
$$\sum_{k_1,\dots,k_s\in\N\cup\{0\}}\prod_{i=1}^s
\frac{\overline{f(p_i)^{k_i}}}{p_i^{k_i+1}}
\bigg(\prod_{i=1}^s S_{p_i^{k_i}}\bigg)
\Big(\E_{n\in\Qf_\P} \overline{f(n)}S_nF\Big) = 0.$$
Then by applying Claim 2, for the function 
$G = \E_{n\in\Qf_\P}\overline{f(n)}S_nF$, we conclude that 
$\E_{n\in\Qf_\P}\overline{f(n)}S_nF = 0$.
This concludes the proof of the proposition.
\end{proof}

\begin{Lemma}\label{Q_P approximation lemma}
Let $\P\subset\Pri$ containing few primes.
Then for any $\varepsilon>0$, there exists a finite $\P_\varepsilon\subset\P$ 
such that $d(\Qf_{\P_\varepsilon}\setminus\Qf_\P)<\varepsilon$.
\end{Lemma}
\begin{proof}
Let $\epsilon>0$ and $\delta=1-(\frac{\epsilon}{d(\Qf_\P)}+1)^{-1}>0$.
Since $\sum_{p\in\P}\frac{1}{p}<\infty$, there exists 
some $N_\varepsilon\in\N$ such that 
$$\sum_{\substack{p\in\P \\ p>N_\epsilon}}\frac{1}{p} < \delta.$$
Let $\P_\epsilon = \{p\in\P \colon p\leq N_\varepsilon\}$ which is a finite
subset of $\P$. Then we have 
\begin{align*}
    d(\Qf_{\P_\epsilon}\setminus\Qf_\P)
    & = d(\Qf_{\P_\epsilon}) - d(\Qf_\P)
    = d(\Qf_\P)\bigg(\frac{d(\Qf_{\P_\epsilon})}{d(\Qf_\P)}-1\bigg)
    = d(\Qf_\P)\bigg(\prod_{p\in\P\setminus\P_\epsilon}
    \Big(1-\frac{1}{p}\Big)^{-1}-1\bigg)\\
    & \leq d(\Qf_\P)\bigg(\bigg(1-\sum_{p\in\P\setminus\P_\epsilon}\frac{1}{p}\bigg)^{-1}-1\bigg)
    < d(\Qf_\P)\big((1-\delta)^{-1}-1\big) = \epsilon.
\end{align*}
This concludes the proof.
\end{proof}

\begin{Corollary}\label{Q_P convergence corollary}
Let $f\in\Mult^\textup{c}_\textup{fg}$, $F\in\V_f$ and $\P\subset\Pri$ containing few primes.
Then 
$$\E_{n\in\Qf_\P}\overline{f(n)}S_nF = 0.$$
\end{Corollary}
\begin{proof}
Let $\varepsilon>0$. Since $\sum_{p\in\P}\frac{1}{p}<\infty$, then 
$d(\Qf_\P)>0$. Consider a finite set $\P_\varepsilon\subset\P$ such that
$d(\Qf_{\P_\varepsilon}\setminus\Qf_\P)
<\frac{d(\Qf_\P)\varepsilon}{4},$
as guaranteed by \cref{Q_P approximation lemma}. Now by 
\cref{finite Q_P convergence lemma}, for any $N$ large (depending on 
$\varepsilon$), we have 
$$\Big\|\E_{n\in\Qf_{\P_\varepsilon}\cap[1,N]}
\overline{f(n)}S_nF \Big\|_2
< \frac{\varepsilon}{2}.$$
Then for any $N$ large, we have 
\begin{align*}
    \Big\|\E_{n\in\Qf_\P\cap[1,N]} 
    \overline{f(n)}S_nF \Big\|_2 &
    \leq \frac{1}{|\Qf_\P\cap[1,N]|}\bigg\|
    \sum_{n\in\Qf_{\P_\varepsilon}\cap[1,N]}
    \overline{f(n)}S_nF \Big\|_2 \\
    & \hspace*{0.4cm} + \frac{1}{|\Qf_\P\cap[1,N]|}\bigg\|
    \sum_{n\in(\Qf_\P\setminus\Qf_{\P_\varepsilon})\cap[1,N]} \overline{f(n)}S_nF \bigg\|_2 \\
    & \leq \frac{|\Qf_{\P_\varepsilon}\cap[1,N]|}{|\Qf_\P\cap[1,N]|}\Big\|
    \E_{n\in\Qf_{\P_\varepsilon}\cap[1,N]}\overline{f(n)}S_nF \Big\|_2 + 
    \frac{|(\Qf_{\P_\varepsilon}\setminus\Qf_\P)\cap[1,N]|}{|\Qf_\P\cap[1,N]|} \\
    & \leq \Big\|\E_{n\in\Qf_{\P_\varepsilon}\cap[1,N]}\overline{f(n)}S_nF \Big\|_2 
    + 2\frac{|(\Qf_{\P_\varepsilon}\setminus\Qf_\P)\cap[1,N]|}{|\Qf_\P\cap[1,N]|}.
\end{align*}
Therefore, 
$$\limsup_{N\to\infty}
\Big\|\E_{n\in\Qf_\P\cap[1,N]} S_nF \Big\|_2 
< \frac{\varepsilon}{2} + 
2\frac{d(\Qf_{\P_\varepsilon}\setminus\Qf_\P)}
{d(\Qf_\P)}
< \varepsilon.$$
Since $\varepsilon$ was arbitrary, the result follows.
\end{proof}

Now we are ready to prove \cref{orthogonality prop 1}, which will follow quite 
easily from \cref{Q_P convergence corollary}.

\begin{proof}[Proof of \cref{orthogonality prop 1}]
Let $f\in\Mult^\textup{c}_\textup{fg}$, and let $F\in\Hilb_f$ and $G\in\V_f$.
Then the set $\P:=\{p\in\Pri\colon S_pF\neq f(p)F\}$ satisfies
$\sum_{p\in\P}\frac{1}{p}<\infty$. It follows by \cref{Q_P convergence corollary}
that $\E_{n\in\Qf_\P}\overline{f(n)}S_nG = 0$. So for any $N\in\N$, we have
$$\langle F,G \rangle 
=\E_{n\in\Qf_\P\cap[1,N]}\langle S_nF,S_nG \rangle
=\E_{n\in\Qf_\P\cap[1,N]}\langle f(n)F,S_nG \rangle
= \Big\langle F, \E_{n\in\Qf_\P\cap[1,N]}
\overline{f(n)}S_nG \Big\rangle.$$
By Cauchy-Schwarz inequality, for any $N\in\N$, we have that
$$|\langle F,G \rangle| 
\leq \|F\|_2\Big\|\E_{n\in\Qf_\P\cap[1,N]}
\overline{f(n)}S_nG\Big\|_2.$$ Letting $N\to\infty$, yields that 
$\langle F, G \rangle = 0.$ This proves that $\Hilb_f\perp\V_f$.
\end{proof}

We continue by proving \cref{orthogonality prop 2}.
\begin{proof}[Proof of \cref{orthogonality prop 2}]
Let $F\in L^2(X)$ and $f\in\Mult^\textup{c}_\textup{fg}$. Suppose that
$$\lim_{N\to\infty}\bigg\|\frac{1}{N}\sum_{n=1}^N\overline{f(n)}S_nF\bigg\|_2>0.$$
Our goal is to show that $F\not\perp\Hilb_f$. By \cref{spectral thm}, it follows from the assumption that
$$\lim_{\N\to\infty} \int_{\Mult^\text{c}_\text{fg}} \bigg|\frac{1}{N}\sum_{n=1}^N
g(n)\overline{f(n)}\bigg|^2 \d\mu_F(g) > 0,$$
where $\mu_F$ is the spectral measure of $F$.
It follows by \cref{halasz mvt} that
$$\mu_F(\A_f\cap\Mult^\text{c}_\text{fg}) = \mu_F(\{g\in\Mult^\text{c}_\text{fg}\colon \D(g,f)<\infty\}) > 0.$$
Now we define the non-zero function $\delta_f\in L^2(\Mult^\text{c}_\text{fg},\mu_F)$ by
$$\delta_f(g) = 
\begin{cases}
    1, & \text{if } g\in\A_f\cap\Mult^\text{c}_\text{fg}, \\
    0, & \text{otherwise},
\end{cases}$$
which satisfies $\int_{\Mult^\text{c}_\text{fg}}\delta_f \d\mu_F > 0$.
Consider the unitary isomorphism $\Phi$ from $L^2(\Mult^\text{c}_\text{fg},\mu_F)$ to the cyclic 
sub-representation of $L^2(X)$ generated by $F$, as guaranteed by \cref{spectral thm}. This isomorphism conjugates $S_n$ to multiplication by $e_n$ 
and also maps $1$ to $F$.
Let $G = \Phi(\delta_f)\in L^2(X)$. Then we have
$$\langle F,G \rangle = \langle 1,\delta_f \rangle_{L^2(\Mult^\text{c}_\text{fg},\mu_F)} 
= \int_{\Mult^\text{c}_\text{fg}} \delta_f \d\mu_F >0.$$
By \cref{spectral thm} once again, we have that,
for any $n\in\N$,
\begin{equation*}
\int_{\Mult^\text{c}_\text{fg}} g(n) \d\mu_G(g) 
= \int_X S_nG\cdot\overline{G}\d\mu 
= \int_{\Mult^\text{c}_\text{fg}} g(n)\delta_f(g)\overline{\delta_f(g)}\d\mu_F(g)
= \int_{\Mult^\text{c}_\text{fg}} g(n) \d(\delta_f\mu_F)(g).
\end{equation*}
By uniqueness of the spectral measure, it follows that 
$\mu_G = \delta_f\mu_F$, which implies that $\mu_G$ is supported on $\A_f\cap\Mult^\text{c}_\text{fg}$.
It follows by \cref{relation between the two distances} that $G\in\overline{\Hilb_f}$,
and since $F$ and $G$ are not orthogonal, we obtain that $F\not\perp\Hilb_f$.
This concludes the proof.
\end{proof}

Recall that our goal is to prove \cref{weighted PMET}. The following is an 
obvious corollary of \cref{orthogonality prop 2}, which will be useful to this end.

\begin{Corollary}\label{almost PMET}
For any $F\in L^2(X)$ and any $f\in\Mult^\textup{c}_\textup{fg}$, we have
$$\lim_{N\to\infty}\bigg\|\frac{1}{N}\sum_{n=1}^N \overline{f(n)}S_nF
- \frac{1}{N}\sum_{n=1}^N \overline{f(n)}S_n(P_fF)\bigg\|_2
= 0.$$
\end{Corollary}

In view of \cref{almost PMET}, we see that in order to conclude the proof of 
\cref{weighted PMET}, we need to calculate the ergodic averages
$\frac{1}{N}\sum_{n=1}^N \overline{f(n)}S_nF$, for $F\in\overline{\Hilb_f}$.
In the next proposition, we do it for $F\in\Hilb_f$.

\begin{Proposition}\label{ergodic averages for pret eigenf}
Let $f\in\Mult^\textup{c}_\textup{fg}$ and $F\in\Hilb_f$. Then 
$$\lim_{N\to\infty}\frac{1}{N}\sum_{n=1}^N \overline{f(n)}S_nF
= \prod_{p\in\P}\Big(1-\frac{1}{p}\Big)
\sum_{n\in\Qf_{\P^\textup{c}}} \frac{\overline{f(n)}}{n}S_nF \qquad\text{in}~L^2(X),$$
where $\P=\{p\in\Pri\colon S_pF \neq f(p)F\}$.
\end{Proposition}

\begin{proof}
We begin by defining a completely multiplicative sequence $(G_n)_{n\in\N}$ of $L^2(X)$ 
functions and a finitely generated completely multiplicative function $h$ given by
$$G_p 
= 
\begin{cases}
S_pF, & p\in\P \\
0, & p\in\P^\textup{c}
\end{cases},
\qquad\text{ and }\qquad
h(p)
= 
\begin{cases}
0, & p\in\P \\
f(p), & p\in\P^\textup{c}
\end{cases} $$
respectively. Then we observe that 
$$S_nF = \sum_{ab=n}G_a h(b).$$
Therefore, we have
\begin{align}\label{conv eq 1}
    \frac{1}{N}\sum_{n=1}^N \overline{f(n)}S_nF 
    & = \frac{1}{N}\sum_{n=1}^N \overline{f(n)}\sum_{ab=n}G_a h(b)
    = \frac{1}{N}\sum_{a=1}^N\overline{f(a)}G_a\sum_{b\leq N/a}\overline{f(b)}h(b) \notag \\
    & = \sum_{n=1}^N \frac{\overline{f(n)}}{n}G_n 
    \Big(\E_{m\leq N/n}\overline{f(m)}h(m)\Big) + \oh(1).
\end{align}
We set $K(M) = \E_{m\leq M}\overline{f(m)}h(m)$ for any $M>0$.
For any positive integers $M<N$, we have
\begin{align*}
    & \bigg\| \sum_{n=1}^N\frac{\overline{f(n)}}{n}G_n K\Big(\frac{N}{n}\Big)
    - K(N)\sum_{n=1}^N\frac{\overline{f(n)}}{n}G_n\bigg\|_2
    = \bigg\|\sum_{n=1}^N\frac{\overline{f(n)}}{n}G_n
    \Big(K\Big(\frac{N}{n}\Big)-K(N)\Big)\bigg\|_2 \\
    & \leq \bigg\|\sum_{n=1}^M\frac{\overline{f(n)}}{n}G_n
    \Big(K\Big(\frac{N}{n}\Big)-K(N)\Big)\bigg\|_2
    + \sum_{n=M+1}^N\Big\|K\Big(\frac{N}{n}\Big)-K(N)\Big\|_\infty
    \frac{|\overline{f(n)}|}{n}\|G_n\|_2 \\
    & \leq \sum_{n=1}^M\frac{1}{n}\Big\|K\Big(\frac{N}{n}\Big)-K(N)\Big\|_2
    + 2\sum_{n=M+1}^N\frac{\|G_n\|_2}{n}.
\end{align*}
Choosing $M<N$ so that $N/M\to\infty$, it follows that
$$\limsup_{M\to\infty}\limsup_{N\to\infty}
\bigg\|\sum_{n=1}^N\frac{\overline{f(n)}}{n}G_n K\Big(\frac{N}{n}\Big)
- K(N)\sum_{n=1}^N\frac{\overline{f(n)}}{n}G_n\bigg\|_2
\leq 2\limsup_{M\to\infty}\sum_{n>M}\frac{\|G_n\|_2}{n} = 0,$$
since 
$\sum_{n\in\N}\frac{\|G_n\|_2}{n} 
= \|F\|_2\sum_{n\in\Qf_{\P^\textup{c}}}\frac{1}{n}
= \|F\|_2\prod_{p\in\P}\big(1-\frac{1}{p}\big)^{-1}
\leq \|F\|_2\exp\big(\sum_{p\in\P}\frac{1}{p}\big)
< \infty$.
Thus, for large $N$, we can rewrite \eqref{conv eq 1} as
\begin{align*}
    &\frac{1}{N}\sum_{n=1}^N\overline{f(n)}S_nF 
    = \bigg(\frac{1}{N}\sum_{n=1}^N\overline{f(n)}h(n)\bigg) 
    \bigg(\sum_{n=1}^N\frac{\overline{f(n)}}{n}G_n\bigg) + \oh(1) \\
    & = \bigg(\frac{1}{N}\sum_{\Qf_\P\cap[1,N]}1\bigg)
    \bigg(\sum_{n\in\Qf_{\P^\textup{c}}\cap[1,N]}
    \frac{\overline{f(n)}}{n}S_nF\bigg) + \oh(1) = \frac{|\Qf_\P\cap[1,N]|}{N}
    \sum_{n\in\Qf_{\P^\textup{c}}\cap[1,N]}
    \frac{\overline{f(n)}}{n}S_nF + \oh(1)
\end{align*}
where equalities are taken in $L^2(X)$. Therefore,
$$\lim_{N\to\infty}\frac{1}{N}\sum_{n=1}^N\overline{f(n)}S_nF 
= d(\Qf_\P)\sum_{n\in\Qf_{\P^\textup{c}}}
\frac{\overline{f(n)}}{n}S_nF
= \prod_{p\in\P}\Big(1-\frac{1}{p}\Big)
\sum_{n\in\Qf_{\P^\textup{c}}}\frac{\overline{f(n)}}{n}S_nF \qquad\text{in}~L^2(X).$$
This concludes the proof.
\end{proof}

We are now ready to conclude the proof of \cref{weighted PMET}.
\begin{proof}[Proof of \cref{weighted PMET}]
In view of \cref{almost PMET}, it suffices to show that the convergence in
\cref{ergodic averages for pret eigenf} holds for functions in $\overline{\Hilb_f}$.
\par
Let $F\in\overline{\Hilb_f}$ and $\epsilon>0$. There exists some $G\in\Hilb_f$
such that $\|F-G\|_2<\frac{\epsilon}{2}$ and let $\P=\{p\in\Pri\colon S_pG\neq f(p)G\}$.
By \cref{ergodic averages for pret eigenf}, we have
$$\lim_{N\to\infty}\bigg\|\frac{1}{N}\sum_{n=1}^N \overline{f(n)}S_nG
- \prod_{p\in\P\cap[1,N]}\Big(1-\frac{1}{p}\Big)
\sum_{n\in\Qf_{\P^\textup{c}}\cap[1,N]}
\frac{\overline{f(n)}}{n}S_nG\bigg\|_2 = 0.$$
Then by triangle inequality we have that
\begin{align*}
    & \bigg\|\frac{1}{N}\sum_{n=1}^N \overline{f(n)}S_nF
    - \prod_{p\in\P\cap[1,N]}\Big(1-\frac{1}{p}\Big)
    \sum_{n\in\Qf_{\P^\textup{c}}\cap[1,N]}
    \frac{\overline{f(n)}}{n}S_nF\bigg\|_2 \\
    & \leq \bigg\|\frac{1}{N}\sum_{n=1}^N \overline{f(n)}S_nF
    - \frac{1}{N}\sum_{n=1}^N \overline{f(n)}S_nG\bigg\|_2 \\
    & \hspace*{0.4cm} + \bigg\|\frac{1}{N}\sum_{n=1}^N \overline{f(n)}S_nG
    - \prod_{p\in\P\cap[1,N]}\Big(1-\frac{1}{p}\Big)
    \sum_{n\in\Qf_{\P^\textup{c}}\cap[1,N]}
    \frac{\overline{f(n)}}{n}S_nG\bigg\|_2 \\
    & \hspace*{0.4cm} + \bigg\|\prod_{p\in\P\cap[1,N]}\Big(1-\frac{1}{p}\Big)
    \sum_{n\in\Qf_{\P^\textup{c}}\cap[1,N]}
    \frac{\overline{f(n)}}{n}S_nG
    - \prod_{p\in\P\cap[1,N]}\Big(1-\frac{1}{p}\Big)
    \sum_{n\in\Qf_{\P^\textup{c}}\cap[1,N]}
    \frac{\overline{f(n)}}{n}S_nF\bigg\|_2 \\
    & \leq \|F-G\|_2 
    + \bigg\|\frac{1}{N}\sum_{n=1}^N \overline{f(n)}S_nG
    - \prod_{p\in\P\cap[1,N]}\Big(1-\frac{1}{p}\Big)
    \sum_{n\in\Qf_{\P^\textup{c}}\cap[1,N]}
    \frac{\overline{f(n)}}{n}S_nG\bigg\|_2 \\
    & \hspace*{0.4cm} + \prod_{p\in\P\cap[1,N]}\Big(1-\frac{1}{p}\Big)
    \|F-G\|_2 \sum_{n\in\Qf_{\P^\textup{c}}\cap[1,N]}\frac{1}{n} \\
    & < \frac{\epsilon}{2}
    + \bigg\|\frac{1}{N}\sum_{n=1}^N \overline{f(n)}S_nG
    - \prod_{p\in\P\cap[1,N]}\Big(1-\frac{1}{p}\Big)
    \sum_{n\in\Qf_{\P^\textup{c}}\cap[1,N]}
    \frac{\overline{f(n)}}{n}S_nG\bigg\|_2 \\
    & \hspace*{0.4cm} + \frac{\epsilon}{2}\prod_{p\in\P\cap[1,N]}\Big(1-\frac{1}{p}\Big)
    \sum_{n\in\Qf_{\P^\textup{c}}}\frac{1}{n}.
\end{align*}
By sending $N\to\infty$, and using Euler products, we obtain
\begin{align*}
    & \lim_{N\to\infty}\bigg\|\frac{1}{N}\sum_{n=1}^N \overline{f(n)}S_nF
    - \prod_{p\in\P\cap[1,N]}\Big(1-\frac{1}{p}\Big)
    \sum_{n\in\Qf_{\P^\textup{c}}\cap[1,N]}
    \frac{\overline{f(n)}}{n}S_nF\bigg\|_2 \\
    & \leq \frac{\epsilon}{2} 
    + \frac{\epsilon}{2}\prod_{p\in\P}\Big(1-\frac{1}{p}\Big)
    \sum_{n\in\Qf_{\P^\textup{c}}}\frac{1}{n} 
    = \frac{\epsilon}{2} 
    + \frac{\epsilon}{2}\prod_{p\in\P}\Big(1-\frac{1}{p}\Big)
    \prod_{p\in\P}\Big(1-\frac{1}{p}\Big)^{-1} 
    = \epsilon.
\end{align*}
Since $\epsilon$ was arbitrary, the result follows.
\end{proof}

\subsection{Proof of the corollaries of \ref{PMET}}

\begin{proof}[Proof of \cref{Halasz gen}]
(i) Let $F\in L^2(X)$ and we may assume that $\int_X F\d\mu = 0$.
Assume that $P_fF$ is constant. By \cref{constants}, we have that
$$P_fF
= \int_X P_fF \d\mu
=
\begin{cases}
    \int_X F\d\mu, & \text{if}~\D(f,1)<\infty, \\
    0, &\text{otherwise}
\end{cases}
= 0.$$ 
It follows by \cref{weighted PMET} that
$$\lim_{N\to\infty}\bigg\|\frac{1}{N}\sum_{n=1}^N\overline{f(n)}S_nF\bigg\|_2 = 0$$
and the proof of (i) is complete. \\
(ii) The first part of the statement, that is, any non-constant $F\in L^2(X)$ with $\D_F(S,f)<\infty$ satisfies \eqref{weighted ergodic averages 2}, immediately follows from \cref{weighted PMET}. Hence, we only have to prove the second statement, that is, for any non-constant $F\in L^2(X)$ with $\D_F(S,f)<\infty$, we have
\begin{equation}\label{weighted PMET pf eq1}
    \lim_{N\to\infty}\bigg\|\frac{1}{N}\sum_{n=1}^N\overline{f(n)}S_nF
- M(\overline{f})\cdot\int_X F\d\mu\bigg\|_2>0.
\end{equation}
Assuming that we have proved \eqref{weighted PMET pf eq1} for zero integral functions,
if $F$ is a non-constant $L^2(X)$ function satisfying $\D_F(S,f)<\infty$, then
we define the non-constant, zero integral function $G=F-\int_X F\d\mu$.
This function satisfies $\D_G(S,f)<\infty$ and so we have 
$$\lim_{N\to\infty}\bigg\|\frac{1}{N}\sum_{n=1}^N\overline{f(n)}S_nF
- M(\overline{f})\cdot\int_X F\d\mu\bigg\|_2
= \lim_{N\to\infty}\bigg\|\frac{1}{N}\sum_{n=1}^N\overline{f(n)}S_nG\bigg\|_2
> 0.$$
Therefore it suffices to show \eqref{weighted PMET pf eq1} for zero integral functions. Let $F\in L^2(X)$ 
be such a function and suppose that $\D_F(S,f)<\infty$. It follows from 
\cref{relation between the two distances} that $\mu_F$ is supported on $\A_f\cap\Mult^\text{c}_\text{fg}$.
Then by \cref{spectral thm} and the dominated convergence theorem,
we have that
\begin{align}\label{Halasz pf eq 1}
\lim_{N\to\infty}\bigg\|\frac{1}{N}\sum_{n=1}^N\overline{f(n)}S_nF\bigg\|_2^2
& = \lim_{N\to\infty} \int_{\Mult^\text{c}_\text{fg}}
\bigg|\frac{1}{N}\sum_{n=1}^N \overline{f(n)}g(n)\bigg|^2 \d\mu_F(g)
= \int_{\Mult^\text{c}_\text{fg}} |M(\overline{f}g)|^2 \d\mu_F(g) \notag \\
& = \int_{\Mult^\text{c}_\text{fg}} |M(g)|^2 \d\nu_F(g),
\end{align}
where $\nu_F = \overline{f}\mu_F$ and is supported on $\A_1\cap\Mult^\text{c}_\text{fg}$.
\par
Let $g\in\A_1$ be finitely generated. We define the function 
$$h(p) := 1-\Big(1-\frac{1}{p}\Big)\Big(1-\frac{g(p)}{p}\Big)^{-1}
= \frac{1-g(p)}{p-g(p)}
= \frac{p-pg(p)-\overline{g(p)}+1}{|p-g(p)|^2},$$ 
and by \cref{halasz mvt}, we have
\begin{equation}\label{Halasz pf eq 2}
    |M(g)|^2 
    \gg \bigg|\prod_{p>5} \Big(1-\frac{1}{p}\Big)\Big(1-\frac{g(p)}{p}\Big)^{-1}\bigg|^2
    = \bigg|\prod_{p>5} (1-h(p))\bigg|^2 
    = \exp\bigg(2\sum_{p>5} \log|1-h(p)|\bigg).
\end{equation}
Now, we have that
\begin{equation}\label{real part of h}
    \Re(h(p)) 
    = \frac{(p+1)(1-\Re(g(p)))}{|p-g(p)|^2}
    \leq 2\frac{1-\Re(g(p))}{|p-g(p)|},
\end{equation}
and then clearly $\Re(h(p))\in[0,1)$. Combining \eqref{Halasz pf eq 2} and 
\eqref{real part of h} and using the inequality $\log(1-x)\geq -\frac{x}{1-x}$ for $x\in[0,1)$
and the fact that $x\mapsto -\frac{x}{1-x}$ is decreasing in $[0,1)$, it follows that
\begin{align}\label{Halasz pf eq 3}
|M(g)|^2 
& \gg \exp\bigg(2\sum_{p>5} \log(1-\Re(h(p)))\bigg)
\geq \exp\bigg(-2\sum_{p>5} \frac{\Re(h(p))}{1-\Re(h(p))}\bigg) \notag \\
& \geq \exp\bigg(-4\sum_{p>5} \frac{1-\Re(g(p))}{|p-g(p)|-2(1-\Re(g(p)))}\bigg)
> \exp\bigg(-16\sum_{p>5} \frac{1-\Re(g(p))}{p}\bigg) \notag \\
& \geq \exp(-16\D(g,1)^2),
\end{align}
where the second last inequality follows from that 
$|p-g(p)|-2(1-\Re(g(p)))\geq p-5>\frac{p}{4}$ for all primes $p>5$.
It follows by \eqref{Halasz pf eq 1} and \eqref{Halasz pf eq 3} that
$$\lim_{N\to\infty} \bigg\|\frac{1}{N}\sum_{n=1}^N \overline{f(n)}S_nF\bigg\|_2^2
\gg \int_{\Mult^\text{c}_\text{fg}} \exp(-16\D(g,1)^2) \d\nu_F(g).$$
Then, by Jensen's inequality for the convex function $e^{-x}$,
it follows that 
\begin{align*}
\lim_{N\to\infty} \bigg\|\frac{1}{N}\sum_{n=1}^N \overline{f(n)}S_nF\bigg\|_2^2
\gg \exp\bigg(-16 \int_{\Mult^\text{c}_\text{fg}}\D(g,1)^2 \d\nu_F(g)\bigg)
& = \exp\bigg(-16\int_{\Mult^\text{c}_\text{fg}}\D(g,f)^2\d\mu_F(g)\bigg) \\
& = \exp(-16\D_F(S,f)^2)
> 0.
\end{align*}
This concludes the proof.
\end{proof}

\begin{proof}[Proof of \cref{aperiodicity criterion}]
By \cref{aperiodic functions in systems equiv}, $S$ is aperiodic if and only if for any Dirichlet character $\chi$ and any $F\in L^2(X)$,
$$\lim_{N\to\infty}\frac{1}{N}\sum_{n=1}^N \chi(n)S_nF =
\begin{cases}
    \int_X F \d\mu, & \text{if}~\chi=\chi_0, \\
    0, & \text{otherwise}.
\end{cases}$$
Then, in view of \cref{change of function in averages}, \ref{PMET} and
\cref{pret ergodicity criterion}, it suffices to show that the following are
equivalent:
\begin{itemize}
    \item[(a)] For any $F\in L^2(X)$, 
    $\lim_{N\to\infty}\big\|\frac{1}{N}\sum_{n=1}^N \chi(n)S_nF\big\|_2 = 0$
    holds for all $\chi\neq\chi_0$.
    \item[(b)] $\Hilb_\chi =\{0\}$ for all $\chi\neq\chi_0$.
    \item[(c)] $\sigma_\text{pr.rat}(S) = \A_1$.
\end{itemize}
(a) is equivalent to that $\V_\chi=L^2(X)$ for all $\chi\neq\chi_0$, which, 
in view of \cref{small decomposition}, is equivalent to (b). 
Now (b) and (c) are equivalent by definition. This concludes the proof.
\end{proof}

\begin{proof}[Proof of \cref{pretentious wm criterion}]
At first, notice that (ii) and (iv) are obviously equivalent. Moreover, in view of 
\cref{pret ergodicity criterion}, the equivalence of (ii) and (iii) follows from the equivalence:
$$\Hilb_f=\{0\},~\forall~f\in\Mult^\text{c}_\text{fg}~\text{with}~\D(f,1)=\infty
\Longleftrightarrow
\sigma_\text{pr}(S)=\A_1,$$
which is true by definition.
It remains to show that (i) and (iv) are equivalent. \\
(i)~$\Longrightarrow$~(iv)
Suppose that $S$ is pretentiously weak-mixing, thus, $S\times S$ is pretentiously ergodic.
We let $F\in L^2(X)$ be a pretentious eigenfunction with corresponding pretentious eigenvalue 
$f\in\Mult^\text{c}_\text{fg}$ and we define the function $G\in L^2(X\times X)$ by
$G(x,y)=F(x)\overline{F(y)}$. Then, by the triangle inequality,
we have that
\begin{align*}
    & \|(S_p\times S_p)G - G\|_{L^2(X\times X)} \\
    & \leq \|S_pF\otimes S_p\overline{F} - (f(p)F)\otimes S_p\overline{F}\|_{L^2(X\times X)}
    + \|(f(p)F)\otimes S_p\overline{F} - G\|_{L^2(X\times X)} \\
    & = \|S_pF\otimes S_p\overline{F} - (f(p)F)\otimes S_p\overline{F}\|_{L^2(X\times X)}
    + \|(f(p)F)\otimes S_p\overline{F} - (f(p)F)\otimes(\overline{f(p)F})\|_{L^2(X\times X)} \\
    & = 2\|F\|_2\|S_pF - f(p)F\|_2,
\end{align*}
and then by \eqref{Dist comparison ineq}, it follows that
\begin{align*}
\D_G(S\times S, 1)^2 
\leq \sum_{p\in\Pri}\frac{\|(S_p\times S_p)G - G\|^2_{L^2(X\times X)}}{p}
\leq 4\|F\|^2_2 \sum_{p\in\Pri} \frac{\|S_pF - f(p)F\|^2_2}{p}
& \leq 8\|F\|^2_2 \D_F(S,f)^2 \\
& < \infty.
\end{align*}
This shows that $G$ is pretentiously invariant with respect to $S\times S$, and then
it follows by \cref{pret ergodicity criterion} that $G$ is constant and hence, $F$ is constant. \\
(iv)~$\Longrightarrow$~(i)
On the other hand, suppose that $S$ has no non-constant pretentious eigenfunctions. 
We want to show that for any $F\in L^2(X)$, we have 
$$\lim_{N\to\infty}\frac{1}{N}\sum_{n=1}^N\bigg|
\int_X S_nF\cdot\overline{F}\d\mu - \int_X F\d\mu \int_X \overline{F}\d\mu\bigg|
= 0.$$
which is an equivalent definition of $S$ being pretentiously weak-mixing.
It is enough to show this for non-constant functions with zero integral. 
Let $F\in L^2(X)$ be a non-constant function with $\int_X F\d\mu=0$.
Since the averages of a sequence converge to zero if and only if the 
averages of the square of the sequence converges to zero, it is enough to show
that
$$\lim_{N\to\infty}\frac{1}{N}\sum_{n=1}^N
\bigg|\int_X S_nF\cdot\overline{F}\d\mu\bigg|^2 = 0.$$
By \cref{spectral thm} and by expanding the square, we have 
\begin{align}\label{wm crit eq1}
    & \frac{1}{N}\sum_{n=1}^N\bigg|\int_X S_nF\cdot\overline{F}\d\mu\bigg|^2 
    = \frac{1}{N}\sum_{n=1}^N\bigg|\int_{\Mult^\text{c}_\text{fg}}f(n)\d\mu_F(f)\bigg|^2 \notag \\
    & =\frac{1}{N}\sum_{n=1}^N \bigg(\int_{\Mult^\text{c}_\text{fg}}f(n)\d\mu_F(f)\bigg)
    \bigg(\int_{\Mult^\text{c}_\text{fg}}\overline{g(n)}\d\mu_F(g)\bigg) 
    = \int_{\Mult^\text{c}_\text{fg}\times\Mult^\text{c}_\text{fg}}\frac{1}{N}\sum_{n=1}^N f(n)\overline{g(n)}
    \d(\mu_F\times\mu_F)(f,g).
\end{align}
We claim that for any $f\in\Mult^\text{c}_\text{fg}$ we have $\mu_F(\A_f)=0$. Suppose for 
contradiction that there exists some $f\in\Mult^\text{c}_\text{fg}$ such that $\mu_F(\A_f)>0$.
Then by \cref{spectral thm lemma}, there exists a function $G\in L^2(X)$
such that $\mu_G(\A_f\cap\Mult^\text{c}_\text{fg})=1$. It can be seen by the proof of \cref{spectral thm lemma},
that $G$ is non-constant since $F$ is non-constant. Since $\mu_G$ is supported on $\A_f$, then
it follows by \cref{relation between the two distances} that $G\in\overline{\Hilb_f}$,
but this contradicts the fact that $\overline{\Hilb_f}$ consists of constants.
Now we define $\Delta = \{(f,g)\in\Mult^\text{c}_\text{fg}\times\Mult^\text{c}_\text{fg}\colon \D(f,g)<\infty\}$.
Then, by Tonelli's theorem, we have that
$$(\mu_F\times\mu_F)(\Delta) = \int_{\Mult^\text{c}_\text{fg}} \mu_F(\A_f)\d\mu_F(f) = 0.$$
It follows by \cref{halasz mvt} that for any 
$(f,g)\in\text{supp}(\mu_F\times\mu_F)$, we have 
$M(f\overline{g}) = 0$. The result then follows by sending $N\to\infty$ in
\eqref{wm crit eq1} and using the dominated convergence theorem.
\end{proof}

\begin{proof}[Proof of \cref{class of pret erg/aper systems}]
Let $f\in\Mult^\text{c}_\text{fg}$ and consider $\xms$ to be the multiplicative rotation by $f$.\\
(i) In view of \cref{pret ergodicity criterion} it is enough to show that
$\Hilb_1$ consists of constants if and only if $\D(f^k,1)=\infty$ for any 
$k<|X|$. Recall that $X=\overline{f(\N)}$ and notice that this space is either $\Sone$
(if and only if $f(p)=e(\alpha)$ for $p\in\Pri$ and some irrational $\alpha$), 
or finite (if and only if for any $p\in\Pri$, $f(p)$ is some rational phase).
We treat each case separately:
\begin{itemize}
    \item[(a)] Suppose that $X=\Sone$. Then for any $k\in\N$, $f^k\neq1$. Moreover,
    the dual space of $X$ is $\widehat{X} = \{x^k\colon k\in\Z\} \simeq \Z$. Hence, $L^2(X)$
    functions can be written as 
    \begin{equation}\label{L2 function form}
        F=\sum_{k\in\Z} c_kx^k \qquad \text{in}~L^2(X),
    \end{equation}
    where $c_k\in\Sone$ for all $k\in\Z$.
    Suppose first that $\Hilb_1$ consists of constants. Then for any $k\in\Z\setminus\{0\}$,
    the $L^2(X)$ function $F_k(x)=x^k$ satisfies $S_pF_k\neq F_k$ for many primes.
    But $S_pF_k=f(p)^kF_k$, hence $f(p)^k\neq1$ for many primes.
    It follows that $\D(f^k,1)=\infty$, and this holds for all $k\in\Z\setminus\{0\}$.
    On the other hand, suppose that $\D(f^k,1)=\infty$ for any $k\in\N$. hence for any 
    $k\in\Z$ as well. Let $F\in L^2(X)$, as in \eqref{L2 function form}, be a pretentiously
    invariant function. We want to show that for any $k\in\Z\setminus\{0\}$, $c_k=0$.
    Fix $k\in\Z\setminus\{0\}$. By assumption, $f(p)^k\neq 1$ for many primes 
    and also $S_pF=F$ for almost every prime. 
    It follows that there exists some $p_0\in\Pri$ such that $f(p_0)^k\neq 1$ 
    and $S_{p_0}F=F$. Using \eqref{L2 function form},
    it is not hard to see that $S_{p_0}F=F$ implies that either $c_k=0$ or $f(p_0)^k=1$,
    but the latter case is excluded, therefore we have $c_k=0$. 
    This shows $\Hilb_1$ consists of constants, concluding the proof in the first case.
    \item[(b)] Suppose that $X$ is finite. Then $f^{|X|}=1$ and the dual space of $X$ 
    is $\widehat{X} = \{x^k \colon 0\leq k<|X|\}$,
    hence $L^2(X)$ functions are expressed in the form
    $$F = \sum_{0\leq k<|X|}c_kx^k \qquad \text{in}~L^2(X),$$
    where $c_k\in\Sone$ for all $0\leq k<|X|$. The proof now 
    is identical to the previous case.
\end{itemize}
The proof of the (i) is complete. \\
(ii) In view of \cref{aperiodicity criterion}, and by using (i), we have to show that
for any non-principal Dirichlet character $\chi$, $\Hilb_\chi=\{0\}$ is equivalent to 
$\D(f^k,\chi)=\infty$. The proof of this is identical to that of (i). \\
(iii) Clearly, the identity $F(x)=x$ is a non-constant pretentious eigenfunction,
so in view of \cref{pretentious wm criterion}, $S$ is not pretentiously weak-mixing.
\end{proof}

Before we move to the proof of the next result we need the following lemma:

\begin{Lemma}\label{eliminate irrationals}
Let $a:\N\to\Z$ be a finitely generated additive function. Then for any $\alpha\not\in\Q$ and 
for any Dirichlet character $\chi$, we have $\D(e(a(n)\alpha),\chi)=\infty$.
\end{Lemma}

\begin{proof}
The non-zero values of any Dirichlet character $\chi$ mod $q$ are the $\phi(q)$-roots of unity.
On the other hand, if $\alpha$ is an irrational number, then the function $e(a(n)\alpha)$
does not take values on the roots of unity at all. Hence $e(a(p)\alpha) \neq \chi(p)$ 
for almost every $p\in\Pri$ and in particular, $\D(e(a(n)\alpha),\chi)=\infty$.
\end{proof}

\begin{proof}[Proof of \cref{class of pret erg/aper systems 2}]
Let $a:\N\to\Z$ be a finitely generated completely additive function, 
$\xmt$ an additive system and consider the finitely generated
multiplicative system $(X,\mu,T^a)$. \\
(i) First, we observe that if $\Hilb_1$ consists of constants, 
then the space of $T$-invariant functions also does.
In view of \cref{pret ergodicity criterion}, this implies that pretentious ergodicity of $T^a$
implies ergodicity of $T$. Thus, assuming that $\xmt$ is ergodic, it suffices to show that $T^a$
is pretentiously ergodic if and only if $a$ satisfies:
\begin{itemize}
    \item for any $\frac{r}{q}\in\sigma_\text{rat}(T)\setminus\{0\}$ with $(r,q)=1$, $q\nmid a(p)$ holds for many primes, and
    \item $a(p)\neq 0$ for many primes.
\end{itemize}
Then, using \ref{PMET}, \cref{classic spectral thm},
\cref{halasz mvt} and \cref{eliminate irrationals}, we have 
\begin{align*}
    T^a \text{ is pretentiously ergodic } & 
    \Longleftrightarrow \forall~F\in L^2(X),~
    \lim_{N\to\infty}\frac{1}{N}\sum_{n=1}^N T^{a(n)}F = \int_X F \d\mu 
    \qquad\text{in}~L^2(X) \\
    & \Longleftrightarrow \forall~\alpha\in\sigma(T)\setminus\{0\},~
    \lim_{N\to\infty}\frac{1}{N}\sum_{n=1}^N e(a(n)\alpha) = 0 \\
    & \Longleftrightarrow \forall~\alpha\in\sigma(T)\setminus\{0\},~
    \D(e(a(n)\alpha),1) = \infty \\
    & \Longleftrightarrow \forall~\frac{r}{q}\in\sigma_\text{rat}(T)\setminus\{0\}~\text{with}~(r,q)=1,~\D(e(a(n)\frac{r}{q}),1)=\infty \\
    & \Longleftrightarrow \forall~\frac{r}{q}\in\sigma_\text{rat}(T)\setminus\{0\}~\text{with}~(r,q)=1,~\text{the set} \\
    & \hspace*{0.855cm} \P:=\{p\in\Pri\colon a(p)\frac{r}{q}\not\in\Z\}~\text{contains many primes},
\end{align*}
and since we can express 
$\P=\{p\in\Pri\colon a(p)\neq 0\}\cap\{p\in\Pri\colon q\nmid a(p)\}$,
the result follows. \\
(ii) Similarly, assuming that $(X,\mu,T)$ is ergodic, it suffices to show that 
$T^a$ is aperiodic if and only if $e(a(n)\alpha)$ is aperiodic for any $\alpha\in(0,1)\cap\Q$.
Then, using \cref{classic spectral thm}, \cref{halasz mvt},
\cref{aperiodic functions equiv} and \cref{eliminate irrationals},
we have:
\begin{align*}
    T^a \text{ is aperiodic } & 
    \Longleftrightarrow  \forall~F\in L^2(X),~\forall~r,q\in\N,
    \lim_{N\to\infty}\frac{1}{N}\sum_{n=1}^NT^{a(qn+r)}F = \int_X F \d\mu~
    \qquad\text{in}~L^2(X) \\
    & \Longleftrightarrow \forall~\alpha\in\sigma(T)\setminus\{0\},~
    \forall~r,q\in\N,~\lim_{N\to\infty}\frac{1}{N}\sum_{n=1}^N e(a(qn+r)\alpha) = 0 \\
    & \Longleftrightarrow \forall~\alpha\in\sigma(T)\setminus\{0\},~
    \forall~\chi,~
    \lim_{N\to\infty}\frac{1}{N}\sum_{n=1}^N \chi(n)e(a(n)\alpha) = 0 \\
    & \Longleftrightarrow \forall~\alpha\in\sigma(T)\setminus\{0\},~
    \forall~\chi,~\D(e(a(n)\alpha),\chi)=\infty \\
    & \Longleftrightarrow \forall~\alpha\in\sigma_\text{rat}(T)\setminus\{0\},~\forall~\chi,~
    \D(e(a(n)\alpha),\chi)=\infty,
\end{align*}
which, in view of \cref{halasz mvt}, is equivalent to that $e(a(n)\alpha)$ 
is aperiodic for any $\alpha\in(0,1)\cap\Q$. \\ 
(iii) Finally, we have that $T^a$ is pretentiously weak-mixing if and only if 
$T^a\times T^a=(T\times T)^a$ is pretentiously ergodic (by definition), which, in view of (i),
is equivalent to that $T\times T$ is ergodic and $a$ satisfies the condition in (i)
which eventually, is equivalent to that $T$ is weak-mixing and $a$ satisfies the condition in (i). This concludes the proof.
\end{proof}

\begin{proof}[Proof of \cref{Omega are good}]
(i) This follows from \cref{class of pret erg/aper systems 2} (i). \\
(ii) Suppose that $a(p)=1$ holds for almost every $p\in\Pri$.
Suppose for sake of contradiction that there exist a Dirichlet character $\chi$ and some 
$\frac{r}{q}\in(0,1)$ such that $\D(e(a(n)\frac{r}{q}),\chi)<\infty$. Hence, $e(a(p)\frac{r}{q})=\chi(p)$ for almost 
every $p\in\Pri$.
On the other hand, the assumption gives that $e(a(p)\frac{r}{q})=e(\frac{r}{q})$ for almost every $p\in\Pri$.
Combining the previous, we have that $\chi(p)$ is equal to the fixed number $e(\frac{r}{q})$ for almost 
every $p\in\Pri$, but since the non-zero values of any Dirichlet character mod $s$
are equidistributed in the $\phi(s)$ roots of unity, it follows that $\chi$ is principal.
Then, $e(\frac{r}{q})=1$, which is not possible since $\frac{r}{q}\in(0,1)$. This yields a contradiction. 
It follows by \cref{class of pret erg/aper systems 2} (ii) 
that $(X,\mu,T^a)$ is aperiodic for any ergodic $\xmt$. This concludes the proof.
\end{proof}

\subsection{Proof of the decomposition theorems}
Here we prove the two decompositions theorems, namely \cref{decomposition}
and \cref{decomposition2}. 
The proof of the first one will follow easily by 
\cref{small decomposition} and the following simple lemma.

\begin{Lemma}\label{lemma for decomp}
For any $\A\subset\Mult^\text{c}_\text{fg}$,
we have that 
$$\bigg(\textup{span}\bigg(\bigcup_{f\in\A}\Hilb_f\bigg)\bigg)^\perp
= \bigg(\bigcup_{f\in\A}\Hilb_f\bigg)^\perp.$$
\end{Lemma}

\begin{proof}
Let $F\in\big(\bigcup_{f\in\A}\Hilb_f\big)^\perp$ and 
$G\in\text{span}\big(\bigcup_{f\in\A}\Hilb_f\big)$. Then 
there exist $k\in\N$, $c_1,\dots,c_k\in\C$, $f_1,\dots,f_k\in\A$ 
and $G_i\in\Hilb_{f_i}$ for any $1\leq i\leq k$, such that
$$G=\sum_{i=1}^k c_iG_i.$$
It follows that
$$\langle F,G\rangle = \sum_{i=1}^k c_i\langle F,G_i\rangle = 0,$$
since by assumption, $\langle F,G_i\rangle = 0$ for any $1\leq i\leq k$.
This shows that 
$$\bigg(\bigcup_{f\in\A}\Hilb_f\bigg)^\perp
\subset\bigg(\text{span}\bigg(\bigcup_{f\in\A}\Hilb_f\bigg)\bigg)^\perp.$$
The other inclusion is obvious.
\end{proof}

\begin{proof}[Proof of \cref{decomposition}]
We want to show that $\Hilb_\text{aper} = \Hilb_\text{pr.rat}^\perp$.
First, we show that 
\begin{equation}\label{aperiodic component expression}
    \Hilb_\text{aper} = \bigcap_\chi \V_\chi,
\end{equation}
where the intersection is over all the Dirichlet characters $\chi$.
\par
By \cref{constants}, it follows that for any $F\in\V_1$, we have $\int_X F\d\mu=0$.
Clearly, $\Hilb_\text{aper}\subset\V_1$, thus, for any $F\in\Hilb_\text{aper}$,
we have $\int_X F \d\mu = 0$. Therefore, \eqref{aperiodic component expression}
follows from \cref{aperiodic functions in systems equiv}.
\par
Now, using \eqref{aperiodic component expression}, \cref{small decomposition} and
\cref{lemma for decomp} for $\A$ being the collection of all Dirichlet characters,
we have
$$\Hilb_{\text{aper}} 
= \bigcap_\chi \V_\chi
= \bigcap_\chi \Hilb_\chi^\perp
= \bigg(\bigcup_\chi \Hilb_\chi\bigg)^\perp
= \bigg(\text{span}\bigg(\bigcup_\chi \Hilb_\chi\bigg)\bigg)^\perp
= \bigg(\overline{\text{span}\bigg(\bigcup_\chi \Hilb_\chi\bigg)}\bigg)^\perp
= \Hilb_\text{pr.rat}^\perp.
$$
The proof of the theorem is complete.
\end{proof}

To prove \cref{decomposition2}, we will use the following lemma.

\begin{Lemma}\label{pos density average is zero}
Let $(w_n)_{n\in\N}$ be a sequence of non-negative real numbers and let $D\subset\N$ be a set
with positive natural density. If $\E_{n\in\N} w_n = 0$, then $\E_{n\in D} w_n =0$.
\end{Lemma}

The above result is classical and easy to show, thus its proof is omitted.

\begin{proof}[Proof of \cref{decomposition2}]
Let $F\in L^2(X)$ be a pretentious eigenfunction and let $G\in\Hilb_\text{pr.wm}$.
Then there exists some $f\in\Mult^\text{c}_\text{fg}$ such that $\D_F(S,f)<\infty$.
Let $\P=\{p\in\Pri\colon S_pF\neq f(p)F\}$ and then $\sum_{p\in\P}\frac{1}{p}<\infty$.
Then we have 
$$|\langle F,G\rangle| = \E_{n\in\Qf_\P}|\langle S_nF,S_nG\rangle|
= \E_{n\in\Qf_\P}|f(n)\langle F,S_nG\rangle|
= \E_{n\in\Qf_\P}|\langle F,S_nG\rangle|.$$
In view of \cref{pos density average is zero}, if we set $w_n = |\langle F,S_nG\rangle|$ 
and $D=\Qf_\P$ with natural density 
$$d(D)
= \prod_{p\not\in\P}\Big(1-\frac{1}{p}\Big)
\geq \exp\bigg(-\sum_{p\not\in\P}\frac{1}{p}\bigg)
>0,$$
then we have that $\E_{n\in\N} w_n = 0$, since $G$ is pretentiously weak-mixing, and then it follows that 
$$\E_{n\in\Qf_\P} |\langle F,S_nG\rangle|
= \E_{n\in D} w_n 
= 0.$$
Hence, we have that $\langle F,G\rangle = 0$
and this shows that $\Hilb_\text{pr.eig}\perp\Hilb_\text{pr.wm}$. 
It remains to prove that $\Hilb_\text{pr.eig}^\perp\subset\Hilb_\text{pr.wm}$.
Let $F\notin\Hilb_\text{pr.wm}$. Then, as in the proof of 
\cref{pretentious wm criterion}, we have 
$$\lim_{N\to\infty}\int_{\Mult^\text{c}_\text{fg}\times\Mult^\text{c}_\text{fg}}\frac{1}{N}\sum_{n=1}^N
f(n)\overline{g(n)} \d(\mu_F\times\mu_F)(f,g)
= \lim_{N\to\infty}\frac{1}{N}\sum_{n=1}^N
\bigg|\int_X S_nF\cdot\overline{F}\d\mu\bigg|^2 > 0.$$
Let $\Delta = \{(f,g)\in\Mult^\text{c}_\text{fg}\times\Mult^\text{c}_\text{fg}\colon \D(f,g)<\infty\}$
and then it follows by \cref{halasz mvt} that 
$(\mu_F\times\mu_F)(\Delta)>0$. By Tonelli's theorem, it follows that
there exists some $f\in\Mult^\text{c}_\text{fg}$ such that $\mu_F(\A_f\cap\Mult^\text{c}_\text{fg})>0$.
Now, as we argued in the proof of \cref{orthogonality prop 2}, we can find some
$G\in L^2(X)$ such that $\D_G(S,f)<\infty$ and $\langle F,G\rangle >0$.
This shows that $F\notin\Hilb_\text{pr.eig}^\perp$. The proof is complete.
\end{proof}

\subsection{Deduction of \cref{halasz mvt} from \ref{PMET}}

Let $f\in\Mult^{c}_\textup{fg}$ such that $\D(f,1)=\infty$.
We want to show that $M(f)=0$. Let $X=\overline{f(\N)}$, $\mu$ be the Haar measure
on $X$ and $S$ be the multiplicative rotation by $f$, that is to say,
$S_n(x) = f(n)x$ for any $n\in\N$. Let $F\in L^2(X)$ be the identity function. 
It suffices to show
$$\lim_{N\to\infty}\bigg\|\frac{1}{N}\sum_{n=1}^N S_nF\bigg\|_2 = 0.$$
In view of \ref{PMET}, we have to show that 
$\E(F\:|\:\sigma(\Ipr))=0$.
Since $\D(f,1)=\infty$ and $f$ is finitely generated,
$S_pF\neq F$ holds for many primes, and consequently, $f(p)\neq 1$ holds for many primes.
Let $A\subset X$ be a measurable set that pretends to be invariant. 
Then $S_p^{-1}A=A$ holds for almost every prime, and consequently, $f(p)A=A$ holds
for almost every prime. It follows that there exists some $p_0\in\Pri$ 
such that $f(p_0)\neq 1$ and $f(p_0)A=A$. 
We write $f(p_0)=e(\alpha)$ for some $\alpha\in(0,1)$ and we distinguish 
cases for $\alpha$:
\begin{itemize}
    \item[(i)] If $\alpha\in\Q$, then $A$ is invariant under 
    a rational rotation, and so it has the form of the following disjoint union:
    $$A = \bigsqcup_{j=0}^{q-1}e\Big(\frac{j}{q}\Big)B,$$
    for some measurable set $B\subset X$, where $q$ is the denominator of $\alpha$. 
    Then 
    $$\int_A F\d\mu  
    = \sum_{j=0}^{q-1}\int_{e(j/q)B}x\d\mu(x)
    = \int_B \sum_{j=0}^{q-1} e\Big(-\frac{j}{q}\Big)x\d\mu(x)
    = 0,$$
    since $\sum_{j=0}^{q-1}e(-j/q)=0$.
    \item[(ii)] If $\alpha\not\in\Q$, then $A$ is invariant under an irrational
    rotation, and since irrational rotations are ergodic with respect to the Haar measure,
    then the set $A$ is trivial in the sense that $\mu(A)\in\{0,1\}$.
    Then clearly, $\int_A F\d\mu = 0$.
\end{itemize}
It follows that for any $A\in\Ipr$, we have $\int_A F\d\mu = 0$.
Now let $B\in\sigma(\Ipr)$ and $\epsilon>0$.
By \cref{invariant algerba}, $\Ipr$ is an algebra, and so there exists $A\in\Ipr$ 
such that $\mu(A\triangle B)<\epsilon$. Then we have that 
$$\bigg|\int_B F\d\mu\bigg| = \bigg|\int_A F\d\mu + \int_{B\setminus A} F\d\mu\bigg|
= \bigg|\int_{B\setminus A} F\d\mu\bigg| \leq \mu(B\setminus A) <\epsilon.$$
Since $\epsilon$ is arbitrary, it follows that $\int_B F\d\mu=0$.
This holds for any $B\in\sigma(\Ipr)$, so we have that $\E(F\:|\:\sigma(\Ipr))=0$.
This shows that $M(f)=0$.
\par
Suppose now that $\D(f,1)<\infty$. Let, as before, $\xms$ be the multiplicative rotation 
by $f$ and $F\in L^2(X)$ be the identity. By assumption, we have that the set 
$\P=\{p\in\Pri\colon f(p)\neq 1\}$ satisfies $\sum_{p\in\P}\frac{1}{p}<\infty$.
For any $G\in L^2(X)$, we have that $\{p\in\Pri\colon S_pG\neq G\}\subset\P$,
hence $G\in\Hilb_1$. This shows that $L^2(X) = \Hilb_1$ and in particular, we have that
$\E(F\:|\:\sigma(\Ipr))=F$. By \ref{PMET}, since $\|F\|_2=1$, it follows that
\begin{align*}
    & \lim_{N\to\infty}\bigg|\frac{1}{N}\sum_{n=1}^N f(n) - 
    \prod_{p\in\Pri}\Big(1-\frac{1}{p}\Big)\Big(1-\frac{f(p)}{p}\Big)^{-1}\bigg|^2 
    = \lim_{N\to\infty}\bigg|\frac{1}{N}\sum_{n=1}^N f(n) - 
    \prod_{p\in\Pri}\Big(1-\frac{1}{p}\Big)\Big(\sum_{k\geq0}\frac{f(p^k)}{p^k}\Big)\bigg|^2 \\
    & = \lim_{N\to\infty}\int_X\bigg|\frac{1}{N}\sum_{n=1}^N f(n)x - 
    \prod_{p\in\Pri}\Big(1-\frac{1}{p}\Big)\Big(\sum_{k\geq0}\frac{f(p^k)x}{p^k}\Big)\bigg|^2\d\mu(x) \\
    & = \lim_{N\to\infty}\bigg\|\frac{1}{N}\sum_{n=1}^N S_nF
    - \prod_{p\in\Pri}\Big(1-\frac{1}{p}\Big)\Big(\sum_{k\geq0}\frac{S_{p^k}F}{p^k}\Big)\bigg\|_2^2
    = 0.
\end{align*} 
The proof is complete for completely multiplicative function.

Extending \ref{PMET} to all weakly multiplicative systems (with exactly the same 
formulation) would similarly yield the more general case of $f$ being multiplicative,
but not necessarily completely, of \cref{halasz mvt}.


\section{Proof of \ref{mt1}}
\label{JE proofs}

Recall that in \ref{mt1}, we have to treat ergodic averages of the form
$$\frac{1}{N}\sum_{n=1}^N T^nF \cdot S_nG.$$
The main idea is to decompose both $F$ and $G$ into two distinct 
components that exhibit opposite behaviours, namely one 
(pretentiously) periodic and one totally ergodic (or aperiodic), 
utilizing \eqref{add decomp} and \cref{decomposition} respectively.
In this way we are reduced into dealing with ergodic averages of 
$T^nF\cdot S_nG$ with the advantage of having more
information regarding the behaviour of $T^nF$ and $S_nG$.
Hence, the proof of \ref{mt1} heavily depends on our previous results
on multiplicative systems.

For the 
rest of this section, we fix a probability space $\xm$, an ergodic additive 
action $T$ and an pretentiously ergodic completely multiplicative action $S$ on $\xm$.
To prove \ref{mt1}, we need the two lemmas below.

First let us remark that for any Dirichlet character $\chi$ and any rational $\frac{r}{q}$
it is not hard to check that the limit 
$$\lim_{N\to\infty}\frac{1}{N}\sum_{n=1}^N e\Big(\frac{rn}{q}\Big) \chi(n)$$
exists. Then, combining \cref{distance prop for fg functions} with
\cite[Corollary 2]{DD}, it follows that 
for any $f\in\Mult_\text{fg}$ with $\D(f,\chi)<\infty$ 
the limit 
$$\lim_{N\to\infty}\frac{1}{N}\sum_{n=1}^N e\Big(\frac{rn}{q}\Big) f(n)$$
also exists.

\begin{Lemma}\label{prelim lemma 1}
Let $q,q_0\in\N$ and $\chi$ be a primitive Dirichlet character mod $q_0$ such that
$q_0\nmid q$. Then we have 
$$\lim_{N\to\infty}\frac{1}{N}\sum_{n=1}^N e\Big(\frac{an}{q}\Big)\chi(n)
= 0 
\qquad \forall~a\in\N.$$
\end{Lemma}

\begin{proof}
Let $a\in\N$. Using \eqref{prelim eq1},
for any $N\in\N$ and for any $K=K(N)$ such that $K\to\infty$
and $K/N\to0$ as $N\to\infty$, we have
\begin{align*}
    \frac{1}{N}\sum_{n=1}^N e\Big(\frac{an}{q}\Big)\chi(n)
    & = \frac{1}{N}\sum_{n=1}^N\frac{1}{K}\sum_{k=1}^K
    e\Big(\frac{a}{q}(n+qk)\Big)\chi(n+qk) + \Oh\Big(\frac{K}{N}\Big) \\
    & = \frac{1}{N}\sum_{n=1}^N e\Big(\frac{an}{q}\Big)
    \bigg(\frac{1}{K}\sum_{k=1}^K \chi(n+qk)\bigg) + \oh_{N\to\infty}(1) \\
    & = \frac{1}{\tau(\overline{\chi})}\sum_{\substack{m=1 \\ (m,q_0)=1}}^{q_0} \overline{\chi}(m)
    \bigg(\frac{1}{N}\sum_{n=1}^N e\Big(\frac{mn}{q_0}+\frac{an}{q}\Big)\bigg)
    \bigg(\frac{1}{K}\sum_{k=1}^K e\Big(\frac{mqk}{q_0}\Big)\bigg)
    + \oh_{N\to\infty}(1) \\
    & \ll \max_{\substack{1\leq m\leq q_0 \\ (m,q_0)=1}}
    \bigg|\frac{1}{K}\sum_{k=1}^K e\Big(\frac{mqk}{q_0}\Big)\bigg| 
    + \oh_{N\to\infty}(1) 
    = \oh_{K\to\infty}(1) + \oh_{N\to\infty}(1)
    = \oh_{N\to\infty}(1),
\end{align*}
(here $\tau(\overline{\chi})$ denotes the Gauss sum of $\overline{\chi}$; see \cref{NT section} for the definition),
since for any $1\leq m \leq q_0$ with $(m,q_0)=1$, we have that $q_0\nmid mq$, 
because $q_0\nmid q$. The proof is complete. 
\end{proof}

\begin{Lemma}\label{prelim lemma 2}
Let $r,q,q_0\in\N$ with $(r,q)=1$ and $\chi$ be a primitive Dirichlet character mod $q_0$.
Then the following are equivalent:
\begin{itemize}
    \item[\textup{(i)}] $\lim_{N\to\infty}\frac{1}{N}\sum_{n=1}^N 
                        e(\frac{rn}{q})\chi(n) = 0$.
    \item[\textup{(ii)}] $q_0\neq q$.
    \item[\textup{(iii)}] There exists some $f\in\Mult_\textup{fg}$ 
    with $\D(f,\chi)<\infty$ such that 
                        $\lim_{N\to\infty}\frac{1}{N}\sum_{n=1}^N 
                        e(\frac{rn}{q})f(n) = 0.$
\end{itemize}
\end{Lemma}

\begin{proof}
(i)~$\Longrightarrow$~(ii)
Assume that $q_0=q$. Then, using \eqref{prelim eq1}, we have
$$\lim_{N\to\infty}\frac{1}{N}\sum_{n=1}^N e\Big(\frac{rn}{q}\Big)\chi(n)
= \frac{1}{q}\sum_{a=1}^q e\Big(\frac{ra}{q}\Big)\chi(a)
= \frac{\tau(\chi)\overline{\chi}(q-r)}{q}
\neq 0,$$
since $\tau(\chi)\neq 0$ by the fact that $\chi$ is primitive and
$\overline{\chi}(q-r)\neq 0$ by the fact that $(r,q)=1$. \\
(ii)~$\Longrightarrow$~(i) 
Assume that $q_0\neq q$. 
Then, using \eqref{prelim eq3} and then \eqref{prelim eq1}, we have 
\begin{align}\label{2nd prelim lemma eq1}
    \frac{1}{N}\sum_{n=1}^N e\Big(\frac{rn}{q}\Big)\chi(n)
    & = \sum_{a=1}^{q_0}\chi(a)\bigg(\frac{1}{N}\sum_{\substack{n\leq N \\ n\equiv a\Mod {q_0}}}e\Big(\frac{rn}{q}\Big)\bigg) \notag \\
    & = \frac{1}{q_0}\sum_{a=1}^{q_0}\chi(a)\sum_{b=1}^{q_0} e\Big(-\frac{ba}{q_0}\Big)
    \bigg(\frac{1}{N}\sum_{n=1}^N e\Big(\Big(\frac{r}{q}+\frac{b}{q_0}\Big)n\Big)\bigg) \notag \\
    & = \frac{1}{q_0}\sum_{a=1}^{q_0}\chi(a) 
    \sum_{\substack{b=1 \\ qq_0 \mid (rq_0+qb)}}^{q_0} e\Big(-\frac{ba}{q_0}\Big)
    + \oh_{N\to\infty}(1) 
    = \frac{\tau(\chi)}{q_0}\sum_{\substack{b=1 \\ qq_0 \mid (rq_0+qb)}}^{q_0}
    \overline{\chi}(b) + \oh_{N\to\infty}(1).
\end{align}
Suppose that $qq_0\mid (rq_0+qb)$. Then $rq_0 + qb = kqq_0$ for some $k\in\N$.
Since $(r,q)=1$, then there exists unique $x,y\in\Z$ such that $1=xr+yq$, thus
$q_0 = xrq_0 + yqq_0 = xkqq_0 - xqb + yqq_0$. It follows that $q\mid q_0$ and since 
$q\neq q_0$, we have that $q_0 = k_1q$ for some integer $k_1>1$.
Therefore,
$rk_1q + qb = kk_1q^2$, hence 
$rk_1 + b = kk_1q$. It follows that $k_1\mid b$.
Therefore, $k_1\mid (b,q_0)$, implying that $(b,q_0)>1$, thus $\overline{\chi}(b)=0$.
We proved that for any $1\leq b\leq q_0$ such that $qq_0\mid (rq_0+qb)$, we have 
$\overline{\chi}(b)=0$. Then, the averages on the left-hand side of 
\eqref{2nd prelim lemma eq1} converge to zero as $N\to\infty$.
\\
(i)~$\Longrightarrow$~(iii)
This is obvious. \\
(iii)~$\Longrightarrow$~(i)
Assume that
$$\lim_{N\to\infty}\frac{1}{N}\sum_{n=1}^N e\Big(\frac{rn}{q}\Big)\chi(n) \neq 0.$$
Since (ii)~$\Longrightarrow$~(i), it follows that $q_0=q$, i.e., $\chi$ is primitive
mod $q$. By \cref{distance prop for fg functions}, for any $f\in\Mult_\text{fg}$ 
with $\D(f,\chi)<\infty$, we have 
$\sum_{p\in\Pri}\frac{1}{p}(1-f(p)\overline{\chi}(p))<\infty$. Using this along
with that $q$ is the conductor of $\chi$, it follows from the formula given in 
\cite[Remark 2.1.1]{DD} that for any such $f$, we have
$$\lim_{N\to\infty}\frac{1}{N}\sum_{n=1}^N e\Big(\frac{rn}{q}\Big)f(n) \neq 0.$$
The proof of the lemma is complete.
\end{proof}

Finally, we need the following lemma.

\begin{Lemma}\label{rat spectrums lemma}
Let $\xm$ be a probability space, $T$ be an additive action and $S$ be a finitely generated multiplicative action. Then 
$\sigma_\textup{rat}(T)\cap\widetilde{\sigma}_\textup{pr.rat}(S) = \{0\}$ if and only if for any $\frac{r}{q}\in\sigma_\textup{rat}(T)\setminus\{0\}$ and for any $q_0\mid q$, we have $\frac{r}{q_0}\not\in\widetilde{\sigma}_\textup{pr.rat}(S)$.
\end{Lemma}

\begin{proof}
We just show that if $\sigma_\textup{rat}(T)\cap\widetilde{\sigma}_\textup{pr.rat}(S) = \{0\}$, then for any $\frac{r}{q}\in\sigma_\textup{rat}(T)\setminus\{0\}$ and for any $q_0\mid q$, we have $\frac{r}{q_0}\not\in\widetilde{\sigma}_\textup{pr.rat}(S)$. The other implication is obvious.

Suppose that there exist $\frac{r}{q}\in\sigma_\text{rat}(T)\setminus\{0\}$ and $q_0\mid q$ such that $\frac{r}{q_0}\in\widetilde{\sigma}_\text{pr.rat}(S)$. Letting $s=\frac{q}{q_0}\in\N$, since $\frac{r}{q}\in\sigma_\text{rat}(T)$, we have that
$\frac{r}{q_0}=\frac{rs}{q}\in\sigma_\text{rat}(T)$, hence the non-zero rational number $\frac{r}{q_0}$ belongs in $\sigma_\textup{rat}(T)\cap\widetilde{\sigma}_\textup{pr.rat}(S)$.
This concludes the proof.
\end{proof}

\begin{proof}[Proof of \ref{mt1}]
Assume that $T,S$ are jointly ergodic.
Let $\frac{r}{q}\in\sigma_\text{rat}(T)\cap\widetilde{\sigma}_{\text{pr.rat}}(S)$
and suppose for sake of contradiction that $r\neq0$. Recall that by definition of $\widetilde{\sigma}_\text{pr.rat}(S)$, we have $(r,q)=1$.
Then there exists $F\in L^2(X)$ such that $TF=e(\frac{r}{q})F$. Since $r\neq0$,
we have that $\int_X F\d\mu = 0$. Moreover, there exist $G\in L^2(X)$ and 
a primitive Dirichlet character $\chi$ mod $q$ such that $\D_G(S,\chi)<\infty$.
Since $T$ is ergodic and $F$ is an eigenfunction, we have that $|F|=1$.
Then, using \cref{spectral thm} and \cref{{relation between the two distances}},
since $T,S$ are jointly ergodic, we have
\begin{align*}
    0 & = \lim_{N\to\infty}\bigg\|\frac{1}{N}\sum_{n=1}^N T^nF\cdot S_nG \bigg\|_2^2
    = \lim_{N\to\infty}\int_X |F|^2\bigg|\frac{1}{N}\sum_{n=1}^N
    e\Big(\frac{rn}{q}\Big)S_nG\bigg|^2 \d\mu \\
    & = \lim_{N\to\infty}\int_X \bigg|\frac{1}{N}\sum_{n=1}^N
    e\Big(\frac{rn}{q}\Big)S_nG\bigg|^2 \d\mu
    = \lim_{N\to\infty}\int_{\Mult^\text{c}_\text{fg}\cap\A_\chi} 
    \bigg|\frac{1}{N}\sum_{n=1}^N e\Big(\frac{rn}{q}\Big)f(n)\bigg|^2 \d\mu_G(f) \\
    & \geq \int_{\Mult^\text{c}_\text{fg}\cap\A_\chi} \liminf_{N\to\infty}
    \bigg|\frac{1}{N}\sum_{n=1}^N e\Big(\frac{rn}{q}\Big)f(n)\bigg|^2 \d\mu_G(f)
    = \int_{\Mult^\text{c}_\text{fg}\cap\A_\chi} \lim_{N\to\infty}
    \bigg|\frac{1}{N}\sum_{n=1}^N e\Big(\frac{rn}{q}\Big)f(n)\bigg|^2 \d\mu_G(f),
\end{align*}
where $\mu_G$ is the spectral measure of $G$ and the inequality holds by Fatou's lemma.
It follows that there exists some $f\in\Mult^\text{c}_\text{fg}\cap\A_\chi$ such that
$$\lim_{N\to\infty}\frac{1}{N}\sum_{n=1}^N e\Big(\frac{rn}{q}\Big)f(n) = 0.$$
By \cref{prelim lemma 2}, the last equation implies that the conductor of $\chi$
is not $q$, yielding a contradiction. Hence,
$\sigma_\text{rat}(T)\cap\widetilde{\sigma}_{\text{pr.rat}}(S)=\{0\}$.
\par
On the other hand, we assume that
$\sigma_\text{rat}(T)\cap\widetilde{\sigma}_{\text{pr.rat}}(S)=\{0\}$.
Let $F,G\in L^2(X)$
and we may assume that $\int_X G\d\mu = 0$. Hence we shall prove that
\begin{equation}\label{mt1 eq0}
\lim_{N\to\infty}\bigg\|\frac{1}{N}\sum_{n=1}^N T^nF\cdot S_nG\bigg\|_2 = 0.
\end{equation}
In view of \eqref{add decomp}, we decompose $F$ as $F=F_\text{rat} + F_\text{tot.erg}$, for some unique
$F_\text{rat}\in\Hilb_\text{rat}(T)$ and $F_\text{tot.erg}\in\Hilb_\text{tot.erg}(T)$.
By \cref{orthogonality criterion}, we see that the
contribution of $F_\text{tot.erg}$ in the ergodic averages in \eqref{mt1 eq0} is zero.
Therefore, we may assume without loss of generality that
$F\in\Hilb_\text{rat}(T)$
and a simple approximation argument allows to
further assume that $F$ is a rational eigenfunction for $T$. Then there exist $r,q\in\N$
such that $TF = e(\frac{r}{q})F$, and so \eqref{mt1 eq0} is reduced to 
\begin{equation}\label{mt1 eq1}
    \lim_{N\to\infty}\bigg\|\frac{1}{N}\sum_{n=1}^N e\Big(\frac{rn}{q}\Big)S_nG\bigg\|_2 = 0.
\end{equation}
We may assume that 
$r\neq0$ (since otherwise \eqref{mt1 eq1} follows from pretentious ergodicity of $S$) and also that $(r,q)=1$.
Now, by \cref{decomposition}, we can decompose $G$ as 
$G=G_\text{pr.rat} + G_\text{aper}$, for some unique 
$G_\text{pr.rat}\in\Hilb_\text{pr.rat}(S)$ and $G_\text{aper}\in\Hilb_\text{aper}(S)$. 
By definition, the contribution of $G_\text{aper}$ in the ergodic averages in 
\eqref{mt1 eq1} is zero. Therefore, we may assume that $G\in\Hilb_\text{pr.rat}(S)$. 
A simple approximation argument allows to assume that $G$ is a finite linear combination 
of pretentious rational eigenfunctions and then, by the triangle inequality for the 
$L^2$ norm, we may eventually assume that $G$ is a pretentious rational eigenfunction 
for $S$. So there exists some Dirichlet character $\chi$
such that $\D_G(S,\chi)<\infty$ and we may assume that $\chi$ is primitive
(or $1$ in which case the result follows trivially). Let $q_0$ be the conductor of $\chi$.
In view of \cref{rat spectrums lemma}, the assumption $\sigma_\text{rat}(T)\cap\widetilde{\sigma}_{\text{pr.rat}}(S)=\{0\}$ implies that $q_0\nmid q$, which implies that 
$$\lim_{N\to\infty}\frac{1}{N}\sum_{n=1}^N e\Big(\frac{an}{q}\Big)\chi(n) = 0,
\qquad \forall~a\in\N,$$
by \cref{prelim lemma 1}, and this is equivalent to that
$$\lim_{N\to\infty}\frac{1}{N}\sum_{n=1}^N \overline{\psi}(n)\chi(n) = 0,
\qquad \forall~\text{Dirichlet character}~\psi~\text{mod}~q,$$
by \cref{aperiodic functions equiv}. 
In view of \cref{halasz mvt}, the latter is equivalent to that
$\D(\chi,\psi)=\infty$ for any Dirichlet character $\psi$ modulo $q$.
Combining \cref{same eigenspaces} and \cref{small decomposition}, it follows that
$\Hilb_\chi\subset\Hilb_\psi^\perp=\V_\psi$, thus $G\in\V_\psi$ 
for any Dirichlet character $\psi$ modulo $q$, and then it follows from
\cref{aperiodic functions in systems equiv} that 
$$\lim_{N\to\infty}\bigg\|\frac{1}{N}\sum_{n=1}^N
e\Big(\frac{an}{q}\Big)S_nG\bigg\|_2 = 0
\qquad \forall~a\in\N.$$
In particular, this gives \eqref{mt1 eq1} and so concludes the proof of the theorem.
\end{proof}


\appendix

\section{Spectral theorems}\label{spectral thms}

\begin{Theorem}[Spectral theorem on unitary operators, see {\cite[Theorem B.12]{einsiedler}}]
\label{classic spectral thm}
Let $\xmt$ be an additive system. For each $F\in L^2(X)$, there exists 
a unique finite Borel measure $\mu_F$ in $\Sone$ such that 
$$\langle T^nF,F\rangle = \int_0^1 e(nt)\d\mu_F(t).$$
Moreover, there exists a unitary isomorphism $\Phi$ from $L^2(\Sone,\mu_F)$
to the cyclic sub-representation of $L^2(X)$ which is generated by $F$ under $T$, 
that conjugates $T$ with multiplication by $e(t)$.
\end{Theorem}

We are also interested in a version of the spectral theorem 
on unitary multiplicative actions.
Recall that $\widehat{(\N,\times)}=\Mult^\text{c}$ and that $e_n(f) = f(n), n\in\N$.

\begin{Theorem}[Spectral theorem on unitary multiplicative actions, cf.{ \cite[(1.47)]{folland}}]
\label{spectral thm}
Let $\xms$ be a multiplicative system. For each $F\in L^2(X)$, there exists 
a unique finite Borel measure $\mu_F$ in $\Mult^\textup{c}$ such that 
$$\langle S_nF,F\rangle = \int_{\Mult^\textup{c}} f(n)\d\mu_F(f).$$
Moreover, there exists a unitary isomorphism $\Phi$ from $L^2(\Mult^\text{c},\mu_F)$ 
to the cyclic sub-representation of $L^2(X)$ which is generated by $F$ under the action $S$,
that conjugates $S_n$ with multiplication by $e_n$, for any $n\in\N$.
\end{Theorem}

If $\xms$ is finitely generated in \cref{spectral thm}, then we can replace $\Mult^\text{c}$ with $\Mult^\text{c}_\text{fg}$. This
follows from \cref{spectral measure supported on fg functions}.
In addition, if the system is weakly multiplicative, 
we have $\Mult$ in place of $\Mult^\text{c}$.

\section{Elementary facts from number theory}
\label{NT section}

We start with a classical result for aperiodic functions, which the analogue of 
\cref{aperiodic functions in systems equiv} for multiplicative functions.
\begin{Lemma}\label{aperiodic functions equiv}
Let $f$ be a bounded multiplicative function. The following are equivalent for any $q\in\N$:
\begin{itemize}
    \item[\textup{(i)}] $\lim_{N\to\infty}\frac{1}{N}\sum_{n=1}^N f(qn+r) = 0$ for any $r\in\N$.
    \item[\textup{(ii)}] $\lim_{N\to\infty}\frac{1}{N}\sum_{n=1}^N e(\frac{rn}{q})f(n) = 0$ 
    for any $r\in\N$.
    \item[\textup{(iii)}]  $\lim_{N\to\infty}\frac{1}{N}\sum_{n=1}^N \chi(n)f(n) = 0$
    for any Dirichlet character $\chi$ of modulus $q$.
\end{itemize}
\end{Lemma}

We omit the proof, since it is very similar with that of 
\cref{aperiodic functions in systems equiv}, which we show later in the appendix.
\par
Now we state some useful classical identities that relate Dirichlet characters, 
linear phases and arithmetic progressions. 
Given a Dirichlet character $\chi$ of some modulus $q$, $\tau(\chi) := \sum_{m=1 }^q e(\frac{m}{q})\chi(m)$ is the Gauss sum of $\chi$,
and if $\chi$ is primitive, then $|\tau(\chi)|=q^\frac{1}{2}$,
in particular, $\tau(\chi)\neq0$.
Moreover, we denote by $\phi$ the classical Euler's 
totient function. We state the following classical identities:

\begin{equation}\label{prelim eq1}
    \chi(n) = \frac{1}{\tau(\overline{\chi})}\sum_{m=1}^q
    \overline{\chi}(m)e\Big(\frac{mn}{q}\Big),
    \qquad \forall~\text{primitive Dirichlet character}~\chi~\text{mod}~q.
\end{equation}
\begin{equation}\label{prelim eq2}
    \1_{n\equiv r\Mod q} = \frac{1}{\phi(q)}\sum_{\chi~\text{mod}~q}\overline{\chi}(r)\chi(n),
    \qquad \forall~r,q\in\N,~(r,q)=1.
\end{equation}
\begin{equation}\label{prelim eq3}
    \1_{n\equiv r\Mod q} = \frac{1}{q}\sum_{a=1}^q e\Big(\frac{a}{q}(n-r)\Big),
    \qquad \forall~r,q\in\N.
\end{equation}

Moreover, concerning the distance of multiplicative functions, we have the following triangle inequalities (see \cite{pretentious}):
For any multiplicative functions $f,g,f',g',h:\N\to\U$ and any $N\in\N$, we have that 
$$\D(f,g;N)\leq \D(f,h;N)+\D(h,g;N)\qquad \text{and} \qquad
\D(ff',gg';N)\leq \D(f,g;N) + \D(f',g';N).$$

We conclude this appendix with some classical results on multiplicative functions.

\begin{Lemma}\label{fg functions description}
Let $f\in\Mult_\textup{fg}$. Then $\D(f,1)<\infty$ if and only if $f(p)=1$ for almost every prime.
\end{Lemma}

The above lemma follows from \cref{almost eigenfunction lemma}
but it can be proved independently as well. It is used implicitly several times 
in the paper.

\begin{Lemma}\label{distance of fg functions}
For any $f\in\Mult_\textup{fg}$, any Dirichlet character $\chi$
and any $t\in\R\setminus\{0\}$, $\D(f,\chi\cdot n^{it})=\infty$.
\end{Lemma}

\begin{proof}
Let $t\in\R\setminus\{0\}$ and $\chi$ be a Dirichlet character. 
Let $\alpha_1,\dots,\alpha_d\in[0,1)$ be distinct numbers such that
for any $p\in\Pri$, $f(p)\in\{e(\alpha_1),\dots,e(\alpha_d)\}$. 
It suffices to show that $\D(f,\chi^\ast\cdot n^{it})=\infty$ for the 
modified character $\chi^\ast$. 
Now there exists some $k\in\N$ such that $\chi^\ast(p)^k=1$ for all primes $p$.
We consider the multiplicative function $g:\N\to\Sone$
given by $g(p)=(f(p)\overline{\chi^\ast}(p))^k$, and then there exist distinct
$\beta_1,\dots,\beta_d\in[0,1)$ such that $g(p)\in\{\beta_1,\dots,\beta_d\}$
for any $p\in\Pri$. For each $j=1,\dots,d$, we set 
$\P_j=\{p\in\Pri\colon g(p)=e(\beta_j)\}$.
Now by the triangle inequality for the distance of multiplicative functions, we have that
$k\D(f,\chi^\ast\cdot n^{it})\geq \D(g,n^{ikt})$, so that it is enough to 
show that $\D(g,n^{ikt})=\infty$. We see that
$$ \D(g,n^{ikt})^2 
= \sum_{j=1}^{d}\sum_{p\in\P_j}\frac{1-\Re(e(\beta_j)p^{-ikt})}{p}.$$
Hence, the result will follow immediately by showing the following claim: 
For any $z\in\Sone$, and for any $t\in\R\setminus\{0\}$, we have 
$\D(z,n^{it})=\infty$.
\par
We now prove the claim. Let $z\in\Sone$ and $t\in\R\setminus\{0\}$.
For any $N\in\N$ and any $\epsilon>0$, we have that
\begin{align*}
    \D(z,n^{it};N)^2 
     = \sum_{p\leq N}\frac{1-\Re(z p^{-it})}{p} 
    & \geq \sum_{\exp((\log{N})^{2/3+\epsilon})\leq p\leq N} 
    \frac{1-\Re(zp^{-it})}{p} \\
    & \geq \sum_{\exp((\log{N})^{2/3+\epsilon})\leq p\leq N}\frac{1}{p}-
    \bigg|\sum_{\exp((\log{N})^{2/3+\epsilon})
    \leq p\leq N}\frac{zp^{-it}}{p}\bigg| \\ 
    &= \Big(\frac{1}{3}-\epsilon\Big)\log\log{N}-\bigg|\sum_{\exp((\log{N})^{2/3+\epsilon})
    \leq p\leq N}\frac{1}{p^{1+it}}\bigg|+O(1),
\end{align*}
and so, it suffices to show that the second term in the last step of the above equation is $\Oh(1)$.
But this term is equal to 
$|\log(\zeta(1+\frac{1}{\log{N}}+it))-\log(\zeta(1+\frac{1}{(\log{N})^{2/3+\epsilon}}+it)|+O(1)$,
and this is $O(1)$, by the Vinogradov-Korobov zero-free region for the Riemann zeta
function (see \cite{richert}).
\end{proof}

\begin{Lemma}\label{distance prop for fg functions}
For any $f\in\Mult_\textup{fg}$ and any Dirichlet character $\chi$, we have that
$\D(f,\chi)<\infty$ if and only if $\sum_{p\in\Pri}\frac{1}{p}(1-f(p)\overline{\chi}(p))<\infty$. 
\end{Lemma}

\begin{proof}
By the previous trick with the modified character, it is enough to show that
$\sum_{p\in\Pri}\frac{1}{p}(1-f(p))<\infty$ provided that $\D(f,1)<\infty$.
By \cref{fg functions description}, the assumption implies that $f(p)=1$
fails only in a set $\P\subset\Pri$ containing few primes. Therefore, we have 
$\sum_{p\in\Pri}\frac{1}{p}(1-f(p)) = \sum_{p\in\P}\frac{1}{p}<\infty$.
\end{proof}

\section{Proof of \cref{aperiodic functions in systems equiv}}\label{proof of aperiodic equiv}

\begin{proof}[Proof of \cref{aperiodic functions in systems equiv}]
Let $\xms$ be a finitely generated multiplicative system and $q\in\N$. \\
(i)~$\Longrightarrow$~(ii) Let $r\in\N$. 
The assertion in case that $q\mid r$ is obvious, so we assume that $q\nmid r$ and 
we let $F\in L^2(X)$. We may assume that $\int_X F\d\mu=0$. Then we have 
\begin{align*}
    \limsup_{N\to\infty}\bigg\|\frac{1}{N}\sum_{n=1}^N e\Big(\frac{rn}{q}\Big)S_nF\bigg\|_2
    & = \limsup_{N\to\infty}\bigg\|\sum_{a=1}^q \frac{1}{N}\sum_{\substack{n=1 \\ n\equiv a\Mod q}}^N 
    e\Big(\frac{rn}{q}\Big)S_nF \bigg\|_2 \\
    & = \limsup_{N\to\infty}\bigg\|\sum_{a=1}^q e\Big(\frac{ra}{q}\Big)
    \bigg(\frac{1}{N}\sum_{n=1}^N \1_{n\equiv a\Mod q} S_nF\bigg)\bigg\|_2 \\
    & \leq \max_{1\leq a\leq q} \limsup_{N\to\infty}\bigg\|\frac{1}{N/q}\sum_{0\leq n< N/q}S_{qn+a}F\bigg\|_2
    = 0.
\end{align*}
This proves (i). \\
(ii)~$\Longrightarrow$~(iii)
Let $F\in L^2(X)$ and suppose that $\int_X F\d\mu=0$.
By \eqref{prelim eq1}, we have that for any primitive Dirichlet character $\chi$
of modulus $q$, 
\begin{align*}
    \limsup_{N\to\infty}\bigg\|\frac{1}{N}\sum_{n=1}^N \chi(n)S_nF \bigg\|_2 
    & = \limsup_{N\to\infty}
    \bigg\|\frac{1}{\tau(\overline{\chi})}\sum_{a=1}^q \overline{\chi}(a) \bigg(\frac{1}{N}\sum_{n=1}^N e\Big(\frac{an}{q}\Big)S_nF\bigg)\bigg\|_2 \\
    & \ll \max_{1\leq a\leq q}\limsup_{N\to\infty}
    \bigg\|\frac{1}{N}\sum_{n=1}^N e\Big(\frac{an}{q}\Big)S_nF\bigg\|_2
    = 0.
\end{align*}
Now for any Dirichlet character $\chi$, we consider the primitive 
character $\chi_1$ inducing $\chi$, which also has modulus $q$.
As we have already seen, we have $\D(\chi,\chi_1)<\infty$,
and then using \cref{change of function in averages}, it follows that
$$\lim_{N\to\infty}\bigg\|\frac{1}{N}\sum_{n=1}^N \chi(n)S_nF\bigg\|_2
= 0,$$
showing the (ii) under the assumption $\int_X F \d\mu = 0$.
\par
Now let $F\in L^2(X)$ be arbitrary and consider the zero-integral function 
$G = F - \int_X F \d\mu$.
Then we have that for any Dirichlet character of modulus $q$,
\begin{align*}
    \bigg\|\frac{1}{N}\sum_{n=1}^N \chi(n)S_nF\bigg\|_2
    & \leq \bigg\|\frac{1}{N}\sum_{n=1}^N \chi(n)S_nG\bigg\|_2
    + \bigg(\frac{1}{N}\sum_{n=1}^N\chi(n)\bigg)\cdot\int_X F\d\mu
\end{align*}
and since 
$$\lim_{N\to\infty}\frac{1}{N}\sum_{n=1}^N\chi(n) =
\begin{cases}
    \frac{1}{\phi(q)}, & \text{if}~\chi=\chi_0, \\
    0, & \text{otherwise},
\end{cases}$$
where $\chi_0$ is the principal character of modulus $q$, the result follows. \\
(iii)~$\Longrightarrow$~(i) Let $r\in\N$ with $d=(r,q)$ and write $r=r_1d$ and 
$q=q_1d$ with $(r_1,q_1)=1$.
Let $F\in L^2(X)$ and we may assume that $\int_X F\d\mu=0$.
Then $\int_X S_dF\d\mu=\int_X F\d\mu=0$.
By \eqref{prelim eq2}, we have that
\begin{align*}
    \limsup_{N\to\infty}\bigg\|\frac{1}{N}\sum_{n=1}^N S_{qn+r}F\bigg\|_2
    & = \limsup_{N\to\infty}\bigg\|\frac{1}{N}\sum_{n=q+r}^{qN+r}
    \1_{n\equiv r\Mod q}S_nF\bigg\|_2 \\
    & = \limsup_{N\to\infty}\bigg\|\frac{1}{N}\sum_{n=q_1+r_1}^{q_1N+r_1}
    \1_{n\equiv r_1\Mod{q_1}}S_{dn}F\bigg\|_2 \\
    & = \limsup_{N\to\infty}\bigg\|\frac{1}{\phi(q_1)}\sum_{\chi~\text{mod}~q_1}
    \bigg(\frac{1}{N}\sum_{n=q_1+r_1}^{q_1N+r_1}\chi(n)S_n(S_dF)\bigg)\bigg\|_2 \\
    & \ll \max_{\chi~\text{mod}~q_1}\limsup_{N\to\infty}
    \bigg\|\frac{1}{N}\sum_{n=q_1+r_1}^{q_1N+r_1}\chi(n)S_n(S_dF)\bigg\|_2
    = 0,
\end{align*}
since any Dirichlet character of modulus $q_1$ has also modulus $q$.
This concludes the proof.
\end{proof}

\begin{Remark}\label{final remark}
The proof of \cref{aperiodic functions in systems equiv} can be extended to all
finitely generated weakly multiplicative actions. The only part that is not straightforward 
in this case is the proof of (iii) $\Longrightarrow$ (i), but one could adapt an
argument of Delange in \cite[pp. 136-138]{Delange} to handle it.
\end{Remark}

\bibliographystyle{aomalpha}
\bibliography{Refs_Multi_Systems}
\addcontentsline{toc}{section}{References}

\bigskip
\bigskip
\footnotesize
\noindent
Dimitrios Charamaras\\
\textsc{École Polytechnique Fédérale de Lausanne (EPFL)} \par\nopagebreak
\noindent
\href{mailto:dimitrios.charamaras@epfl.ch}
{\texttt{dimitrios.charamaras@epfl.ch}}

\end{document}